\documentclass[11pt]{article}

\usepackage[T1]{fontenc}
\usepackage[utf8]{inputenc}
\usepackage{amsmath,amssymb,amsthm,mathtools}
\usepackage{aliascnt}
\usepackage{graphicx}
\usepackage{geometry}
\usepackage{microtype}
\usepackage{xcolor}
\usepackage{hyperref}
\usepackage{authblk}
\usepackage{comment}
\usepackage{mhequ}

\usepackage{thmtools}
\usepackage[nameinlink,noabbrev]{cleveref}
\usepackage{bbm}
\usepackage{caption}
\hypersetup{
  colorlinks=true,
  linkcolor=blue,
  citecolor=blue,
  urlcolor=blue
}
\newtheorem{theorem}{Theorem}[section]
\newtheorem{proposition}[theorem]{Proposition}
\newtheorem{lemma}[theorem]{Lemma}
\newtheorem{corollary}[theorem]{Corollary}
\theoremstyle{definition}
\newtheorem{assumption}{Assumption}

\newtheorem{remark}[theorem]{Remark}

\Crefname{theorem}{theorem}{theorems}
\Crefname{theorem}{Theorem}{Theorems}
\Crefname{proposition}{proposition}{propositions}
\Crefname{proposition}{Proposition}{Propositions}
\Crefname{lemma}{lemma}{lemmas}
\Crefname{lemma}{Lemma}{Lemmas}
\Crefname{corollary}{corollary}{corollaries}
\Crefname{corollary}{Corollary}{Corollaries}
\Crefname{assumption}{assumption}{assumptions}
\Crefname{assumption}{Assumption}{Assumptions}
\Crefname{definition}{definition}{definitions}
\Crefname{definition}{Definition}{Definitions}
\Crefname{remark}{remark}{remarks}
\Crefname{remark}{Remark}{Remarks}

\newcommand{\Sphere}{\mathbb{S}}
\newcommand{\Torus}{\mathbb{T}}
\newcommand{\R}{\mathbb{R}}
\newcommand{\divg}{\operatorname{div}}
\newcommand{\Hess}{\operatorname{Hess}}
\newcommand{\Ric}{\operatorname{Ric}}

\definecolor{editorange}{RGB}{210,120,20}
\definecolor{editgreen}{RGB}{0,100,0}
\definecolor{editviolet}{RGB}{110,0,150}
\definecolor{editpink}{RGB}{190,40,120}

\def\K{\mathcal K}
\def\Kb{\mathcal K_b}
\def\P{\mathcal P}
\def\M{\mathcal M}
\def\div{\text{div}}
\def\d{\mathrm{d}}
\def\I{\mathrm{I}}
\title{\vspace{-1cm} Quantitative Diffusive Limits for Singular Nonlocal Transport}

\author{
Andrea Agazzi\thanks{
Department of Mathematics and Statistics, University of Bern;
\texttt{andrea.agazzi@unibe.ch}}
\quad
Giuseppe Bruno\thanks{
Department of Mathematics and Statistics, University of Bern;
\texttt{giuseppe.bruno@unibe.ch}}
\quad
Federico Pasqualotto\thanks{
Department of Mathematics, University of California, San Diego;
\texttt{fpasqualotto@ucsd.edu}}
\quad
Philippe Rigollet\thanks{
Department of Mathematics, Massachusetts Institute of Technology;
\texttt{rigollet@math.mit.edu}}
}

\date{\today}

\date{\today}

\begin{document}

\maketitle

\begin{abstract}
We study the nonlocal continuity equation
\[
\partial_t\mu_b
=\operatorname{div}\!\left(
\mu_b\nabla\log\bigl((I-b^2\Delta)^{-1}\mu_b\bigr)
\right)
\]
on a closed connected Riemannian manifold. For smooth strictly positive
initial data, we prove that as $b \to 0$, its global solution converges to heat flow $\mu(t)$ at the
sharp, uniform-in-time rate
\[
\sup_{t\ge0}\|\mu_b(t)-\mu(t)\|_{L^1}\le Cb^2.
\]
The key estimate is the uniform dissipation of a $b$-weighted higher-order
resolvent energy, which yields exponential relaxation despite the absence of
a Wasserstein gradient-flow structure. 

On the circle, we
also analyze the corresponding deterministic $N$-particle dynamics. A weak--strong modulated energy argument gives
\[
\mathbb E\!\left[
\sup_{t\ge0}W_1(\mu_b^N(t),\mu_b(t))
\right]
\le C(Nb)^{-1/2}
\]
for iid initialization. Consequently, the choice $b\asymp N^{-1/5}$
approximates heat flow uniformly in time at rate $N^{-2/5}$.
\end{abstract}

\section{Introduction}\label{sec:setting}\label{sec:framework}

For every smooth positive density $\mu$, the right-hand side of the heat
equation $\partial_t\mu=\Delta\mu$ admits the nonlinear factorization
\[
\Delta\mu=\operatorname{div}(\mu\nabla\log\mu).
\]
Consequently, the heat equation itself can be written as the continuity
equation
\[
\partial_t\mu+\operatorname{div}(\mu v)=0,
\qquad v=-\nabla\log\mu.
\]
This elementary reformulation is the starting point of diffusion--velocity
methods
\cite{degond1990deterministic,lacombe1999presentation,lions2001methode}:
instead of realizing diffusion through stochastic forcing, they transport
particles deterministically along a density-dependent velocity. The identity
is exact for a smooth positive density, but the velocity is undefined on an
empirical measure, whose logarithmic gradient does not exist. This singularity
naturally leads to replacing $\mu$ by a regularized density before evaluating
the velocity. Variational blob methods pursue a related construction by
regularizing the entropy before taking its Wasserstein gradient
\cite{carrillo2019blob}. At a broader level, deterministic interacting
transport also underlies particle-based variational inference, notably Stein
variational gradient descent
\cite{liu2016svgd} and subsequent schemes studied from a
Wasserstein-gradient-flow perspective.

This paper examines a new regularization scheme with uniform approximation guarantees for the heat flow.

\paragraph{Problem setting} Let $(\M,g)$ be a smooth, closed (compact and without boundary), connected Riemannian manifold of dimension $d$ with normalized volume measure $\mathrm{d}x$ and let $\P(\M)$ be the space of probability measures on $\M$. This paper studies the family of equations on $\P(\M)$
\begin{equ}
    \label{eq:pde}
    \partial_t \mu_b = \text{div}\left(\mu_b \nabla \log(\mathcal K_b \mu_b)\right)
\end{equ}
where $\text{div}$, $\nabla$ are the Riemannian divergence and gradient, and, for $b \geq 0$, $\mathcal K_b$ denotes an integral operator acting on measures as follows:
\begin{equ}\label{eq:kernel}
\mathcal K_b \mu(x) =  \int_{\mathcal M} K_b(x,y)\mu(\d y)\,,
\end{equ}
 for a family of (singular) Green kernels $K_b~:~\M\times\M\to \R$ solving, in the distributional sense,
\begin{equ}\label{eq:green}
(\I-b^2\Delta_y)K_b(x,y)=\delta_x(y).
\end{equ}
Here $\delta_x$ denotes the Dirac delta function and $\Delta$, $\I$ are, respectively, the Laplace-Beltrami and identity operators.  Equivalently, applying $\I-b^2\Delta$ to \eqref{eq:kernel}, we may characterize the operator $\Kb$ as the resolvent $(\I-b^2\Delta)^{-1}$.

At a purely formal
level, the approximation of the identity
$\mathcal K_b=\I+O(b^2)$ suggests that solutions of
\eqref{eq:pde} should converge to heat flow as $b\downarrow0$. The main issue addressed here is whether this
approximation remains accurate for arbitrarily large times and whether it continues to hold after replacing the evolving density by finitely many particles.

\paragraph{Main results} Our first main result, \Cref{thm:main_convergence}, proves quantitative, uniform-in-time convergence 
of $\mu_b$ solving \eqref{eq:pde} as $b \to 0$ to the solution $\mu$ of the heat equation 
\begin{equ}\label{eq:heat}
\partial_t \mu = \Delta \mu\, ,
\end{equ} 
with the same smooth initial condition, establishing the estimate (also holding in $L^1$)
$$W_1 (\mu, \mu_b) \leq C b^2\,.$$

Furthermore, because of its transport structure, \eqref{eq:pde} can be viewed as the continuity equation associated with a system of particles interacting through
the convolutional vector field generated by \(\Kb\). This gives a natural
interpretation of \eqref{eq:pde} as describing the dynamics of the empirical
measure \(\mu_b^N=\frac{1}{N}\sum_{i=1}^N\delta_{x_i}\) for a deterministic mean-field particle system, as detailed in Section~\ref{sec:particles_intro}. In our second main result, \Cref{thm:main_s1}, we prove, for $d = 1$, quantitative, uniform-in-time propagation of chaos estimates in $W_1$ that are polynomial in both $b$ and the number $N$ of particles:
$$ \mathbb E\!\left[
\sup_{t\ge0}W_1(\mu_b^N(t),\mu_b(t))
\right]\leq C \frac 1{\sqrt {N b}}\,.$$

Together, our results
\Cref{thm:main_convergence} and \Cref{thm:main_s1} justify  $\mu_b^N$ as an approximation of $\mu$ in the joint limit $b \to 0$ and $Nb \to \infty$. These results readily extend to the presence of an  external, noninteracting gradient field $-\nabla V(x)$ leading to an additional forcing/advection term and to the Fokker-Planck equation 
\begin{equ}\label{eq:fp}
\partial_t\mu=\Delta\mu+\div\!\left(\mu \nabla V\right)
\end{equ}
in the limit. Provided that the drift term is sufficiently regular and confining, this result carries over to the classical setting $\mathcal M = \mathbb R^d$ ($g = \mathbbm{1}$), or without a drif to $\mathcal M = \Torus^d$, so that all operators appearing in the above equations can be interpreted in the usual (non-Riemannian) sense.

\paragraph{Related work}
The present work connects two strands of the literature that have largely developed independently. In each, our results address a natural question left open by the existing theory.

First, in \cite{degond1990deterministic,lacombe1999analyse}, the authors introduce \emph{diffusion-velocity methods}: a family of deterministic, first-order particle methods for simulating second-order equations such as the heat and porous medium equations. While previous works have established mostly qualitative convergence estimates for the porous medium equation \cite{oelschlager1985law, oelschlager2001sequence,lions2001methode,figalli2008convergence,carrillo2016gradient,carrillo2019blob, peletier2025nonlinear} or for the heat equation in the presence of additional regularizing convolution kernels outside the logarithm in \eqref{eq:pde} \cite{craig2025nonlocal,carrillo2024nonlocallinear} (see Remark~\ref{r:singleconv} below), our result quantifies convergence of the dynamics of a diffusion-velocity method to the \emph{heat equation} with a \emph{single} convolution operation inside the logarithm, for the specific choice of kernel \eqref{eq:kernel}. While removing the convolution kernel outside the logarithm destroys the Wasserstein gradient flow structure of the equation, it provides significant advantages from an implementation standpoint: it yields a straightforward particle method and bypasses potentially high-dimensional numerical integration. In this sense, our result provides guarantees for a class of particle-based dynamics that underlie several recent and influential approaches to sampling algorithms; see \cite[Chapter~6]{CheNilRig25} and~\cite{Che26}.

Second, the spherical case $\mathcal M=\mathbb S^d$ is particularly relevant to a recent line of work initiated in \cite{geshkovski2024emergence,geshkovski2025mathematical}, in which Equation~\eqref{eq:pde} emerges as a central object for analyzing signal propagation across the layers of deep transformer architectures, see~\cite{Rig26}. Our results are closely connected to this line of research: the heat equation has been formally identified in \cite{sander2022sinkformers,multiscale2025transformers} as an effective description of information propagation in certain limiting hyperparameter regimes. For the specific family of kernels considered here, we rigorously establish this connection by proving the corresponding diffusion limit. This provides the first fully rigorous derivation of the heat equation as a scaling limit of this class of signal propagation dynamics arising from transformer architectures.

\subsection{Structure of the paper} 
Section~\ref{sec:main_results} specifies the setup, states our main results, and discusses related work. Section~\ref{sec:kernel_and_energy} proves uniform-in-time convergence to the heat equation and establishes the sharpness of the $b^2$ rate. Section~\ref{sec:s1} proves well-posedness of the particle dynamics on the circle and the uniform-in-time mean-field limit.

Throughout the paper, $C$ denotes a positive finite constant whose value may change from line to line. Its dependence on relevant parameters may be indicated explicitly by subscripts, as in $C_{\alpha,\beta}$.

\section{Main results}\label{sec:main_results}

In this section, we provide precise statements for each of our main results. For each of them, we also review relevant previous results in the literature.
We first discuss our choice of kernel and the well-posedness of \eqref{eq:pde}, required for our main results, which are presented in the two following subsections~\ref{subsub:pdeconvergence} and \ref{sec:particles_intro}.

\subsection{Setup}

We recall the setting and notation introduced in Section~\ref{sec:setting}.
Let $(\M,g)$ be a smooth, closed, connected Riemannian manifold of dimension $d$, equipped with its normalized volume measure $\d x$, so that $\int_{\M}\d x=1$. We write $\P(\M)$ for the space of probability measures on $\M$ and identify absolutely continuous measures with their densities relative to $\d x$. For each $b>0$, we study the nonlocal continuity equation \eqref{eq:pde},
\begin{equ}
\partial_t\mu_b
=\div\!\left(\mu_b\nabla\log(\Kb\mu_b)\right),
\qquad \mu_b(0)=\mu_{\mathrm{in}},
\end{equ}
with a common initial probability density $\mu_{\mathrm{in}}$, independent of $b$.  Here and throughout, $\div$, $\nabla$, and $\Delta$ denote the Riemannian divergence, gradient, and Laplace--Beltrami operator, respectively, and $\I$ denotes the identity operator.

\paragraph{The Bessel--Yukawa Kernel}
The convolution operator \(\Kb\) in \eqref{eq:pde} acts as a nonlocal regularization of the density\footnote{In fact, the action of $\Kb$ can also be viewed as a single implicit Euler step for the heat equation in measure space, with time step $b^2$.},
where the parameter \(b>0\) in \eqref{eq:green} governs the length scale of particle interactions. 
Formally, as \(b\downarrow0\), the resolvent \(\Kb=(\I-b^2\Delta)^{-1}\) approaches the identity, and \eqref{eq:pde} reduces to the heat equation.
We further note that the kernel \(K_b\) is singular on the diagonal. More precisely, for $d \geq 2$ , \(K_b\) has a Coulomb-type singularity: if \(r=\operatorname{dist}_{\M}(x,y)\), then as \(r\downarrow0\), 
\begin{equation} 
\label{eq:kernel_singularity}
K_b(x,y)\sim \begin{cases} -b^{-2}\log\!\left(r/b\right), \quad& d=2,\\ b^{-d}(r/b)^{2-d}, & d\ge3, \end{cases} \end{equation}
up to constants depending on \(d\), the geometry of \(\M\), and the normalization of measure. 
In the one-dimensional case $d = 1$,  \(K_b\) remains bounded but exhibits a cusp singularity at the origin. For instance, for \(\theta, \theta'  \in \Sphere^1\simeq\mathbb R/2\pi\mathbb Z\), defining $d_{\Sphere^1}(\theta,\theta') := \min_{m\in\mathbb Z}
\left|\theta-\theta'+2\pi m\right|$ we have
\begin{equation} 
\label{eq:explicit_circle_resolvent_kernel}
K_b(\theta, \theta') = \sum_{k\in\mathbb Z}\frac{e^{ik(\theta- \theta')}}{1+b^2k^2} = \frac{\pi}{b\sinh(\pi/b)} \cosh\!\left(\frac{\pi-d_{\Sphere^1}(\theta,\theta')}{b}\right). \end{equation}

\begin{remark}\label{rem:generative_modeling}
Equation~\eqref{eq:pde} and our convergence theorem are connected to generative modeling in machine learning, in particular the recent drifting models \cite{deng2026drifting,lai2026unified,turan2026drifting}. Given a target density $p$ and a smoothing kernel $K_b$, these models generate new samples from $p$ by approximating the evolution of particles driven by the (measure-dependent) vector field:
\[
v_b=\nabla\log(\Kb p)-\nabla\log(\Kb\mu_b).
\]
For the uniform target $p=1$, this gives exactly \eqref{eq:pde}, since $\Kb1=1$. For a smooth, strictly positive target $p$, the formal limit as $b\to0$ is the Fokker-Planck equation \eqref{eq:fp} with $V = \log p$.

We also note that, while the convergence outlined above occurs, heuristically, for a much wider class of kernels, the specific choice of smoothing kernel appears to affect the accuracy of the approximation in the  finite-$N$ and finite-$b$ regime. Specifically,  the Laplace kernel $k_\tau(x,y)\propto e^{-|x-y|/\tau}$ used in the original paper \cite{deng2026drifting}, and the Bessel-Yukawa kernel \eqref{eq:kernel} appear to  reduce fluctuations with respect to the smoother Gaussian kernel often used in transformers, as reported in \cite[Table~1]{turan2026drifting} and in our own numerical experiments discussed in Appendix~\ref{app:numerics}, respectively.
\end{remark}

\paragraph{Initial condition and well-posedness} Since the Bessel--Yukawa kernel is non-smooth, well-posedness of \eqref{eq:pde} does not follow directly from standard results. We include a dedicated argument for completeness, but streamline the presentation by imposing sufficiently regular and strictly positive initial data; the analysis can be extended to substantially weaker assumptions. Thus, for every $b>0$, we take
\[
\mu_b(0,\d x)=\mu_{\mathrm{in}}(x)\d x,
\]
where $\mu_{\mathrm{in}}$ satisfies the following assumption.

\begin{assumption}\label{assump:initial}
Let $\mu_{\mathrm{in}}\in C^\infty( \M)$ be a probability density on  $\M$ such that
\[
\inf_{x\in  \M}\mu_{\mathrm{in}}(x) > 0.
\]
\end{assumption}
\begin{proposition}\label{thm:pde_global_smooth_existence} Let $\mu_{\mathrm{in}}$ satisfy  \Cref{assump:initial}. For every $b>0$, the Cauchy problem
\eqref{eq:pde} with initial condition $\mu_b(0,\d x)=\mu_{\mathrm{in}}(x)\d x$ admits a unique global classical solution
\[
\mu_b\in C^\infty([0,\infty)\times \M).
\]
This solution satisfies $\int_{ \M}\mu_b(t,x)\,\d x=1,$ for all $t \ge 0$ and  $\mu_b(t,x)>0$ for all $(t,x)\in[0,\infty)\times \M$. 
Moreover,
\[
\|\mu_b(t)\|_{L^\infty( \M)}
\le
\|\mu_{\mathrm{in}}\|_{L^\infty( \M)}
\qquad\text{for all }t\ge0,
\]
and $\K_b\mu_b$ is smooth and strictly positive for all times.
\end{proposition}

The proof of \Cref{thm:pde_global_smooth_existence} is analogous
to the Yudovich argument of
\cite[Section~2]{chizat2026quantitative_kmd} for the Riesz vector field $\nabla (-\Delta^s)^{-1}\mu$ (which is linear in $\mu$, unlike the one considered in this work) and is reported in the appendix for completeness. We record the
adapted kernel and well-posedness estimates needed for the present equation in
\Cref{app:kernel_facts,app:pde_wellposedness}. The result is qualitative: the
continuation estimates may depend on the fixed value of $b$. The uniform lower
bounds, decay estimates, and Sobolev bounds used later are stronger
quantitative conclusions proved independently in \Cref{sec:kernel_and_energy}.

\subsection{Quantitative convergence to the heat flow}\label{subsub:pdeconvergence}
We now proceed to state the first main result of this paper, establishing uniform-in-time, quantitative convergence of solutions $\mu_b$ of \eqref{eq:pde} to $\mu$ solving \eqref{eq:heat} as $b \to 0$:

\begin{theorem}\label{thm:main_convergence}
Let $\mu_{\mathrm{in}}$ satisfy \Cref{assump:initial} and let $\mu_b(0,\d x) = \mu(0,\d x)=\mu_{\mathrm{in}}(x)\d x$ for all $b > 0$.  There exist constants $C>0$, and $b_0>0$ depending on $\mu_{\mathrm{in}}$ such that, for every $b>0$ with $b < b_0$ we have
\[
\sup_{t\ge0}\|\mu_b(t)-\mu(t)\|_{L^1( \M)} \le C b^2.
\]
\end{theorem}

\Cref{thm:main_convergence} is proved in \Cref{sec:kernel_and_energy}
and results from combining exponential decay estimates for $\Kb\mu_b-1$,
uniformly in $b$ (\Cref{thm:main_decay}) with Duhamel's
formula to compare $\mu_b$ with the heat flow. The decay estimates
keep the $L^1$ error of order $b^2$ uniformly in time.

We note that, while the Duhamel part of the argument requires sufficient regularity of the kernel $K_b$, obtaining our energy decay estimates requires the same kernel to be sufficiently singular (i.e., repulsive) on the diagonal. Our choice of Bessel--Yukawa kernel\eqref{eq:kernel} solving \eqref{eq:green} strikes a delicate balance between these two conditions.

The energy mechanism enabling our analysis is already visible in the first-order interaction energy associated with the Bessel potential.
Indeed, for a signed density \(\rho\) and
\(w=\Kb\rho\), we have
\begin{equation}
\label{eq:interaction_energy}
\frac12\int_{\M}\int_{\M}K_b(x,y)\rho(x)\rho(y)\,\d x\,\d y
=
\frac12\int_{\M}\bigl(|w|^2+b^2|\nabla w|^2\bigr)\,\d x,
\end{equation}
which controls a \(b\)-weighted \(H^{-1}(\M)\) norm of $\rho$. 
This quantity is monotone in time: if
\(u_b=\Kb\mu_b\), equivalently \((\I-b^2\Delta)u_b=\mu_b\), then testing
\eqref{eq:pde} against \(u_b\) gives
\[
\frac12\frac{d}{dt}\int_{\M}\bigl(|u_b|^2+b^2|\nabla u_b|^2\bigr)\,\d x
+\int_{\M}\frac{\mu_b}{u_b}|\nabla u_b|^2\,\d x=0.
\]

We note that the $b^2$ rate appearing in the statement of Theorem~\ref{thm:main_convergence} is consistent with the leading order  of the formal expansion (in $b$) of \eqref{eq:kernel} following from \eqref{eq:green}. We prove that our rates are sharp in $b$ by providing an example with a matching lower bound on $\M = \mathbb S^1$ in Section~\ref{s:lowerbound}.
\begin{remark}
Our analysis readily extends to dynamics with a smooth external potential \(V\), on a closed manifold, e.g., upon replacing \(\Delta\) by the symmetric diffusion generator
$
L_V := \Delta - \nabla (V \cdot \nabla),
$
working relative to its invariant measure $d\pi_V = Z_V^{-1} e^{-V}\,dx$, and replacing \(K_b\) by \((I-b^2L_V)^{-1}\). These results can also be extended further to \(\mathbb{R}^d\) for sufficiently confining \(V\), provided the invariant measure satisfies a spectral-gap inequality and one establishes the required weighted resolvent and commutator estimates, together with adequate control at infinity. 
\end{remark}

\begin{remark}\label{r:singleconv}
The recent works \cite{craig2025nonlocal,carrillo2024nonlocallinear}  on entropy-regularized nonlocal approximations of
the heat equation lead to continuity equations with velocity fields of the form
\[
-\nabla\bigl(V_\beta*\log(V_\beta*\mu)\bigr)
\qquad\text{and}\qquad
-\nabla\bigl(\varphi_\varepsilon*f'_\varepsilon(\varphi_\varepsilon*\mu)\bigr),
\]
where \(V_\beta\) and \(\varphi_\varepsilon\) are regularizing kernels and $*$ denotes the convolution operation. These
dynamics are Wasserstein gradient flows of regularized entropy functionals, and
involve an \emph{outer} convolution acting on the logarithmic term. By contrast,
in \eqref{eq:pde} the regularization appears only \emph{inside} the logarithm:
the drift is simply
\[
\nabla\log(\Kb\mu_b),
\]
with no further convolution acting on \(\log(\Kb\mu_b)\).

The absence of an outer convolution has both analytical and computational consequences. On the analytical side, it allows us to obtain the nonlocal-to-local limit by comparing the nonlocal evolution directly with the heat equation, thereby avoiding the compactness and commutator arguments used to treat outer-convolution models in \cite{craig2025nonlocal,carrillo2024nonlocallinear}. On the computational side, evaluating \(\Kb\mu_b^N\) for an empirical measure requires a single interaction sum, while an outer convolution introduces an additional integral over the ambient space that must itself be approximated. This difference is particularly consequential in high dimension. The simpler structure comes, however, at the cost of losing the interpretation of the dynamics as the Wasserstein gradient flow of a regularized entropy. The corresponding variational arguments are therefore unavailable, and our proof relies instead on a direct stability comparison with the heat equation.

\end{remark}

\paragraph{Discussion}
The starting point of our scheme is the elementary identity
\[
\Delta\mu=\divg\bigl(\mu\,\nabla\log\mu\bigr),
\]
which recasts diffusion as transport along the velocity field
\(-\nabla\log\mu\). Replacing \(\mu\) inside the logarithm by a smoothed density
turns this into a well-defined, mesh-free particle dynamics. This
\emph{diffusion-velocity} idea goes back to
\cite{degond1990deterministic}, with later developments by
\cite{lions2001methode,lacombe1999analyse,lacombe1999presentation,
mas2002diffusion} and the geometric reformulation of
\cite{brenier2017geometric}. These works introduced the method and documented
its numerical behaviour, but the convergence analysis was left largely open:
\cite{degond1990deterministic} defers the proof to a forthcoming paper, while
\cite{lions2001methode} sketches a proof for a simpler model and uses a system
close to ours only in the numerical experiments. Despite renewed interest in the use of particle methods for efficient sampling algorithms (see, e.g., the recent monograph~\cite{Che26}) rigorous
convergence to the heat equation was not established within this original
diffusion-velocity framework.

A parallel line of work approximates nonlinear diffusion equations
\(\partial_t\mu=\Delta(\mu^m)\) through regularized energies. Variational blob
methods were developed in \cite{craig2016blob,carrillo2019blob}, with complete
convergence results for the quadratic porous medium equation (\(m=2\)) while for \(m>2\) results are conditional  on assumed uniform bounds in $\epsilon$ and time for the solutions. The analysis was extended to general nonlinear diffusions
with \(m>1\) in \cite{burger2023porous,carrillo2023nonlocal}, see also
\cite{craig2023blob,daneri2022deterministic,
di2025approximation, amassad2025deterministic} and the extensions to
systems, higher-order equations (e.g., Cahn-Hilliard), the Landau equation, different manifolds \cite{peletier2025nonlinear}, and optimal control
\cite{doumic2024multispecies,elbar2024limit,carrillo2020particle,carrillo2024landau,craig2025blob}.

Despite the limiting equation being linear, the special case \(m=1\), corresponding to the heat
equation, is the most delicate because the first variation of the entropy is
singular when the solution does not have full support.
This case was treated only recently in
\cite{craig2025nonlocal,carrillo2024nonlocallinear}, see also
Remark~\ref{r:singleconv} for a comparison of the equations. 
While these results admit a larger class of initial conditions by leveraging the Wasserstein gradient flow structure of their equation, they establish \emph{qualitative} convergence, through compactness arguments, on
\emph{finite} time horizons, without an explicit rate between the nonlocal and
local solutions. At the particle level,
\cite{carrillo2024nonlocallinear} obtains quantitative stability estimates
which, for the entropy regularization, require an \emph{exponential}
relation between the number of particles and the inverse regularization scale.
In contrast, our dynamics involve no outer convolution and lack a Wasserstein
gradient-flow structure.

The stochastic counterpart, namely moderately interacting diffusions
converging to the viscous porous medium equation, has a long history starting
with
\cite{oelschlager1985law,oelschlager1990large,oelschlager2001sequence}, see also
\cite{figalli2008convergence,morale2005interacting,
chen2021rigorous}. The splitting
\(W_\beta=V_\beta*V_\beta\) introduced in this literature underlies much of the
later quantitative theory. However, these systems rely on diffusion at the particle
level and are therefore of a different nature from the deterministic dynamics
considered here.

Finally, recent quantitative results involving singular kernels in nonlocal continuity equations have appeared in several related but distinct settings: nonlocal-to-local limits for the quadratic porous-medium equation \cite{carrillo2025rate, carrillo2026new,amassad2025deterministic}  (in \cite{carrillo2025rate} considering the same Bessel-Yukawa kernel in $d = 1$), and Kernel Mean Discrepancy or Stein variational flows with Riesz-type kernels \cite{chizat2026quantitative_kmd,chizat2026quantitative_svgd}. None of those, however, coincide with the setting considered in this paper.

\subsection{Quantitative Propagation of Chaos \texorpdfstring{for $d=1$}{in Dimension One}}\label{sec:particles_intro}

The results presented in the previous section concern smooth solutions starting
from a smooth initial datum. We now turn to the deterministic particle
approximation, in which the initial datum is an empirical measure.
For $N \in \mathbb N$, consider a set of particles with pairwise distinct initial
positions $x_1^0,\dots,x_N^0\in \M$.  Then, the evolution of the empirical measure associated to the particles under \eqref{eq:pde} can be equivalently described in Lagrangian form as
\begin{equ}\label{eq:empirical}
\mu_b^N(t) := \frac{1}{N}\sum_{j=1}^N \delta_{x_j(t)},
\end{equ}
where single particle trajectories $x_i~:~[0, \infty) \to \M$ are solutions to 

\begin{equation}\label{eq:particle_system_main}
\dot x_i(t)
=
-\frac{\sum_{j\neq i}\nabla_x  K_b(x_i(t),x_j(t))}{\sum_{j} K_b(x_i(t),x_j(t))},
\qquad
i=1,\dots,N,
\end{equation}
with initial condition $x_i(0)=x_i^0$ for all $i \in \{1,\dots,N\}$. Here, the vector field on the right-hand side coincides with the particle-wise evaluation of $\nabla \log \K_b \mu_b^N$, under the convention that the derivatives of the self-interaction terms are taken to be zero.

 We now proceed to state our second main result, establishing uniform in time
propagation of chaos estimates, explicit in the localization parameter \(b\),
for the convergence of \(\mu_b^N\) to the solution \(\mu_b\) of
\eqref{eq:pde} as \(N \to \infty\). These estimates are formulated in the
Wasserstein distance: for any \(\nu,\eta \in \P(\M)\), we set
\[
W_1(\nu,\eta)
:= \inf_{\pi \in \Pi(\nu,\eta)}
\int_{\M\times \M} d_g(x,y)\,\pi(\d x,\d y),
\]
where \(d_g\) denotes the Riemannian distance on \(\M\), and
\(\Pi(\nu,\eta)\) is the set of couplings between \(\nu\) and \(\eta\).
We state and prove the result when  $d=1$ on $\mathbb{S}^1$, since every one-dimensional closed connected Riemannian manifold is isometric to a circle. In this setting, global existence of solution to \eqref{eq:particle_system_main} with pairwise distinct trajectories is established in
\Cref{prop:s1_particle_wellposed}. 
\begin{theorem}\label{thm:main_s1}
Let $d= 1$, $\M = \mathbb S^1$ and $\mu_b (0, \d x) = \mu_{\mathrm{in}}(x)\d x$ for $\mu_{\mathrm{in}}$ satisfying \Cref{assump:initial}. Furthermore, let $\{x_i(0)\}_{i=1}^N$ be drawn iid from $\mu_{\mathrm{in}}$. Then there exist $b_0\in(0,1]$ and $C>0$, such that for every $0<b\le b_0$,
\[
\mathbb E\Bigl[\sup_{t\ge0}W_1\bigl(\mu^N_b(t),\mu_b(t)\bigr)\Bigr]
\le
\frac{C}{\sqrt{N b}}.
\]
\end{theorem}

Theorem~\ref{thm:main_s1} is proved in \Cref{sec:s1}.
We note that empirical measures, the main object of this result, are too singular for the Sobolev-based machinery  developed in the previous sections to apply, and as such require a different class of techniques. 
To obtain an
estimate that is both uniform in time and algebraic in $N$ and $b$, we use a
modulated energy, in the spirit of Serfaty's modulated energy method for
Coulomb and Riesz mean-field limits
\cite{serfaty2017mean,duerinckx_serfaty_2020}. 
Two features distinguish our argument from standard modulated energy proofs.
First, the interaction is the resolvent quadratic form, whose cusp singularity
on $\Sphere^1$ is mild enough that the estimate closes without truncating the
kernel and without removing self-interactions from the energy definition. Second, and more substantially, the field $\nabla\log(\Kb\mu)$ is
nonlinear in the density. We control the resulting ratio $\Kb\mu_b/\Kb u^N_b$  through the
positivity of the resolvent quadratic form, and the only a-priori estimate required is the
integrated decay of $\log \Kb\mu_b$ provided by \Cref{thm:main_decay}. It is precisely this decay that
makes the modulated energy estimate uniform in time rather than exponentially
growing in $T$.

The resulting bound has order $(Nb)^{-1/2}$. This scaling coincides with the condition $Nb\to\infty$ for the consistency of the kernel density estimator $\Kb \mu^N_b$ at $t=0$ to the initial condition~\cite{Dev83}.

\begin{remark}
\label{rem:truncation}
The boundedness of $K_b$ in dimension $d=1$ allows us to streamline the proof of \Cref{thm:main_s1} and focus on the main novelty of our argument, namely the nonlinear dependence of the vector field on the measure.
In higher dimensions, \(K_b\) exhibits the singularity of the intrinsic Coulomb kernel, suggesting that the truncation methods developed for singular Coulomb/Riesz interactions, such as those in \cite{duerinckx_serfaty_2020}, provide the natural route for extending the argument. We do not pursue this direction here, since it would require
adapting technical truncation estimates of a type already well developed in the
modulated energy literature, while our focus is on the genuinely different
point, i.e., the nonlinearity of the vector field, already present in dimension one.
\end{remark}

We conclude by combining our main results \Cref{thm:main_convergence} and \Cref{thm:main_s1} to relate the dynamics of the finite-$b$ particle system to the solution $\mu$ to the heat equation \eqref{eq:heat} whenever $b=b_N\to0$ and $Nb_N\to\infty$: 

\begin{corollary}
Let the assumptions of \Cref{thm:main_s1} hold. If $b=b_N\to0$ and $N b_N\to\infty$, then
\[
\mathbb E\Bigl[\sup_{t\ge0}W_1\bigl(\mu_b^N(t),\mu(t)\bigr)\Bigr]\to0.
\]
Moreover, choosing $b_N\asymp N^{-1/5}$, we get the following quantitative rate of convergence:
\[
\mathbb E\Bigl[\sup_{t\ge0}W_1\bigl(\mu_b^N(t),\mu(t)\bigr)\Bigr]\le C N^{-2/5}.
\]
\end{corollary}

\begin{proof}
For \(N\) large enough, \(b_N\) is in the range of
\Cref{thm:main_convergence,thm:main_s1}. By the triangle inequality for
\(W_1\),
\[
\mathbb E\Bigl[\sup_{t\ge0}W_1\bigl(\mu_b^N(t),\mu(t)\bigr)\Bigr]
\le
\mathbb E\Bigl[\sup_{t\ge0}W_1\bigl(\mu_b^N(t),\mu_b(t)\bigr)\Bigr]
+
\sup_{t\ge0}W_1\bigl(\mu_b(t),\mu(t)\bigr).
\]
On the compact manifold \(\M\), \(W_1(\rho,\eta)\le
\operatorname{diam}(\M)\|\rho-\eta\|_{L^1(\M)}\) for probability densities $\rho,\eta\in \mathcal{P}(\mathcal{M})$. Hence, using \Cref{thm:main_s1}
for the first term and \Cref{thm:main_convergence} for the second,
\[
\mathbb E\Bigl[\sup_{t\ge0}W_1\bigl(\mu_b^N(t),\mu(t)\bigr)\Bigr]
\le
\frac{C}{\sqrt{N b_N}}+C b_N^2,
\]
which tends to zero under the given assumptions. 

The second statement follows readily from the above display.
\end{proof}
The above result provides the first rigorous
convergence of a diffusion-velocity particle system to the heat equation, with
a regularization that is at once analytically convenient and cheap to simulate,
and with time-uniform quantitative rates at every level: from the regularized
PDE to the heat flow, and from the particles to the regularized PDE. The
higher-dimensional mean-field limit, where $K_b$ carries a genuine
Coulomb-type singularity (\eqref{eq:kernel_singularity}) and the truncation
machinery of \cite{duerinckx_serfaty_2020} is expected to apply, is left for future work.

\paragraph{Discussion} 
The large-\(N\) limit in \Cref{thm:main_s1} is related to the
propagation-of-chaos literature for singular mean-field systems. The classical
Wasserstein stability estimates for Lipschitz vector fields, starting from
Dobrushin's argument \cite{dobrushin1979vlasov} (see also
\cite{sznitman1991topics,chaintron2022propagation}) do not
apply directly here because of the singularity of the kernel. 
Our proof builds instead on the weak--strong approach of the modulated energy framework developed for Coulomb and Riesz interactions \cite{serfaty2017mean,duerinckx2016mean,serfaty2020mean, duerinckx_serfaty_2020,nguyen2021mean,rosenzweig2023global, rosenzweig2026sharp} and, in dimension one, the argument is also close in spirit to \cite{hauray2012mean}. The estimates have to be adapted to the present nonlinear vector field, \( \nabla\log(\Kb\mu), \) with bounds that are quantitative in the regularization parameter \(b\) and uniform in time.

Our particle analysis on $\Sphere^1$ uses the modulated energy method
introduced by \cite{serfaty2017mean} and developed for singular and Riesz
interactions in
\cite{duerinckx_serfaty_2020,bresch2019modulated,rosenzweig2026sharp}.
In its usual form this method controls the mean-field limit of systems whose
force is \emph{linear} in the density, $F=\nabla K*\mu$, with $K$ a Coulomb or
Riesz kernel, as discussed above. The dynamics studied here are different in two respects: the
velocity $-\nabla\log(\Kb\mu)$ is a \emph{nonlinear} function of the density,
and the relevant kernel is the Green kernel \eqref{eq:green} (also interpretable as a resolvent, see \eqref{eq:resolvent})
rather than a
Riesz potential. The modulated energy must therefore be adapted to this
nonlinearity, which we carry out using the positivity of the resolvent
quadratic form.

\section{Convergence to the heat equation}
\label{sec:kernel_and_energy}
This section is devoted to the proof of \Cref{thm:main_convergence}. 
We study equation \eqref{eq:pde} on $\P(\M)$, which we rewrite to streamline notation as 
\begin{equ}\label{eq:pde2}
\partial_t \mu_b
=
\divg\!\left(\mu_b \nabla \log u_b
\right),
\qquad
u_b=\Kb\mu_b.
\end{equ}
where $\mathcal K_b$, defined in \eqref{eq:kernel} and \eqref{eq:green}, can be equivalently defined for $b> 0$ as the resolvent
\begin{equ}\label{eq:resolvent}
\Kb := (\I - b^2 \Delta)^{-1}.
\end{equ}
In other words, in \eqref{eq:pde2}, the drift evolving the probability measure $\mu_b$ is generated by a smoothed version $u_b$ of $\mu_b$ obtained solving on  $\M$
\begin{equ}\label{eq:invresolvent}
\mu_b = u_b - b^2 \Delta u_b.
\end{equ}

The proof of~\Cref{thm:main_convergence} results from the composition of the following two results. The first (\Cref{prop:w1_b2}) is a comparison estimate, based on the Duhamel formula, establishing the desired bound conditional on an exponential decay of $u_b$ to $1$. The second (\Cref{thm:main_decay}) proves that such decay estimates hold based on the dissipation of a suitable energy, defined below.

\label{sec:proofsketch}

\begin{proposition}[Heat Comparison]\label{prop:w1_b2}
Let \(T\in(0,\infty]\) and \(b>0\), and let \(\mu\) and \(\mu_b\)
have the common initial condition
\(\mu(0)=\mu_b(0)=\mu_{\mathrm{in}}\), where
\(\mu_{\mathrm{in}}\) satisfies \Cref{assump:initial}.
Suppose that, for some \(c_\ast,A,\lambda>0\),
\begin{equation}
\label{eq:heat_comp_positivity}
u_b(t,x)\ge c_\ast
\qquad\text{for all }(t,x)\in[0,T)\times \M,
\end{equation}
and
\begin{equation}
\label{eq:heat_comp_sobolev}
\|u_b(t)-1\|_{H^3( \M)}
\le
Ae^{-\lambda t}
\qquad\text{for all }t\in[0,T).
\end{equation}
Then there exists $C>0$, depending only on  $(\M,g)$, $c_\ast,A,\lambda$, such that
for every $t\in[0,T)$,
\[
\|\mu_b(t)-\mu(t)\|_{L^1( \M)}
\le
Cb^2.
\]
\end{proposition}
We defer the proof of this result to \Cref{subsec:comp_mub_mu}.
We now turn the interaction energy observation from \eqref{eq:interaction_energy} into the
a priori estimate needed for the convergence proof. Since the equilibrium is
the constant density \(1\), the natural variable is the zero-mean fluctuation
\[
v_b:=u_b-1=\mathcal K_b(\mu_b-1)
\]
for $u_b$ defined in \eqref{eq:pde2}. The corresponding first-order energy \eqref{eq:interaction_energy} controls only \(\|v_b\|_{L^2}\) and
\(b\|\nabla v_b\|_{L^2}\). To estimate the nonlinear drift
\(\nabla\log u_b\), however, we need \(W^{2,\infty}\) control of \(v_b\):
by Sobolev embedding on  \(\M\), this follows from an \(H^s\)
bound with \(s> d/2+2\). We therefore use the following
higher-order analogue of the resolvent energy.
Set
\begin{equ}\label{eq:m}
m:=\min\Bigl\{\ell\in\mathbb N:\ 2\ell+1>\frac{ d}{2}+2\Bigr\},
\end{equ}
and define
\begin{equ}\label{eq:energy}
\mathcal E_m[w]
:=
\frac12
\int_{ \M}
\bigl(
|\nabla \Delta^m w|^2
+
b^2 |\Delta^{m+1} w|^2
\bigr)\, dx.
\end{equ}
For the solution \(v_b\), we write \(\mathcal E_m(t):=\mathcal E_m[v_b(t)]\).

The next theorem shows that the energy decays exponentially, uniformly in
$b$, and yields the bounds needed in \Cref{prop:w1_b2}.
 Throughout, we denote by $\lambda_1=\lambda_1(\M,g)>0$ the first positive eigenvalue of $-\Delta$, equivalently the spectral gap on mean-zero functions.

\begin{theorem}[Uniform energy decay and regularity]\label{thm:main_decay}
Let $\mu_b$ solve \eqref{eq:pde} with $\mu_b(0)=\mu_{\mathrm{in}}$ for $\mu_{\mathrm{in}}$ satisfying \Cref{assump:initial}. Set 
\[
c_0:=\inf_{x\in \M}\mu_{\mathrm{in}}(x)>0.
\]
Then there exist $b_0>0$ and $C>0$, depending only on  $(\M,g)$ and $\mu_{\mathrm{in}}$, such that for every $0<b\le b_0$ and every $t\ge0$,
\[
\mathcal E_m(t)\le E_0 e^{-\frac{ \lambda_1}{4}t},
\quad\text{where}\
E_0:=\sup_{0<b\le1}\mathcal E_m(0)<\infty,
\]
\[
u_b(t,x)\ge \frac{3}{4}c_0
\qquad\text{for all }x\in \M,
\]
and
\[
\|u_b(t)-1\|_{H^{2m+1}( \M)}+
\|u_b(t)-1\|_{W^{2,\infty}( \M)}+\|\nabla \log u_b(t)\|_{W^{1,\infty}( \M)}\le C e^{-\frac{ \lambda_1}{8}t}.
\]
\end{theorem}

\noindent
This is proved in \Cref{sec:preliminaries} and \Cref{sec:proofintermediate} below. We now proceed to use this result to
conclude the proof of \Cref{thm:main_convergence}.

\begin{proof}[Proof of \Cref{thm:main_convergence}]
Setting \(\lambda:=\lambda_1/8\), by \Cref{thm:main_decay} we have that, for $b \leq b_0$
\[
\|u_b(t)-1\|_{H^{2m+1}(\M)}
\le
Ce^{-\lambda t}.
\]
    Since \(2m+1>d/2+2\), in particular \(2m+1\ge3\), and together with the lower bound $u_b\geq \frac{3}{4}c_0$ we can apply \Cref{prop:w1_b2}, proving the claim.
\end{proof}

\begin{remark}\label{rem:entropy_route}
A similar uniform-in-time convergence statement can also be obtained by
studying the evolution of the relative entropy, or Kullback--Leibler
divergence, between $\mu_b$ and the heat flow $\mu$:
\[
\mathcal H(\mu_b\mid\mu)
:=
\int_\M \mu_b\log\frac{\mu_b}{\mu}\,dx .
\]
We include a sketch of
this alternative argument in \Cref{app:entropy_convergence}.  We do not follow
this route in the main proof because the particle analysis in \Cref{sec:s1} requires a
uniform estimate on
\[
\|\nabla^2\log u_b\|_{L^\infty(\M)}
\]
and such a bound is not an obvious byproduct of the entropy argument, whereas
it follows directly from the interaction-energy estimates used above.
\end{remark}

We now proceed to prove the main results of this section.

\subsection{Proof of \Cref{prop:w1_b2}}
\label{subsec:comp_mub_mu}
We aim to control the difference
\[
\nu_b:=\mu_b-\mu
\]
in $L^1$, where $\mu$ solves \eqref{eq:heat}.
Combining \eqref{eq:pde2} with \eqref{eq:invresolvent}
we obtain
\[
\partial_t\mu_b
=
\divg\bigl(\nabla u_b-b^2(\Delta u_b)\nabla\log u_b\bigr)
=
\Delta \mu_b + b^2\Delta^2 u_b-b^2\divg\bigl((\Delta u_b)\nabla\log u_b\bigr),
\]
and since $\mu$ evolves according to \eqref{eq:heat} we have
\begin{equation}\label{eq:comparison_error_eq}
\partial_t\nu_b-\Delta\nu_b
=
b^2\divg F_b,
\qquad
F_b:=\nabla\Delta u_b-(\Delta u_b)\nabla\log u_b.
\end{equation}
Moreover, $\nu_b(0)=0$ because $\mu_b(0)=\mu(0)=\mu_{\mathrm{in}}$.

Let now
$P_t:=e^{t\Delta}$ be the heat semigroup on  $\M$. The Duhamel
formula for \eqref{eq:comparison_error_eq} gives
\begin{equ}\label{eq:duhamel}
\|\nu_b(t)\|_{L^1( \M)}
\le
b^2\int_0^t \|P_{t-s}\divg F_b(s)\|_{L^1( \M)}\,ds.
\end{equ}
Using $L^1-L^\infty$ duality, self-adjointness of $P_t$, and integration by parts  we can bound the integrand as follows:
\begin{equs}
\|P_{t-s}\divg F_b(s)\|_{L^1}
& \le
\sup_{\|\phi\|_{L^\infty}\le 1}
\int_{ \M} |F_b(s)|\,|\nabla P_{t-s}\phi| \, \d x
\\
&\le
C(t-s)^{-1/2}\|F_b(s)\|_{L^1},\label{eq:bound1}
\end{equs}
where in the last line we have used a classical smoothing property of the heat semigroup, i.e., the existence of a constant $C(\M)$ such that
\begin{equ}
\|\nabla P_\tau\phi\|_{L^\infty( \M)}
\le
C( \M) \tau^{-1/2}\|\phi\|_{L^\infty( \M)},
\end{equ}
proved for completeness in the manifold setting in \Cref{lem:heat_semigroup_smoothing}.
 Since for any $\lambda > 0$ and for all $t$
 $$\int_0^t (t-s)^{-1/2}e^{-\lambda s}\,ds \leq  \int_0^\infty |t-s|^{-1/2}e^{-\lambda s}\,ds\leq C_\lambda<\infty$$
 the desired, uniform-in-time bound 
\begin{equation}\label{eq:comparison_l1}
\|\mu_b(t)-\mu(t)\|_{L^1( \M)}\le Cb^2
\end{equation}
follows immediately combining \eqref{eq:duhamel} and \eqref{eq:bound1}, conditionally on establishing the  uniform-in-$b$ dissipation estimate
\begin{equ}
    \|F_b(s)\|_{L^1} = \|\nabla\Delta u_b(s)-(\Delta u_b(s))\nabla\log u_b(s)\|_{L^1} \leq C e^{- \lambda s}.\label{eq:bound_F}
\end{equ}

Since
\(\Delta 1=0\), we have
\[
\|\nabla\Delta u_b(t)\|_{L^1}
=
\|\nabla\Delta(u_b(t)-1)\|_{L^1}
\le
C\|u_b(t)-1\|_{H^3}
\le
Ce^{-\lambda t}.
\]
Similarly,
\[
\|\Delta u_b(t)\|_{L^2}
=
\|\Delta(u_b(t)-1)\|_{L^2}
\le
C\|u_b(t)-1\|_{H^2}
\le
Ce^{-\lambda t}
\]
and, since $u_b\geq c_*$,
\[
\|\nabla\log u_b(t)\|_{L^2}
\le
C\|\nabla u_b(t)\|_{L^2}
\le
Ce^{-\lambda t}.
\]
Therefore, 
\[
\begin{aligned}
\|F_b(t)\|_{L^1}
&\le
\|\nabla\Delta u_b(t)\|_{L^1}
+
\|\Delta u_b(t)\|_{L^2}
\|\nabla\log u_b(t)\|_{L^2}  \\
&\le
Ce^{-\lambda t}
+
Ce^{-2\lambda t}
\le
Ce^{-\lambda t}.
\end{aligned}
\]
yielding \eqref{eq:bound_F} and \eqref{eq:comparison_l1} by the preceding comparison principle.

\subsection{Preliminaries to the proof of \Cref{thm:main_decay}}
\label{sec:preliminaries}
\subsubsection{Kernel and Regularity estimates}

Two other properties of $K_b$ are needed in the proofs below: its Sobolev regularity at
the pole, where it is not smooth, and its strict positivity on nonnegative
measures. We record them in the next two lemmas.

\begin{lemma}[Sobolev regularity of the Bessel--Yukawa kernel]\label{lem:Kb_kernel_sobolev}
For every $b>0$ and every $y\in \M$, the Green kernel
$K_b(\cdot,y)$ of $\I-b^2\Delta$ satisfies
\[
K_b(\cdot,y)\in H^s( \M)
\qquad\Longleftrightarrow\qquad
s<2-\frac{ d}{2}=\frac{ 4-d}{2}.
\]
In particular, $K_b(\cdot,y)\notin H^{2- d/2}( \M)$, and the
threshold is sharp. Moreover,
\[
K_b(\cdot,y)\in C^\infty( \M\setminus\{y\}).
\]
\end{lemma}

\begin{proof}[Proof of \Cref{lem:Kb_kernel_sobolev}]
The point mass $\delta_y$ belongs to $H^{-\alpha}(\M)$ exactly when
$\alpha> d/2$. Since
\[
(\I-b^2\Delta)K_b(\cdot,y)=\delta_y
\]
and $\I-b^2\Delta:H^s(\M)\to H^{s-2}(\M)$ is an
isomorphism for all $s\in\mathbb R$, the claimed Sobolev threshold follows.
Smoothness away from the pole follows from interior regularity, for instance
\cite[Theorem~6.17]{gilbarg_trudinger_2001}.
\end{proof}

\begin{lemma}[Positivity of the kernel]\label{lem:Kb_positive}
For each $b>0$, the operator $\Kb=(\I-b^2\Delta)^{-1}$ is positive on
finite nonnegative measures on  $\M$. More precisely, if $\nu$ is a
finite nonnegative Borel measure with $\nu( \M)>0$, then 
\[
\Kb\nu(x)>0
\qquad\text{for every }x\in \M.
\]
\end{lemma}

\begin{proof}[Proof of \Cref{lem:Kb_positive}]
Let \(p_s(x,y)\) denote the heat kernel on  \(\M\) with respect to
the normalized  Riemannian volume measure, that is,
\[
e^{s\Delta}f(x)=\int_{ \M}p_s(x,y)f(y)\,dy .
\]
The kernel green function can be written as:
\begin{equation}\label{eq:Kb_heat_kernel_representation}
K_b(x,y)=\int_0^\infty e^{-t}p_{b^2t}(x,y)\,dt.
\end{equation}
Indeed, for smooth \(f\), if
\[
w:=\int_0^\infty e^{-t}e^{b^2t\Delta}f\,dt,
\]
then
\[
(I-b^2\Delta)w
=
\int_0^\infty e^{-t}(I-b^2\Delta)e^{b^2t\Delta}f\,dt
=
-\int_0^\infty \frac{d}{dt}\left(e^{-t}e^{b^2t\Delta}f\right)\,dt
=f.
\]
Therefore \(w=(\I-b^2\Delta)^{-1}f=\K_bf\), which gives the kernel formula.
Consequently,
\[
K_b(x,y)=\int_0^\infty e^{-t}p_{b^2t}(x,y)\,dt
\ge
\int_1^2 e^{-t}p_{b^2t}(x,y)\,dt .
\]
The last integrand is  strictly positive and continuous on
\([1,2]\times \M\times \M\). Hence
\begin{equation}\label{eq:Kb_positive_lower_bound}
\kappa_b:=\inf_{x,y\in\M}\int_1^2 e^{-t}p_{b^2t}(x,y)\,dt>0,
\end{equation}
 because \(\M\) is closed and connected, so \(p_s(x,y)>0\) for every \(s>0\).
Therefore, for every $x\in \M$,
\[
\Kb\nu(x)
=
\int_{ \M}K_b(x,y)\,d\nu(y)
\ge
\kappa_b\,\nu( \M)
>
0\,,
\]
concluding the proof.
\end{proof}

\subsubsection{Commutator and Remainder Estimates}
\label{app:commutator_remainder_estimates}

We collect the commutator estimate and remainder bound used in the energy
argument.

\begin{lemma}\label{lem:mixed_moser}
Let $\mathcal{M}$ be a closed Riemannian manifold.  Let \(r,a,b\ge0\) be integers, and let
\(P\in\operatorname{Diff}^a(\mathcal{M})\) and \(Q\in\operatorname{Diff}^b(\mathcal{M})\) be 
smooth-coefficient differential operators with \(a+b\le r\).  Then, for all
smooth functions $F,G$ on $\mathcal{M}$,
\[
\|P(F) Q(G)\|_{L^2(\mathcal{M})}
\le
C_{M,r,P,Q}
\bigl(
\|F\|_{L^\infty(\M)}\|G\|_{H^r(\mathcal{M})}
+
\|F\|_{H^r(\M)}\|G\|_{L^\infty(\mathcal{M})}
\bigr).
\]
\end{lemma}

\begin{proof}
The case $r=0$ follows from H\"older's inequality, since \(P\) and \(Q\) are
zeroth-order operators.  For $r\ge1$, choose a finite atlas and a subordinate
partition of unity.  The localized coordinate Sobolev norms are equivalent to
the \(H^r(\mathcal{M})\) norms on the closed manifold, see, for instance,
\cite[Chapter~13, Section~1]{taylor2011pde3}.  
The Euclidean product estimate
\cite[Chapter~13, Proposition~3.6]{taylor2011pde3} gives
\[
\|\partial^\alpha F\,\partial^\beta G\|_{L^2(\R^n)}
\le
C
\bigl(
\|F\|_{L^\infty(\R^n)}\|G\|_{H^r(\R^n)}
+
\|F\|_{H^r(\R^n)}\|G\|_{L^\infty(\R^n)}
\bigr).
\]
Summing over the finitely many components and coordinate patches proves the claim.
\end{proof}

\begin{lemma}[Divergence commutator on closed manifolds]\label{lem:integer_commutators}
Let $\mathcal{M}$ be a closed  Riemannian manifold and let $m\ge1$ be an integer.  For smooth
$f\in C^\infty(\mathcal{M})$ and $X\in \Gamma(T\mathcal{M})$, define
\[
\mathcal B_m[f,X]
:=
\Delta^m\divg(fX)-\divg((\Delta^m f)X).
\]
Then
\[
\begin{aligned}
\|\mathcal B_m[f,X]\|_{L^2(\mathcal{M})}
\le
C_{M,m}
\bigl(
\|X\|_{W^{1,\infty}(\mathcal{M})}\|f\|_{H^{2m}(\mathcal{M})}
+
\|X\|_{H^{2m+1}(\mathcal{M})}\|f\|_{L^\infty(\mathcal{M})}
\bigr).
\end{aligned}
\]
\end{lemma}

\begin{proof}
Set \(L_Xf:=df(X)\) and
\(\mathcal C_m[f,X]:=[\Delta^m,L_X]f\).   

For \(m=1\) we have, in abstract index notation,
\[
[\Delta,L_X]h
=
(\Delta X^i)\nabla_i h
+
2\nabla_iX^j\,\nabla_i\nabla_jh
+
X^j\Ric_{j\ell}\nabla^\ell h ,
\]

When $m \geq 1$, we use the algebraic identity
\[
[\Delta^m,L_X]
=
\sum_{j=0}^{m-1}\Delta^j[\Delta,L_X]\Delta^{m-1-j},
\]
and Leibniz' rule then show that, after expressing all terms in local coordinates, \(\mathcal C_m[f,X]\) is a finite sum of
smooth-coefficient products of the form
\[
(\partial^\alpha \partial_j X^k)(\partial^\beta f),\qquad |\alpha|+|\beta|\le2m,
\]
together with lower-order products \(X\,\partial^\gamma f\), \(|\gamma|\le2m-1\). Here, $\alpha$, $\beta$  and $\gamma$ are multi-indices in the local coordinates, $|\alpha|$ denotes the order of the multi-index $\alpha$, and $\partial^\alpha$ denotes differentiation with respect to the multi-index $\alpha$ in the local coordinates. The operators
appearing in this finite sum have smooth and bounded coefficients depending
only on \(M\) and \(m\), including the curvature terms.

Hence
Lemma~\ref{lem:mixed_moser}, applied to each of these terms, gives
\[
\|\mathcal C_m[f,X]\|_{L^2}
\le
C
\bigl(
\|X\|_{W^{1,\infty}}\|f\|_{H^{2m}}
+
\|X\|_{H^{2m+1}}\|f\|_{L^\infty}
\bigr).
\]

To conclude, consider the identity

\[
\mathcal B_m[f,X]
=
\mathcal C_m[f,X]+\Delta^m(f\divg X)-(\divg X)\,(\Delta^m f)
\]
Expanding \(\Delta^m(f\, \divg X)\) by Leibniz' rule and applying
Lemma~\ref{lem:mixed_moser} (in local coordinates) to each term gives
\[
\|\Delta^m(f\ \divg X)\|_{L^2}
\le
C
\bigl(
\|\divg X\|_{L^\infty}\|f\|_{H^{2m}}
+
\|\divg X\|_{H^{2m}}\|f\|_{L^\infty}
\bigr),
\]
while $\|(\divg X)\,\Delta^m f\|_{L^2}\le\|\divg X\|_{L^\infty}\|f\|_{H^{2m}}$.  Since
\[
\|\divg X\|_{L^\infty}\le C\|X\|_{W^{1,\infty}},
\qquad
\|\divg X\|_{H^{2m}}\le C\|X\|_{H^{2m+1}},
\]
the stated estimate follows.
\end{proof}

\begin{corollary}[Remainder bound]\label{cor:compact_remainder}
Let $m\in\mathbb N$, let $u$ be a smooth probability density on
$\mathcal{M}$, and set $v:=u-1$.  Assume that
\[
u\ge c_0>0,
\qquad
\|v\|_{W^{2,\infty}(\mathcal{M})}\le M_0
\]
for some constants $c_0,M_0>0$.  Define
\[
\mathcal R_m[u]
:=
\Delta^m \divg\bigl((\Delta u)\nabla \log u\bigr)
-
\divg\bigl((\Delta^{m+1}u)\nabla \log u\bigr).
\]
Then there exists $C=C(d,m,c_0,M_0)>0$ such that
\[
\|\mathcal R_m[u]\|_{L^2(\mathcal{M})}
\le
C\|v\|_{H^{2m+2}(\mathcal{M})}.
\]
\end{corollary}

\begin{proof}
Set
\[
f:=\Delta u=\Delta v,
\qquad
X:=\nabla\log u .
\]
Then $\mathcal R_m[u]=\mathcal B_m[f,X]$.  Applying
\Cref{lem:integer_commutators} on $\mathcal{M}$ gives
\[
\|\mathcal R_m[u]\|_{L^2}
\le
C
\bigl(
\|X\|_{W^{1,\infty}}\|f\|_{H^{2m}}
+
\|X\|_{H^{2m+1}}\|f\|_{L^\infty}
\bigr).
\]
Moreover,
\[
\|f\|_{H^{2m}}\le C\|v\|_{H^{2m+2}},
\qquad
\|f\|_{L^\infty}\le C\|v\|_{W^{2,\infty}}\le CM_0.
\]
The identities
\[
\nabla\log u=\frac{\nabla v}{u},
\qquad
\nabla^2\log u
=
\frac{\nabla^2v}{u}
-
\frac{\nabla v\otimes\nabla v}{u^2}
\]
and the hypotheses on $u$ and $v$ imply
\[
\|X\|_{W^{1,\infty}}\le C(c_0,M_0).
\]

Since the measure on $\mathcal{M}$ is normalized, the probability density
$u$ has average one, and hence $c_0\le1$.  Thus $u=1+v$ takes values in
$[c_0,1+M_0]$ and $0\in[c_0-1,M_0]$.  Choose
$G\in C^\infty(\R)$ with bounded derivatives of every order such that
$G(z)=\log(1+z)$ on $[c_0-1,M_0]$ and $G(0)=0$.  The composition estimate
\cite[Theorem~2.87]{bahouri2011fourier}, applied after localization and using
$H^r=B^r_{2,2}$ for integer $r$, gives, with constants depending only on
finitely many derivatives of this fixed cutoff \(G\), hence only on
\(c_0,M_0,m,d\),
\[
\|\log u\|_{H^{2m+2}}
=
\|G(v)\|_{H^{2m+2}}
\le
C(c_0,M_0,m,d)\|v\|_{H^{2m+2}}.
\]
Therefore
\[
\|X\|_{H^{2m+1}}
\le
C\|\log u\|_{H^{2m+2}}
\le
C(c_0,M_0,m,d)\|v\|_{H^{2m+2}}.
\]
Combining the preceding estimates proves the claim.
\end{proof}

\subsubsection{Energy and Sobolev estimates}
\label{sec:sobolev}

 We recall here the higher-order energy $\mathcal E_m$ from \eqref{eq:energy},
now as a functional of a general test function, and introduce the corresponding
dissipation $\mathcal D_m$. 
Let
\[
v_b := u_b - 1,
\]
and $m\in \mathbb{N}$ as in \eqref{eq:m}, for a smooth function $w$, define

\[
\mathcal E_m[w]
:=
\frac12
\int_{ \M}
\bigl(
|\nabla \Delta^m w|^2
+
b^2 |\Delta^{m+1} w|^2
\bigr)\, dx,
\]

\[
\mathcal D_m[w]
:=
\int_{ \M}
|\Delta^{m+1} w|^2\, dx.
\]
For the solution $v_b$, we write $\mathcal E_m(t) := \mathcal E_m[v_b(t)]$ and $\mathcal D_m(t) := \mathcal D_m[v_b(t)]$. Since $\mu_b(t)$ is a probability density and $K_b 1 = 1$, we have $\int_{ \M} u_b(t)\, dx = 1$, hence $v_b(t)$ has zero mean for every $t \ge 0$.

The first lemma collects the comparisons between $\mathcal E_m$, $\mathcal D_m$,
and Sobolev norms used throughout the bootstrap.
\begin{lemma}[Comparison for $H^{2m+1}$, $\mathcal E_m$, and $\mathcal D_m$]\label{lem:sobolev_control}
Let $w \in C^\infty( \M)$ have zero mean. Then, for every $0 < b \le 1$,
\[
\|w\|_{H^{2m+1}( \M)}^2
\le
C_{ \M,m}\mathcal E_m[w].
\]
Also,
\[
\|w\|_{H^{2m+2}( \M)}^2
\le
C'_{ \M,m}\mathcal D_m[w].
\]
If, in addition, $0 < b \le \lambda_1^{-1/2}$, then
\begin{equation}\label{eq:conditional_coercivity}
\mathcal D_m[w] \ge  \lambda_1\mathcal E_m[w].
\end{equation}
\end{lemma}

\begin{proof}
For every $k\ge0$, the function $\Delta^k w$ has zero mean. Therefore the Poincar\'e inequality on  $\M$ gives
\begin{equation}\label{eq:poincare_delta_k}
\|\Delta^k w\|_{L^2( \M)}^2
\le
\frac{1}{ \lambda_1}\|\nabla \Delta^k w\|_{L^2( \M)}^2.
\end{equation}
By integration by parts and \eqref{eq:poincare_delta_k},
\begin{align*}
\|\nabla \Delta^k w\|_{L^2( \M)}^2
&=
-\int_{ \M}\Delta^k w\,\Delta^{k+1} w\, dx \\
&\le
\|\Delta^k w\|_{L^2( \M)}\|\Delta^{k+1}w\|_{L^2( \M)} \\
&\le
\frac{1}{\sqrt{ \lambda_1}}\|\nabla \Delta^k w\|_{L^2( \M)}
\|\Delta^{k+1}w\|_{L^2( \M)}.
\end{align*}
hence
\begin{equation}\label{eq:grad_lap_control}
\|\nabla \Delta^k w\|_{L^2( \M)}^2
\le
\frac{1}{ \lambda_1}\|\Delta^{k+1}w\|_{L^2( \M)}^2.
\end{equation}
Combining \eqref{eq:poincare_delta_k} and \eqref{eq:grad_lap_control}, we also obtain
\begin{equation}\label{eq:lap_lap_control}
\|\Delta^k w\|_{L^2( \M)}^2
\le
\frac{1}{ \lambda_1^2}\|\Delta^{k+1}w\|_{L^2( \M)}^2.
\end{equation}

By \cite[Corollary~3.1]{xu2004spectral}, applied with $r=2$ on the closed Riemannian manifold  $\M$, there exists $C_{ \M,m}>0$ such that
\[
\|w\|_{H^{2m+1}( \M)}^2
\le
C_{ \M,m}\sum_{k=0}^m
\left(
\|\Delta^k w\|_{L^2( \M)}^2
+
\|\nabla\Delta^k w\|_{L^2( \M)}^2
\right).
\]
Using \eqref{eq:poincare_delta_k}, \eqref{eq:grad_lap_control} and \eqref{eq:lap_lap_control} iteratively, every term on the right-hand side is bounded by $C_{ \M,m}\|\nabla\Delta^m w\|_{L^2( \M)}^2$. Therefore
\[
\|w\|_{H^{2m+1}( \M)}^2
\le
C_{ \M,m}\|\nabla\Delta^m w\|_{L^2( \M)}^2
\le
2C_{ \M,m}\mathcal E_m[w].
\]

The same corollary also gives a constant $C'_{ \M,m}>0$ such that
\[
\|w\|_{H^{2m+2}( \M)}^2
\le
C'_{ \M,m}\sum_{k=0}^{m+1}\|\Delta^k w\|_{L^2( \M)}^2.
\]
Using \eqref{eq:lap_lap_control} iteratively to absorb the lower-order terms into $\|\Delta^{m+1}w\|_{L^2}^2 = \mathcal D_m[w]$, we obtain
\[
\|w\|_{H^{2m+2}( \M)}^2
\le
C'_{ \M,m}\mathcal D_m[w].
\]

Taking $k=m$ in \eqref{eq:grad_lap_control} gives
\[
\|\nabla \Delta^m w\|_{L^2( \M)}^2
\le
\frac{1}{ \lambda_1}\mathcal D_m[w].
\]
Hence
\[
2\mathcal E_m[w]
\le
\left(\frac{1}{ \lambda_1}+b^2\right)\mathcal D_m[w].
\]
If $0 < b \le \lambda_1^{-1/2}$, then
\[
\frac{1}{ \lambda_1}+b^2\le \frac{2}{ \lambda_1},
\]
and therefore $\mathcal D_m[w] \ge  \lambda_1\mathcal E_m[w]$.
\end{proof}

We will also use a Gagliardo-Nirenberg interpolation on compact manifolds, stated below. The proof of this result, which follows directly from the one on $\mathbb R^d$ in \cite[(86)]{brezis2011functional} applied to each chart of the atlas on $\M$, is omitted. 
\begin{lemma}[$L^1$--$L^\infty$ interpolation]\label{lem:l1_to_linf}
Let $\M$ be a compact Riemannian manifold of dimension $d$. There exists $C_\M>0$ such that every $f\in W^{1,\infty}(\M)$ satisfies
\[
\|f\|_{L^\infty(\M)}
\le
C_\M
\|f\|_{W^{1,\infty}(\M)}^{\frac{d}{d+1}}
\|f\|_{L^1(\M)}^{\frac{1}{d+1}}.
\]
\end{lemma}

\subsection{Proof of \Cref{thm:main_decay}}
\label{sec:proofintermediate}

Since $\mu_b(0)=\mu_{\mathrm{in}}$ by \Cref{assump:initial}, we have $u_b(0)=K_b\mu_{\mathrm{in}}$. In  a Laplacian eigenbasis, $K_b$ is the multiplier $\bigl(1+b^2\lambda_\ell\bigr)^{-1}$, so every derivative of $u_b(0)$ is bounded uniformly by the corresponding derivative of $\mu_{\mathrm{in}}$. Hence $E_0 < \infty$. Set
\[
c_0 := \inf_{x\in  \M}\mu_{\mathrm{in}}(x) > 0.
\]

We also claim that
\begin{equation}\label{eq:initial_positive_u}
u_b(0,x)\ge c_0
\qquad\text{for all }x\in  \M.
\end{equation}
Indeed,
\[
u_b(0)=K_b\mu_{\mathrm{in}}=K_b(\mu_{\mathrm{in}}-c_0)+c_0 K_b1.
\]
Since $\mu_{\mathrm{in}}-c_0\ge0$, \Cref{lem:Kb_positive} gives
$K_b(\mu_{\mathrm{in}}-c_0)\ge0$, while $K_b1=1$. Hence
\[
u_b(0)\ge c_0,
\]
which is \eqref{eq:initial_positive_u}.

\medskip
\noindent\emph{Step 1: Bootstrap assumption.} We work on the maximal interval on which both the energy and a rough positive lower bound remain controlled.

Let $T_\ast$ be the supremum of all $T>0$ such that
\begin{equation}\label{eq:bootstrap_interval}
\mathcal E_m(t)\le 2E_0,
\qquad
u_b(t,x)\ge \frac{c_0}{2}
\qquad\text{for all }(t,x)\in [0,T]\times \M.
\end{equation}
Since $\mathcal E_m(0)\le E_0$ and \eqref{eq:initial_positive_u} holds,
continuity gives $T_\ast>0$.

We work on $[0,T_\ast)$. By \Cref{lem:sobolev_control} and \eqref{eq:bootstrap_interval},
\[
\|u_b(t)-1\|_{H^{2m+1}( \M)}
=
\|v_b(t)\|_{H^{2m+1}( \M)}
\le
C E_0^{1/2}.
\]
Since $2m+1>  d/2+2$, the compact-manifold Sobolev embedding in
\cite[Theorem~6.3]{hebey_robert_2008}, applied with $p=2$, differentiability
order $k=2m+1$, and target order $2$, yields
\begin{equation}\label{eq:bootstrap_w2infty}
\|v_b(t)\|_{W^{2,\infty}( \M)} \le C E_0^{1/2}
\qquad\text{for all } t\in[0,T_\ast).
\end{equation}

\medskip
\noindent\emph{Step 2: Energy dissipation.} Under the bootstrap lower bound, we differentiate the energy and estimate separately the principal term and the commutator remainder.

Since $v_b = u_b - 1$, the equation
\[
\partial_t \mu_b = \operatorname{div}(\mu_b \nabla \log u_b),
\qquad
\mu_b = u_b - b^2 \Delta u_b,
\]
rewrites as
\begin{equation}\label{eq:vb_equation}
\partial_t v_b - b^2 \Delta \partial_t v_b
=
\Delta v_b
-
b^2 \operatorname{div}\bigl((\Delta v_b)\nabla \log u_b\bigr).
\end{equation}

Apply $\Delta^m$ to \eqref{eq:vb_equation} and test against $-\Delta^{m+1}v_b$. Since $\Delta$ is self-adjoint on  $\M$, we obtain
\[
\frac{d}{dt}\mathcal E_m(t) + \mathcal D_m(t)
=
b^2
\int_{ \M}
\Delta^m \operatorname{div}\bigl((\Delta v_b)\nabla \log u_b\bigr)
\Delta^{m+1} v_b\, dx.
\]
We split the right-hand side into a principal part and a remainder:
\[
\frac{d}{dt}\mathcal E_m(t) + \mathcal D_m(t)
=
I_{\mathrm{pr}}(t) + I_{\mathrm{rem}}(t),
\]
with
\[
I_{\mathrm{pr}}
:=
b^2
\int_{ \M}
\operatorname{div}\bigl((\Delta^{m+1} v_b)\nabla \log u_b\bigr)
\Delta^{m+1} v_b\, dx
\]
and
\[
I_{\mathrm{rem}}
:=
b^2
\int_{ \M}
\mathcal R_m[u_b]\,\Delta^{m+1} v_b\, dx,
\]
where $\mathcal R_m[u_b]$ is the commutator analyzed in \Cref{cor:compact_remainder}:
\[
\mathcal R_m[u]
:=
\Delta^m \divg\bigl((\Delta u)\nabla \log u\bigr)
-
\divg\bigl((\Delta^{m+1}u)\nabla \log u\bigr).
\]

For the principal part, an integration by parts gives 
\[
I_{\mathrm{pr}}
=
\frac{b^2}{2}
\int_{ \M}
(\Delta^{m+1} v_b)^2 \Delta \log u_b\, dx.
\]
Since $\mu_b = u_b - b^2 \Delta u_b > 0$ by \Cref{thm:pde_global_smooth_existence},
\[
b^2 \Delta u_b < u_b.
\]
Moreover,
\[
\Delta \log u_b
=
\frac{\Delta u_b}{u_b} - |\nabla \log u_b|^2
\le
\frac{\Delta u_b}{u_b}.
\]
Therefore
\[
b^2 \Delta \log u_b \le b^2\frac{\Delta u_b}{u_b} \le 1
\qquad
\text{pointwise on } [0,\infty)\times  \M,
\]
and hence
\[
I_{\mathrm{pr}}(t) \le \frac12 \mathcal D_m(t).
\]

To estimate the remainder, we apply \Cref{cor:compact_remainder} with
$u=u_b(t)$. Since $u_b(t)$ is a probability density, $v_b(t)=u_b(t)-1$ has zero
mean, so \Cref{lem:sobolev_control} applies. Together with
\eqref{eq:bootstrap_w2infty} and the bootstrap lower bound $u_b\ge c_0/2$, we obtain
\[
\|\mathcal R_m[u_b(t)]\|_{L^2( \M)}
\le
C( \M,c_0,E_0)\mathcal D_m(t)^{1/2},
\]
and therefore
\[
|I_{\mathrm{rem}}(t)|
\le
b^2 \|\mathcal R_m[u_b(t)]\|_{L^2}\|\Delta^{m+1}v_b(t)\|_{L^2}
\le
C( \M,c_0,E_0) b^2 \mathcal D_m(t).
\]
Choose $b_2>0$ so that
\[
C(c_0,E_0) b_2^2 \le \frac14.
\]
Let
\[
\bar b_0
:=
\min\left\{
1,\,
\frac{1}{\sqrt{ \lambda_1}},\,
b_2
\right\}.
\]
Then, for every $0 < b \le \bar b_0$,
\[
\frac{d}{dt}\mathcal E_m(t) + \frac14 \mathcal D_m(t) \le 0.
\]
Applying \Cref{lem:sobolev_control} to $v_b(t)$, we obtain
\[
\mathcal D_m(t) \ge  \lambda_1\mathcal E_m(t),
\]
and therefore
\[
\frac{d}{dt}\mathcal E_m(t)
\le
-\frac{ \lambda_1}{4}\mathcal E_m(t).
\]
Gronwall's lemma gives
\[
\mathcal E_m(t) \le E_0 e^{-\frac{ \lambda_1}{4}t}.
\]

\medskip
\noindent\emph{Step 3: Consequences under the bootstrap.} The energy decay yields exponential Sobolev and logarithmic bounds on $[0,T_\ast)$.

By \Cref{lem:sobolev_control},
\[
\|u_b(t)-1\|_{H^{2m+1}( \M)}
=
\|v_b(t)\|_{H^{2m+1}( \M)}
\le
C \mathcal E_m(t)^{1/2}
\le
C e^{-\frac{ \lambda_1}{8}t}.
\]
By the definition of $m$, we have $2m+1 >  d/2 + 2$. Applying
\cite[Theorem~6.3]{hebey_robert_2008} again with $p=2$,
differentiability order $k=2m+1$, and target order $2$, we obtain
$H^{2m+1}( \M) \hookrightarrow W^{2,\infty}( \M)$, and so
\[
\|u_b(t)-1\|_{W^{2,\infty}( \M)}
\le
C e^{-\frac{ \lambda_1}{8}t}.
\]
Using the bootstrap lower bound $u_b\ge c_0/2$, we estimate
\[
\nabla \log u_b = \frac{\nabla u_b}{u_b},
\qquad
\nabla^2 \log u_b
=
\frac{\nabla^2 u_b}{u_b}
-
\frac{\nabla u_b \otimes \nabla u_b}{u_b^2}.
\]
The linear terms are bounded by $C_{c_0}\|u_b-1\|_{W^{2,\infty}}$, while the quadratic term decays faster, namely like $\|u_b-1\|_{W^{1,\infty}}^2$. Consequently,
\[
\|\nabla \log u_b(t)\|_{W^{1,\infty}( \M)}
\le
C e^{-\frac{ \lambda_1}{8}t}.
\]

\medskip
\noindent\emph{Step 4: Lower bound on $u_b$.} We now compare $u_b$ with the heat flow and use the direct $b^2$ estimate to improve the bootstrap lower bound.

Let $\mu$ be the heat flow starting from $\mu_{\mathrm{in}}$. By \Cref{lem:heat_flow_bounds},
\begin{equation}\label{eq:heat_lower_bound}
\mu(t,x)\ge c_0
\qquad
\text{for all } (t,x)\in [0,\infty)\times  \M.
\end{equation}
Applying \Cref{prop:w1_b2} on $[0,T_\ast)$ with $c_\ast=c_0/2$ and the
$H^{2m+1}$ decay just proved, we obtain
\begin{equation}\label{eq:l1_u_mu}
\|u_b(t)-\mu(t)\|_{L^1( \M)}
\le
\|\mu_b(t)-\mu(t)\|_{L^1( \M)}
+
b^2\|\Delta u_b(t)\|_{L^1( \M)}
\le
Cb^2.
\end{equation}
for all $t\in[0,T_\ast)$.
Moreover, \eqref{eq:bootstrap_w2infty} and \Cref{lem:heat_flow_bounds} give
\[
\|u_b(t)-\mu(t)\|_{W^{1,\infty}( \M)}
\le
\|v_b(t)\|_{W^{2,\infty}( \M)}
+
\|\mu(t)-1\|_{W^{1,\infty}( \M)}
\le
CE_0^{1/2}+C( \M,\mu_{\mathrm{in}}).
\]
Together with \Cref{lem:l1_to_linf}, this yields
\begin{equation}\label{eq:linfty_u_mu}
\|u_b(t)-\mu(t)\|_{L^\infty( \M)}
\le
C b^{\frac{2}{ d+1}}
\qquad
\text{for all } t\in[0,T_\ast).
\end{equation}
Choose $b_3>0$ so small that
\[
C b_3^{\frac{2}{ d+1}} \le \frac{c_0}{4}.
\]
Then, for every $0<b\le b_3$, \eqref{eq:heat_lower_bound} and \eqref{eq:linfty_u_mu} imply
\[
u_b(t,x)\ge \frac{3}{4}c_0
\qquad
\text{for all } (t,x)\in[0,T_\ast)\times  \M.
\]

\medskip
\noindent\emph{Step 5: Closing the bootstrap.} Set
\[
b_0
:=
\min\left\{
\bar b_0,\,
b_3
\right\}.
\]
For every $0<b\le b_0$, we have
\[
\mathcal E_m(t)\le E_0 e^{-\frac{ \lambda_1}{4}t}<2E_0
\qquad\text{for all }t\in[0,T_\ast),
\]
and
\[
u_b(t,x)\ge \frac{3}{4}c_0>\frac{c_0}{2}
\qquad\text{for all }(t,x)\in[0,T_\ast)\times \M.
\]
By continuity, this extends \eqref{eq:bootstrap_interval} beyond $T_\ast$
unless $T_\ast=\infty$. Hence $T_\ast=\infty$, and the estimates proved above on
$[0,T_\ast)$ hold for all $t\ge0$. This proves the claim.\qed

\begin{remark}
    Note that the left-hand side in \eqref{eq:vb_equation} has the same Voigt-type regularization of the time derivative as in the Euler--Voigt and Navier--Stokes--Voigt equations \cite{larios2009higher}.
\end{remark}

\subsection{Sharpness of $b^2$ rate}\label{s:lowerbound}

The rate in \Cref{thm:main_convergence} is optimal in general. Indeed, take
\(\M=\mathbb S^1\simeq\R/2\pi\mathbb Z\) and
\[
\mu_{\mathrm{in}}(\theta)=1+\varepsilon\cos\theta,
\qquad 0<\varepsilon<\frac12 .
\]
Let \(\mu_b\) solve \eqref{eq:pde} and let \(\mu\) solve \eqref{eq:heat} with
the same initial datum. Then there exist \(t_0>0\), \(b_1>0\), and \(c>0\) such
that, for every \(0<b\le b_1\),
\[
\|\mu_b(t_0)-\mu(t_0)\|_{L^1(\mathbb S^1)}\ge c b^2 .
\]

To see this, set \(\nu_b:=\mu_b-\mu\). From the comparison equation \eqref{eq:comparison_error_eq} on $\mathbb{S}^1$ we have
\[
\partial_t\nu_b-\partial_{\theta\theta}\nu_b
=
b^2\partial_\theta F_b,
\qquad
F_b:=\partial_{\theta\theta\theta}u_b
-
(\partial_{\theta\theta}u_b)\partial_\theta\log u_b .
\]
Since \(\nu_b(0)=0\),
\[
\partial_t\nu_b(0)=b^2\partial_\theta F_b(0).
\]
Moreover,
\[
u_b(0,\theta)=K_b\mu_{\mathrm{in}}
=
1+a_b\cos\theta,
\qquad
a_b:=\frac{\varepsilon}{1+b^2}.
\]
Therefore
\[
F_b(0,\theta)
=
\frac{a_b\sin\theta}{1+a_b\cos\theta},
\]
and hence
\[
G_b(\theta):=\partial_\theta F_b(0,\theta)
=
\frac{a_b(\cos\theta+a_b)}{(1+a_b\cos\theta)^2}.
\]
As \(b\downarrow0\), \(G_b\to G_0\) smoothly, where
\[
G_0(\theta)
=
\frac{\varepsilon(\cos\theta+\varepsilon)}
{(1+\varepsilon\cos\theta)^2}.
\]
In particular \(G_0\not\equiv0\).

Now test against \(G_0\). Define
\[
J_b(t):=\int_{\mathbb S^1}\nu_b(t,\theta)G_0(\theta)\,d\theta .
\]
Then
\[
J_b'(0)
=
b^2\int_{\mathbb S^1}G_bG_0\,d\theta
\ge
\frac12 b^2\|G_0\|_{L^2}^2
\]
for all sufficiently small \(b\). By Taylor expansion at \(t=0\), 
with the derivative bounds from \Cref{thm:main_decay} ensuring that the remainder is uniform for sufficiently small $b$,
\[
J_b(t)=tJ_b'(0)+O(t^2b^2).
\]
Choosing \(t_0>0\) small but fixed gives
\[
|J_b(t_0)|\ge c b^2 .
\]
Finally,
\[
\|\mu_b(t_0)-\mu(t_0)\|_{L^1}
=
\|\nu_b(t_0)\|_{L^1}
\ge
\frac{|J_b(t_0)|}{\|G_0\|_{L^\infty}}
\ge
c b^2 .
\]
Thus the \(b^2\) rate cannot be improved in general.

\section{Proof of the one-dimensional mean-field limit}\label{sec:s1}

We now compare the regularized continuum flow with its deterministic particle
approximation. We work in the case  $d=1$ where the kernel has only a cusp
singularity. In this setting the particle dynamics and the
modulated energy estimate can be treated without truncating the kernel.   Every one-dimensional closed connected Riemannian manifold is isometric to a circle. After normalizing the total length, we may without loss identify $\M$ with
$\Torus:=\R/2\pi\mathbb Z$, denote by $d\theta$ the normalized Lebesgue
measure on $\Torus$, and write
\[
\Delta=\partial_{\theta\theta},
\qquad
f'=\partial_\theta f.
\]

For an initial configuration $X^0=(x_1^0,\dots,x_N^0)\in\Torus^N$ with
distinct entries, set
\[
\mu^N_b(t)=\frac1N\sum_{i=1}^N\delta_{x_i(t)},
\qquad
u^N_b(t)=K_b*\mu^N_b(t),
\qquad
\Kb=(\I-b^2\partial_{\theta\theta})^{-1},
\]
where $*$ denotes the circular convolution ($K_b*\mu^N_b(\theta):=\int_{\Torus}K_b(\theta-\varphi)\,\mathrm d\mu_b^N(\varphi)$).
We consider the particle system
\begin{equation}\label{eq:s1_particle_system}
\dot x_i(t)=-\frac{(u^N_b)'(t,x_i(t))}{u^N_b(t,x_i(t))},
\qquad
x_i(0)=x_i^0,
\qquad
i=1,\dots,N,
\end{equation}
where the derivative at a particle is evaluated by suppressing its
self-interaction:
\[
(u^N_b)'(t,x_i(t))
=
\frac1N\sum_{j\ne i}K_b'\bigl(x_i(t)-x_j(t)\bigr).
\]
This is the natural value at the cusp of the even kernel $K_b$: 
it is the arithmetic mean of the two one-sided limits of $(u^N_b)'$ at the
particle, and corresponds to the convention $K_b'(0)=0$. The noncollision result \Cref{lem:s1_noncollision} below keeps trajectories in the collision-free configuration space, and
\Cref{prop:s1_particle_wellposed} gives global well-posedness for distinct
initial positions. In the mean-field theorem, these initial positions are
sampled independently with law $\mu_{\mathrm{in}}(\theta)\,d\theta$, so this assumption holds almost surely.

\subsection{Well-posedness of the particle dynamics}

To prove the well-posedness of the particles dynamics, we first show that particles cannot collide in finite time.

\begin{lemma}[Noncollision on $\Sphere^1$]\label{lem:s1_noncollision}
Let $0<b\le1$, let $T\in(0,\infty]$, and let
$x_1,\dots,x_N\in C^1([0,T);\Torus)$ solve the particle system
\eqref{eq:s1_particle_system}. If
\[
x_i(0)\ne x_j(0)\qquad\text{in }\Torus,\quad i\ne j,
\]
then the particles remain distinct. More precisely, if $N\ge2$ and
\[
\delta(t):=\min_{i\ne j}d_{\Torus}(x_i(t),x_j(t)),
\]
then there exists $\eta_{b,N}>0$, depending only on $b$ and $N$, such that
\[
\delta(t)\ge \min\{\delta(0),\eta_{b,N}\}
\qquad\text{for every }0\le t<T.
\]
\end{lemma}

\begin{proof}
The case $N=1$ is immediate. Assume $N\ge2$. On any collision-free time interval
choose cyclic lifts $y_1<\dots<y_N<y_1+2\pi$, with $y_{N+1}=y_1+2\pi$
and write
\[
\delta(t)=\min_{1\le i\le N}(y_{i+1}-y_i).
\]
If $y_{i+1}-y_i=\delta(t)$, then the particle ODE gives
\[
\frac{\mathrm d}{\mathrm dt}(y_{i+1}-y_i)
=
-\frac{\sum_{j\ne i+1} K_b'(y_{i+1}-y_j)}
        {\sum_{j=1}^N K_b(y_{i+1}-y_j)}
+
\frac{\sum_{j\ne i} K_b'(y_i-y_j)}
        {\sum_{j=1}^N K_b(y_i-y_j)}.
\]
Moreover,
\[
-\lim_{\theta\downarrow0}K_b'(\theta)
=
\lim_{\theta\uparrow2\pi}K_b'(\theta)
=
\frac{\pi}{b^2}>0.
\]
Hence there are $\eta_{b,N},c_{b,N}>0$ such that
\[
y_{i+1}-y_i=\delta(t)\le\eta_{b,N}
\quad\Longrightarrow\quad
\frac{\mathrm d}{\mathrm dt}(y_{i+1}-y_i)\ge c_{b,N}.
\]
Indeed, as $y_{i+1}-y_i=\delta(t)\to0$, after making a cyclic cut that does
not split the limiting cluster, the two denominators converge to the same
positive limit $D_\ast$, while the non-mutual terms in the two numerators
converge to the same limit $A$. Therefore,
\[
\frac{\mathrm d}{\mathrm dt}(y_{i+1}-y_i)
\longrightarrow
-\frac{-\pi/b^2+A}{D_\ast}
+\frac{\pi/b^2+A}{D_\ast}
=
\frac{2\pi}{b^2D_\ast}>0.
\]

Set $m:=\min\{\delta(0),\eta_{b,N}\}$. Whenever $\delta(t)\le\eta_{b,N}$, every gap realizing the minimum has
derivative at least $c_{b,N}$. Hence the right derivative of $\delta$ is at
least $c_{b,N}$, and $\delta$ cannot cross downward through $m$. It follows
that
\[
\delta(t)\ge m
\qquad\text{for every }0\le t<T,
\]
which proves the claim.
\end{proof}

\begin{remark}
Since the initial law $\mu_{\mathrm{in}}(\theta)\,d\theta$ is nonatomic, the i.i.d.\
initial particles in \Cref{thm:main_s1} are distinct almost surely.
Thus \Cref{lem:s1_noncollision} rules out later overlaps almost surely.
\end{remark}

With noncollision in hand, global well-posedness follows by the classical ODE continuation theorem:

\begin{corollary}[Particle well-posedness on $\Sphere^1$]\label{prop:s1_particle_wellposed}
Let $0<b\le1$ and assume that $x_i^0\ne x_j^0$ in $\Torus$ for $i\ne j$.
Then the particle system \eqref{eq:s1_particle_system} has a unique global
$C^1$ solution. Along this solution the particles remain distinct,
$u^N>0$, the maps $t\mapsto u^N(t,\cdot)$ and $t\mapsto w(t,\cdot)$ are
continuous on compact time intervals, $t\mapsto E^N_b(t)$ is of class $C^1$,
and \eqref{eq:s1_continuity_mu_n} holds in the sense of distributions.
\end{corollary}

\subsection{Modulated Energy estimate}

We compare the dynamics of $\mu^N_b$ with 
the continuum solution $\mu_b$ of the regularized PDE \eqref{eq:pde} from the previous
sections, started from the same density $\mu_{\mathrm{in}}$. We write
\begin{equation}
\label{eq:s1_continuity_mu_b}
    u_b:=K_b*\mu_b,
\qquad
\partial_t\mu_b=\partial_\theta\!\bigl(\mu_b(\log u_b)'\bigr),
\end{equation}
For the global particle solution constructed above, the empirical measure
satisfies
\begin{equation}\label{eq:s1_continuity_mu_n}
\partial_t\mu^N_b=\partial_\theta\!\bigl(\mu^N_b(\log u^N_b)'\bigr),
\end{equation}
in the sense of distributions, and
\begin{equation}
\label{eq:s1_resolvent_identities}
\mu^N_b=u^N_b-b^2(u^N_b)'',\qquad \qquad
\mu_b=u_b-b^2u_b''.
\end{equation}
For the comparison, set
\begin{equ}\label{eq:w}
w:=u^N_b-u_b = K_b*(\mu^N_b-\mu_b).
\end{equ}
 Then, the natural modulated energy is the associated
kernel quadratic form
\begin{equation}\label{eq:s1_modulated_energy}
E^N_b(t)
:=
\langle \mu^N_b-\mu_b,w\rangle
=
\|w(t)\|_{L^2(\Torus)}^2+b^2\|w'(t)\|_{L^2(\Torus)}^2.
\end{equation}
This quantity is finite for empirical measures and will later be converted
into a $W_1$ bound by duality. Throughout this section, $W_1$ denotes the
$1$-Wasserstein distance on $\Torus$ associated with the geodesic distance.
  We now compare the particle and continuum evolutions through a weak--strong modulated energy estimate: the possible exponential growth is bounded only by the continuum solution, whose coefficients are controlled by the decay results of \Cref{sec:kernel_and_energy}.

\begin{proposition}[Modulated energy estimate]
\label{prop:s1_energy}
Under \Cref{assump:initial} and $\M=\mathbb S^1$, suppose that the
initial particle positions are distinct, that is,
\[
x_i(0)\neq x_j(0)
\qquad\text{for all }i\neq j,
\]
and let $\mu_b^N$ be the global particle solution given by
\Cref{prop:s1_particle_wellposed}.
Then there exist $b_0\in(0,1]$ and $C>0$, depending only on
$\mu_{\mathrm{in}}$, such that, for every $0<b\le b_0$, every $N\ge1$, and
every $t\ge0$,
\begin{equation}
\label{eq:s1_energy_gronwall}
E_b^N(t)\le C E_b^N(0).
\end{equation}
The constant $C$ is uniform in $b$, $N$, $t$, and the initial particle configuration.
\end{proposition}

This result, combined with an estimate on the expected value of the modulated energy at initialization, is sufficient to prove our second main result \Cref{thm:main_s1}:

\begin{proof}[Proof of Theorem~\ref{thm:main_s1}]

Let \(\phi\) be smooth and \(1\)-Lipschitz. After subtracting its
mean, Poincare's inequality gives
\[
\|\phi\|_{L^2(\Torus)}\le \|\phi'\|_{L^2(\Torus)}\le1 .
\]
Since
$
\mu^N_b-\mu_b=w-b^2w''$
in distribution,
\[
\begin{aligned}
\left|\int_{\Torus}\phi\,d(\mu^N_b-\mu_b)\right|
&=
\left|
\int_{\Torus}\phi w\,d\theta
+
b^2\int_{\Torus}\phi'w'\,d\theta
\right|
\\
&\le
\bigl(\|\phi\|_{L^2}^2+b^2\|\phi'\|_{L^2}^2\bigr)^{1/2}
(E^N_b)^{1/2}
\\
&\le
\sqrt{1+b^2}\,(E^N_b)^{1/2}.
\end{aligned}
\]
Approximating Lipschitz test functions by smooth ones and using the
Kantorovich--Rubinstein duality gives
\[
W_1(\mu^N_b(t),\mu_b(t))
\le
\sqrt{1+b^2}\,E^N_b(t)^{1/2}.
\]
Combining this bound with \Cref{prop:s1_energy} (the initial law is nonatomic, so the initial particles are distinct almost
surely), taking the supremum in time, then using Jensen's inequality, 
we obtain
\begin{equation}
\label{eq:w1_energy_bound}
\mathbb E\Bigl[\sup_{t\ge0}W_1(\mu^N_b(t),\mu_b(t))\Bigr]
\le
C\,\mathbb E\bigl[E^N_b(0)^{1/2}\bigr]
\le
C\,\mathbb E[E^N_b(0)]^{1/2}.
\end{equation}

It remains to estimate the initial energy. Setting
\[
\nu_N:=\mu_b^N(0)-\mu_{\mathrm{in}},
\]
since \(\nu_N(\Torus)=0\), Parseval's identity and the Fourier representation
of \(K_b\) give
\[
\mathbb E\bigl[E_b^N(0)\bigr]
=
\sum_{k\neq0}
\frac{\mathbb E\bigl[|\widehat{\nu_N}(k)|^2\bigr]}
     {1+b^2k^2}.
\]
For every \(k\neq0\),
\[
\widehat{\nu_N}(k)
=
\frac1N\sum_{j=1}^N
\left(e^{-ikx_j(0)}-\widehat{\mu_{\mathrm{in}}}(k)\right).
\]
The summands are independent and centered, so the cross terms vanish and
\[
\begin{aligned}
\mathbb E\bigl[|\widehat{\nu_N}(k)|^2\bigr]
&=
\frac1{N^2}\sum_{i,j=1}^N
\mathbb E\left[
\left(e^{-ikx_i(0)}-\widehat{\mu_{\mathrm{in}}}(k)\right)
\overline{
\left(e^{-ikx_j(0)}-\widehat{\mu_{\mathrm{in}}}(k)\right)}
\right]
\\
&=
\frac1{N^2}\sum_{j=1}^N
\mathbb E\left[
\left|e^{-ikx_j(0)}-\widehat{\mu_{\mathrm{in}}}(k)\right|^2
\right]
\\
&=
\frac1N\left(1-|\widehat{\mu_{\mathrm{in}}}(k)|^2\right).
\end{aligned}
\]
Consequently, for $0 < b \leq 1$
\[
\mathbb E\bigl[E_b^N(0)\bigr]
=
\frac1N\sum_{k\neq0}
\frac{1-|\widehat{\mu_{\mathrm{in}}}(k)|^2}
     {1+b^2k^2} \le
\frac1N\sum_{k\neq0}\frac1{1+b^2k^2} \leq C \frac 1{Nb}.
\]
Substituting into \eqref{eq:w1_energy_bound} yields the claim.
\end{proof}

We first give a sketch of the proof of Proposition~\ref{prop:s1_energy} for a smooth version of the modulated energy estimate to present the main ideas. In the following computation, all integrals are taken over $\Torus$ with respect to the normalized Lebesgue measure $\mathrm d\theta$. Let
\(\mu_\alpha\), \(\alpha=1,2\), be two smooth positive solutions of \eqref{eq:pde}:
\[
\partial_t\mu_\alpha
=
\partial_\theta\bigl(\mu_\alpha(\log u_\alpha)'\bigr),
\qquad
u_\alpha:=\Kb\mu_\alpha,
\qquad
\mu_\alpha=u_\alpha-b^2(u_\alpha)'',
\qquad \alpha=1,2.
\]
Set \(w:=u_1-u_2=\Kb(\mu_1-\mu_2)\). The modulated energy is
\[
E(t)
:=
\langle \mu_1-\mu_2,\Kb(\mu_1-\mu_2)\rangle
=
\int_{\Torus}\bigl(w^2+b^2|w'|^2\bigr).
\]
Differentiating \(E\), and using the self-adjointness of \(\Kb\), gives
\[
\begin{aligned}
\frac12\frac{d}{dt}E
&=
-\int
\left[
\mu_1(\log u_1)'-\mu_2(\log u_2)'
\right]w' \\
&=
\underbrace{
-\int \frac{\mu_1}{u_1}w'
\bigl(w'-(\log u_2)'w\bigr)
}_{\text{(I)}}
\underbrace{
-\int(\log u_2)'w'(w-b^2w'')
}_{\text{(II)}}.
\end{aligned}
\]
where in the last line we added and subtracted the reference velocity \((\log u_2)'\), and used the identities
\[
(\log u_1)'-(\log u_2)'=
\frac{w'-(\log u_2)'w}{u_1},
\qquad
\mu_1-\mu_2=w-b^2w''.
\]
The first term (I), which is the new term with respect to the
standard modulated energy argument for linear Coulomb/Riesz interactions
\cite{serfaty2017mean,duerinckx_serfaty_2020}, can be estimated by a
pointwise Young's inequality
\[
-w'\bigl(w'-(\log u_2)'w\bigr)
=-|w'|^2+(\log u_2)'ww'
\le \frac12 |(\log u_2)'|^2w^2,
\]
and the positivity of \(\mu_1/u_1\) gives
\[
-\int \frac{\mu_1}{u_1}w'
\bigl(w'-(\log u_2)'w\bigr)
\le
\frac12\|(\log u_2)'\|_{L^\infty}^2
\int\frac{\mu_1}{u_1}w^2 .
\]
This last integral is controlled by \(E\):  using \(\mu_1=u_1-b^2(u_1)''\), integrating by parts on
\(\Torus\) and completing the square,
\[
\begin{aligned}
\int\frac{\mu_1}{u_1}w^2
=
\int w^2+b^2\int (u_1)'
\left(\frac{w^2}{u_1}\right)'
=
E
-
b^2\int
\left|w'-(\log u_1)'w\right|^2
\le E.
\end{aligned}
\]
For the term (II), which is the classical stress-tensor term
in
\cite{serfaty2017mean,duerinckx_serfaty_2020}, integration by parts on \(\Torus\) gives directly
\[
-\int(\log u_2)'w'(w-b^2w'')
=
\frac12\int(\log u_2)''\bigl(w^2-b^2|w'|^2\bigr)
\le
\frac12\|(\log u_2)''\|_{L^\infty}E.
\]
Combining the two estimates gives the formal stability bound
\begin{equation}
\label{eq:mod_energy_ineq}
\frac{d}{dt}E(t)
\le
\left(
\|(\log u_2)'\|_{L^\infty}^2
+
\|(\log u_2)''\|_{L^\infty}
\right)E(t).
\end{equation}
Thus, the growth of the modulated energy is controlled only by the reference solution $\mu_2$, and sufficiently fast decay of $\|(\log u_2)''\|_{L^\infty}$ (uniformly in $b$) concludes the proof by Grönwall. This asymmetry in \eqref{eq:mod_energy_ineq} is what will permit both weak–strong stability and the uniform-in-time mean-field limit. 

The particle proof below is the singular analogue of the one-dimensional calculation above, where we identify 
$$\mu_1 = \mu_b^N \qquad \text{and} \qquad \mu_2 = \mu_b$$
and where, as a consequence, the integration by parts involving $\mu_1$ has to be justified across the cusps of $K_b * \mu^N_b$. The necessary decay of the resulting Grönwall rate 
\begin{equ}\label{eq:abt}
A_b(t)
:=
\|(\log u_b)'(t)\|_{L^\infty(\Torus)}^2
+
\|(\log u_b)''(t)\|_{L^\infty(\Torus)}.
\end{equ}
results from the energy estimates of \Cref{thm:main_decay}, yielding the following lemma.

\begin{lemma}\label{lem:s1_input}
Assume \Cref{assump:initial} and $d=1$. Then there
exist $b_\ast\in(0,1]$ and $C>0$, depending only on $\mu_{\mathrm{in}}$, such that for
every $0<b\le b_\ast$
\[
\int_0^\infty A_b(t)\,dt\le C.
\]
\end{lemma}

\begin{proof}
By \Cref{thm:main_decay} with $d=1$, after possibly reducing $b_\ast$,
\[
\|(\log u_b)'(t)\|_{W^{1,\infty}(\Torus)}
\le C e^{-t/8}
\qquad\text{for all }t\ge0.
\]
The stated integral bound follows immediately, after increasing $C$.
\end{proof}

\begin{remark}
    We stress again that the technical nature of the truncation estimates required to deal with more singular kernels are the only obstruction to extending this proof approach to higher dimensions (see \Cref{rem:truncation}): indeed the preceding argument has the same form on closed manifolds in arbitrary dimension, with $(\log u_2)''$ replaced by $\nabla^2 \log u_2$. 
\end{remark} 

We now proceed to carry out the proof of the main result of this section in full generality.

  \begin{proof}[Proof of Proposition~\ref{prop:s1_energy}]
All derivatives of \(u^N_b\) and \(w=u^N_b-u_b\) at particle locations are
understood with the no-self-interaction convention.

We first differentiate the finite-dimensional expression for the energy:
\begin{equation}
\label{eq:s1_energy_finite_dim}
E^N_b
=
\frac1{N^2}\sum_{i,j=1}^N K_b(x_i-x_j)
-\frac2N\sum_{i=1}^N u_b(x_i)
+\int_{\Torus}u_b\,d\mu_b .
\end{equation}
By \Cref{lem:s1_noncollision}, the particles remain distinct, so for
$i\neq j$ the difference $x_i(t)-x_j(t)$ never reaches the cusp of $K_b$ at
the origin. Hence all off-diagonal terms are differentiable in time, while
the diagonal terms are constant.
Using the oddness of
\(K_b'\), the equation for the particles, and
\(\partial_tu_b=K_b*\partial_t\mu_b\), we obtain
\[
\frac12\frac{d}{dt}E^N_b
=
-\frac1N\sum_{i=1}^N w'(x_i)(\log u^N_b)'(x_i)
-\int_{\Torus}w\,\partial_t\mu_b\,d\theta .
\]
The continuum equation \eqref{eq:s1_continuity_mu_b}:
\[
\partial_t\mu_b=\partial_\theta\bigl(\mu_b(\log u_b)'\bigr)
\]
then gives, by ordinary integration by parts against the \(H^1\) function \(w\),
\[
\frac12\frac{d}{dt}E^N_b
=
-\frac1N\sum_{i=1}^N
w'(x_i)\Bigl((\log u^N_b)'(x_i)-(\log u_b)'(x_i)\Bigr)
+T_2 ,
\]
where
\[
T_2
:=
-\frac1N\sum_{i=1}^N(\log u_b)'(x_i)w'(x_i)
+
\int_{\Torus}(\log u_b)'w'\,d\mu_b .
\]

For the first term, since \(u^N_b=u_b+w\),
\[
(\log u^N_b)'(x_i)-(\log u_b)'(x_i)
=
\frac{w'(x_i)-(\log u_b)'(x_i)w(x_i)}{u^N_b(x_i)} .
\]
Therefore
\[
\begin{aligned}
&-\frac1N\sum_i
w'(x_i)\Bigl((\log u^N_b)'(x_i)-(\log u_b)'(x_i)\Bigr)
\\
&\qquad
=
-\frac1N\sum_i\frac{|w'(x_i)|^2}{u^N_b(x_i)}
+
\frac1N\sum_i
\frac{(\log u_b)'(x_i)w(x_i)w'(x_i)}{u^N_b(x_i)}
\\
&\qquad
\le
-\frac1N\sum_i\frac{|w'(x_i)|^2}{u^N_b(x_i)}
+
\frac1{2N}\sum_i\frac{|w'(x_i)|^2}{u^N_b(x_i)}
+
\frac1{2N}\sum_i
\frac{|(\log u_b)'(x_i)|^2w(x_i)^2}{u^N_b(x_i)}
\\
&\qquad
\le
\frac12\|(\log u_b)'\|_{L^\infty(\Torus)}^2
\frac1N\sum_i\frac{w(x_i)^2}{u^N_b(x_i)} .
\end{aligned}
\]
It remains to control the residual particle average by \(E^N_b\). Expanding,
\begin{equation}
\label{eq:s1_T1_expansion}
\frac1N\sum_i\frac{w(x_i)^2}{u^N_b(x_i)}
=
\frac1N\sum_i u^N_b(x_i)
-\frac2N\sum_i u_b(x_i)
+
\frac1N\sum_i\frac{u_b(x_i)^2}{u^N_b(x_i)} .
\end{equation}
Set \(a_i:=u_b(x_i)/u^N_b(x_i)\). Positivity of the kernel quadratic form
applied to
\[
\mu_b-\frac1N\sum_i a_i\delta_{x_i}
\]
gives
\[
0
\le
\int_{\Torus}u_b\,d\mu_b
-\frac2N\sum_i\frac{u_b(x_i)^2}{u^N_b(x_i)}
+
\frac1{N^2}\sum_{i,j}K_b(x_i-x_j)a_i a_j .
\]
Using \(u^N_b(x_i)=N^{-1}\sum_jK_b(x_i-x_j)\), the symmetry of \(K_b\), and \(K_b\ge0\),
\[
\begin{aligned}
&\frac1N\sum_i\frac{u_b(x_i)^2}{u^N_b(x_i)}
-
\frac1{N^2}\sum_{i,j}K_b(x_i-x_j)a_i a_j
\\
&\qquad =
\frac1{N^2}\sum_{i,j}K_b(x_i-x_j)(a_i^2-a_i a_j)
\\
&\qquad =
\frac1{2N^2}\sum_{i,j}K_b(x_i-x_j)(a_i^2+a_j^2-2a_i a_j)
\\
&\qquad =
\frac1{2N^2}\sum_{i,j}K_b(x_i-x_j)(a_i-a_j)^2
\ge0 .
\end{aligned}
\]
Hence
\[
\frac1N\sum_i\frac{u_b(x_i)^2}{u^N_b(x_i)}
\le
\int_{\Torus}u_b\,d\mu_b ,
\]
and consequently, comparing \eqref{eq:s1_T1_expansion} with \eqref{eq:s1_energy_finite_dim},
\begin{align*}
\frac1N\sum_i\frac{w(x_i)^2}{u^N_b(x_i)}
&=
\frac1N\sum_i u^N_b(x_i)
-\frac2N\sum_i u_b(x_i)
+
\frac1N\sum_i\frac{u_b(x_i)^2}{u^N_b(x_i)} \\
\le & \frac1N\sum_i u^N_b(x_i)
-\frac2N\sum_i u_b(x_i)
+
\int_{\Torus}u_b\,d\mu_b = 
E^N_b .
\end{align*}
Therefore
\begin{equation}\label{eq:me-t1}
-\frac1N\sum_i
w'(x_i)\Bigl((\log u^N_b)'(x_i)-(\log u_b)'(x_i)\Bigr)
\le
\frac12\|(\log u_b)'\|_{L^\infty}^2E^N_b .
\end{equation}
It remains to estimate \(T_2\). Choose cyclic lifts
\[
y_1<\cdots<y_N<y_1+2\pi .
\]
and let \(\varepsilon>0\) be smaller than half the minimal particle distance.
By \eqref{eq:w}, on the complement of
\[
\Omega_\varepsilon := \mathbb T \setminus \bigcup_i (y_i-\varepsilon,y_i+\varepsilon)
\subset \text{supp}(\mu_b^N)^c\]
the function \(w\) is smooth and satisfies
\[
w-b^2w''=\mu_b^N-\mu_b= -\mu_b ,
\]
so that
\begin{equation}
    \label{eq:T2}
T_2
=
-\frac1N\sum_i(\log u_b)'(y_i)w'(y_i)
-\lim_{\varepsilon\downarrow0}
\int_{\Omega_\epsilon}
(\log u_b)'w'(w-b^2w'')\,d\theta,
\end{equation}
where we used that  the integral over $\Omega_\varepsilon$ converges to the integral over $\Torus$ because $(\log u_b)'w'\mu_b\in L^\infty(\Torus)$ and $|\Torus\setminus\Omega_\varepsilon|\to0$.
Integrating by parts on the smooth pieces gives: 
\[
-\int_{\Omega_\varepsilon}(\log u_b)'ww'\,\mathrm d\theta
=
\frac12\int_{\Omega_\varepsilon}(\log u_b)''w^2\,\mathrm d\theta 
+\frac12\sum_{i=1}^N
\left[(\log u_b)'w^2\right]_{y_i-\varepsilon}^{y_i+\varepsilon}.
\]
Since $(\log u_b)'w^2$ is continuous at each $y_i$, the boundary sum
vanishes as $\varepsilon\downarrow0$. Hence
\[
-\lim_{\varepsilon\downarrow0}
\int_{\Omega_\varepsilon}(\log u_b)'ww'\,\mathrm d\theta
=
\frac12\int_{\Torus}(\log u_b)''w^2\,\mathrm d\theta.
\]
Similarly, since $(\log u_b)''|w'|^2\in L^\infty(\Torus)$,
\begin{equation}
    \label{eq:boundary_contribution}
\begin{aligned}
b^2\lim_{\varepsilon\downarrow0}\int_{\Omega_\varepsilon}
(\log u_b)'w'w''\,d\theta
&=
-\frac{b^2}{2}\int_{\Torus}(\log u_b)''|w'|^2\,d\theta
\\
&\quad
-\frac{b^2}{2\pi}\lim_{\varepsilon\downarrow0}
\sum_i(\log u_b)'(y_i)
\left(
\frac{|w'(y_i+\varepsilon)|^2}{2}
-
\frac{|w'(y_i-\varepsilon)|^2}{2}
\right).
\end{aligned}
\end{equation}
The cusp of \(K_b\) gives
\[
\lim_{\varepsilon\downarrow0}
\bigl(w'(y_i+\varepsilon)-w'(y_i-\varepsilon)\bigr)
=
-\frac{2\pi}{Nb^2},
\]
while the no-self-interaction convention gives
\[
w'(y_i)
=
\lim_{\varepsilon\downarrow0}
\frac{w'(y_i+\varepsilon)+w'(y_i-\varepsilon)}2 .
\]
Therefore
\[
\lim_{\varepsilon\downarrow0}
\left(
\frac{|w'(y_i+\varepsilon)|^2}{2}
-
\frac{|w'(y_i-\varepsilon)|^2}{2}
\right)
=
-\frac{2\pi}{Nb^2}w'(y_i).
\]
The boundary contribution in \eqref{eq:boundary_contribution} is exactly
\[
\frac1N\sum_i(\log u_b)'(y_i)w'(y_i),
\]
which cancels the particle term in the \(T_2\) equality \eqref{eq:T2}. Hence
\begin{equation} \label{eq:me-t2}
T_2
=
\frac12\int_{\Torus}(\log u_b)''
\bigl(w^2-b^2|w'|^2\bigr)\,d\theta \le
\frac12\|(\log u_b)''\|_{L^\infty(\Torus)}E^N_b .
\end{equation}

Combining \eqref{eq:me-t1} and \eqref{eq:me-t2} yields
\[
\frac12\frac{d}{dt}E^N_b(t)
\le
\frac12
\left(
\|(\log u_b)'\|_{L^\infty}^2
+
\|(\log u_b)''\|_{L^\infty}
\right)E^N_b(t).
\]
Finally, recalling the definition of $A_b(t)$ from \eqref{eq:abt} Gronwall's lemma and \Cref{lem:s1_input} give
\[
E^N_b(t)
\le
E^N_b(0)\exp\!\left(\int_0^tA_b(\tau)\,d\tau\right)
\le
C E^N_b(0)\,,
\]
proving the claim.
\end{proof}

\subsubsection*{Acknowledgements}  AA and GB acknowledge the partial support of the Swiss National Science Foundation project grant 2000-1-243054. FP acknowledges support from the Bergman Fellowship of the American Mathematical Society.
GB gratefully acknowledges Filippo Giovagnini for insightful discussions.
\bibliographystyle{alpha}
\bibliography{bibliography}

@article{brenier2017geometric,
  author = {Brenier, Yann},
  title = {Geometric diffusions of 1-currents},
  journal = {Annales de la Facult{\'e} des sciences de Toulouse: Math{\'e}matiques},
  volume = {26},
  number = {4},
  pages = {831--846},
  year = {2017},
  doi = {10.5802/afst.1554},
}

@article{degond1990deterministic,
  author = {Degond, Pierre and Mustieles, Francisco-Jos{\'e}},
  title = {A deterministic approximation of diffusion equations using particles},
  journal = {SIAM Journal on Scientific and Statistical Computing},
  volume = {11},
  number = {2},
  pages = {293--310},
  year = {1990},
  publisher = {SIAM},
  doi = {10.1137/0911018},
}

@article{lions2001methode,
  author = {Lions, Pierre-Louis and Mas-Gallic, Sylvie},
  title = {Une m{\'e}thode particulaire d{\'e}terministe pour des {\'e}quations diffusives non lin{\'e}aires},
  journal = {Comptes Rendus de l'Acad{\'e}mie des Sciences-Series I-Mathematics},
  volume = {332},
  number = {4},
  pages = {369--376},
  year = {2001},
  publisher = {Elsevier},
}

@article{lacombe1999analyse,
  author = {Lacombe, Gilles},
  title = {Analyse d'une {\'e}quation de vitesse de diffusion},
  journal = {Comptes Rendus de l'Acad{\'e}mie des Sciences-Series I-Mathematics},
  volume = {329},
  number = {5},
  pages = {383--386},
  year = {1999},
  publisher = {Elsevier},
}

@inproceedings{lacombe1999presentation,
  author = {Lacombe, Gilles and Mas-Gallic, Sylvie},
  title = {Presentation and analysis of a diffusion-velocity method},
  booktitle = {ESAIM: Proceedings},
  volume = {7},
  pages = {225--233},
  year = {1999},
  organization = {EDP Sciences},
}

@article{mas2002diffusion,
  author = {Mas-Gallic, Sylvie},
  title = {The diffusion velocity method: a deterministic way of moving the nodes for solving diffusion equations},
  journal = {Transport Theory and Statistical Physics},
  volume = {31},
  number = {4--6},
  pages = {595--605},
  year = {2002},
  publisher = {Taylor \& Francis},
  doi = {10.1081/TT-120015516},
}

@article{geshkovski2025mathematical,
  author = {Geshkovski, Borjan and Letrouit, Cyril and Polyanskiy, Yury and Rigollet, Philippe},
  title = {A mathematical perspective on transformers},
  journal = {Bulletin of the American Mathematical Society},
  volume = {62},
  number = {3},
  pages = {427--479},
  year = {2025},
  doi = {10.1090/bull/1863},
}

@inproceedings{Rig26,
  author = {Philippe Rigollet},
  title = {The Mean-Field Dynamics of Transformers},
  booktitle = {Proceedings of the International Congress of Mathematicians 2026},
  volume = {7},
  pages = {389--404},
  year = {2026},
  publisher = {SIAM},
  doi = {10.1137/25M1826342},
}

@article{carrillo2019blob,
  author = {Carrillo, Jos{\'e} Antonio and Craig, Katy and Patacchini, Francesco S},
  title = {A blob method for diffusion},
  journal = {Calculus of Variations and Partial Differential Equations},
  volume = {58},
  number = {2},
  pages = {53},
  year = {2019},
  publisher = {Springer},
  doi = {10.1007/s00526-019-1486-3},
}

@article{carrillo2024nonlocallinear,
  author = {Carrillo, Jos{\'e} Antonio and Esposito, Antonio and Skrzeczkowski, Jakub and Wu, Jeremy Sheung-Him},
  title = {Nonlocal particle approximation for linear and fast diffusion equations},
  journal = {arXiv preprint arXiv:2408.02345},
  year = {2024},
  doi = {10.48550/arXiv.2408.02345},
}

@article{carrillo2016gradient,
  author = {Carrillo, Jos{\'e} Antonio and Patacchini, Francesco Saverio and Sternberg, Peter and Wolansky, Gershon},
  title = {Convergence of a particle method for diffusive gradient flows in one dimension},
  journal = {SIAM Journal on Mathematical Analysis},
  volume = {48},
  number = {6},
  pages = {3708--3741},
  year = {2016},
  doi = {10.1137/16M1077210},
}

@book{bahouri2011fourier,
  author = {Bahouri, Hajer and Chemin, Jean-Yves and Danchin, Rapha{\"e}l},
  title = {{Fourier} Analysis and Nonlinear Partial Differential Equations},
  year = {2011},
  publisher = {Springer},
  address = {Berlin},
  doi = {10.1007/978-3-642-16830-7},
}

@book{taylor2011pde3,
  author = {Taylor, Michael E.},
  title = {Partial Differential Equations III: Nonlinear Equations},
  edition = {2},
  series = {Applied Mathematical Sciences},
  volume = {117},
  year = {2011},
  publisher = {Springer},
  address = {New York},
  doi = {10.1007/978-1-4419-7049-7},
}

@book{brezis2011functional,
  author = {Brezis, Haim},
  title = {Functional Analysis, {Sobolev} Spaces and Partial Differential Equations},
  year = {2011},
  publisher = {Springer},
  address = {New York},
  doi = {10.1007/978-0-387-70914-7},
}

@book{gilbarg_trudinger_2001,
  author = {Gilbarg, David and Trudinger, Neil S.},
  title = {Elliptic Partial Differential Equations of Second Order},
  edition = {2},
  series = {Classics in Mathematics},
  volume = {224},
  year = {2001},
  publisher = {Springer},
  address = {Berlin},
  doi = {10.1007/978-3-642-61798-0},
}

@misc{xu2004spectral,
  author = {Xu, Bin},
  title = {Derivatives of spectral function and {Sobolev} norms of eigenfunctions on a closed {Riemannian} manifold},
  pages = {60--77},
  year = {2004},
  howpublished = {RIMS K{\=o}ky{\=u}roku 1389},
}

@incollection{hebey_robert_2008,
  author = {Hebey, Emmanuel and Robert, Fr{\'e}d{\'e}ric},
  editor = {Krupka, Demeter and Saunders, David},
  title = {{Sobolev} spaces on manifolds},
  booktitle = {Handbook of Global Analysis},
  pages = {375--415},
  year = {2008},
  publisher = {Elsevier},
  address = {Amsterdam},
  doi = {10.1016/B978-044452833-9.50008-5},
}

@article{wang2004gradient,
  author = {Wang, Feng-Yu},
  title = {Gradient estimates of {Dirichlet} heat semigroups and application to isoperimetric inequalities},
  journal = {The Annals of Probability},
  volume = {32},
  number = {1A},
  pages = {424--440},
  year = {2004},
  doi = {10.1214/aop/1078415841},
}

@article{serfaty2017mean,
  author = {Serfaty, Sylvia},
  title = {Mean Field Limits of the {Gross--Pitaevskii} and Parabolic {Ginzburg--Landau} Equations},
  journal = {Journal of the American Mathematical Society},
  volume = {30},
  number = {3},
  pages = {713--768},
  year = {2017},
  doi = {10.1090/jams/872},
}

@article{duerinckx_serfaty_2020,
  author = {Serfaty, Sylvia},
  title = {Mean Field Limit for {Coulomb}-Type Flows},
  journal = {Duke Mathematical Journal},
  volume = {169},
  number = {15},
  pages = {2887--2935},
  year = {2020},
  note = {With an appendix by Mitia Duerinckx and Sylvia Serfaty},
  doi = {10.1215/00127094-2020-0019},
}

@inproceedings{sander2022sinkformers,
  author = {Sander, Michael E and Ablin, Pierre and Blondel, Mathieu and Peyr{\'e}, Gabriel},
  title = {Sinkformers: Transformers with doubly stochastic attention},
  booktitle = {International Conference on Artificial Intelligence and Statistics},
  series = {Proceedings of Machine Learning Research},
  volume = {151},
  pages = {3515--3530},
  year = {2022},
  organization = {PMLR},
}

@inproceedings{geshkovski2024emergence,
  author = {Geshkovski, Borjan and Letrouit, Cyril and Polyanskiy, Yury and Rigollet, Philippe},
  title = {The emergence of clusters in self-attention dynamics},
  booktitle = {Advances in Neural Information Processing Systems},
  volume = {36},
  pages = {57026--57037},
  year = {2023},
}

@incollection{sznitman1991topics,
  author = {Sznitman, Alain-Sol},
  title = {Topics in propagation of chaos},
  booktitle = {{\'E}cole d'{\'E}t{\'e} de Probabilit{\'e}s de Saint-Flour XIX--1989},
  series = {Lecture Notes in Mathematics},
  volume = {1464},
  pages = {165--251},
  year = {1991},
  publisher = {Springer},
  doi = {10.1007/BFb0085169},
}

@inproceedings{liu2016svgd,
  author = {Liu, Qiang and Wang, Dilin},
  title = {{Stein} Variational Gradient Descent: A General Purpose {Bayesian} Inference Algorithm},
  booktitle = {Advances in Neural Information Processing Systems},
  volume = {29},
  year = {2016},
  url = {https://papers.nips.cc/paper_files/paper/2016/hash/b3ba8f1bee1238a2f37603d90b58898d-Abstract.html},
}

@book{CheNilRig25,
  author = {Chewi, Sinho and Niles-Weed, Jonathan and Rigollet, Philippe},
  title = {Statistical optimal transport},
  series = {Lecture Notes in Mathematics},
  volume = {2364},
  year = {2025},
  publisher = {Springer, Cham},
  doi = {10.1007/978-3-031-85160-5},
}

@misc{Che26,
  author = {Chewi, Sinho},
  title = {Log-concave sampling},
  year = {2026},
  howpublished = {Book draft. \url{https://chewisinho.github.io/}},
  note = {Final draft dated August 20, 2026},
  url = {https://chewisinho.github.io/},
}

@article{peyre2018comparison,
  author = {Peyre, R{\'e}mi},
  title = {Comparison between ${W}_2$ distance and $\dot{H}^{-1}$ norm, and localization of {W}asserstein distance},
  journal = {ESAIM: Control, Optimisation and Calculus of Variations},
  volume = {24},
  number = {4},
  pages = {1489--1501},
  year = {2018},
  publisher = {EDP Sciences},
  doi = {10.1051/cocv/2017050},
}

@article{bresch2019modulated,
  author = {Bresch, Didier and Jabin, Pierre-Emmanuel and Wang, Zhenfu},
  title = {Modulated free energy and mean field limit},
  journal = {S{\'e}minaire Laurent Schwartz---EDP et applications},
  pages = {1--22},
  year = {2019--2020},
  note = {Talk no. 2},
  doi = {10.5802/slsedp.135},
}

@inproceedings{multiscale2025transformers,
  author = {Bruno, Giuseppe and Pasqualotto, Federico and Agazzi, Andrea},
  title = {A multiscale analysis of mean-field transformers in the moderate interaction regime},
  booktitle = {Advances in Neural Information Processing Systems},
  volume = {38},
  year = {2025},
  url = {https://openreview.net/forum?id=WCRPgBpbcA},
}

@article{cheng_li_yau_1981_upper_heat,
  author = {Cheng, Shiu-Yuen and Li, Peter and Yau, Shing-Tung},
  title = {On the upper estimate of the heat kernel of a complete {Riemannian} manifold},
  journal = {American Journal of Mathematics},
  volume = {103},
  pages = {1021--1063},
  year = {1981},
  doi = {10.2307/2374257},
}

@article{hsu1999estimates,
  author = {Hsu, Elton P.},
  title = {Estimates of derivatives of the heat kernel on a compact {Riemannian} manifold},
  journal = {Proceedings of the American Mathematical Society},
  volume = {127},
  number = {12},
  pages = {3739--3744},
  year = {1999},
  doi = {10.1090/S0002-9939-99-04967-9},
}

@article{ambrosio2014metric,
  author = {Ambrosio, Luigi and Gigli, Nicola and Savar{\'e}, Giuseppe},
  title = {Metric measure spaces with {Riemannian Ricci} curvature bounded from below},
  journal = {Duke Mathematical Journal},
  volume = {163},
  number = {7},
  pages = {1405--1490},
  year = {2014},
  doi = {10.1215/00127094-2681605},
}

@article{dobrushin1979vlasov,
  author = {Dobrushin, Roland L'vovich},
  title = {Vlasov equations},
  journal = {Functional Analysis and Its Applications},
  volume = {13},
  number = {2},
  pages = {115--123},
  year = {1979},
  publisher = {Springer},
  doi = {10.1007/BF01077243},
}

@article{chaintron2022propagation,
  author = {Chaintron, Louis-Pierre and Diez, Antoine},
  title = {Propagation of chaos: a review of models, methods and applications. {I}. Models and methods},
  journal = {arXiv preprint arXiv:2203.00446},
  year = {2022},
}

@article{Dev83,
  author = {Devroye, Luc},
  title = {The Equivalence of Weak, Strong and Complete Convergence in {$L_1$} for Kernel Density Estimates},
  journal = {The Annals of Statistics},
  volume = {11},
  number = {3},
  pages = {896--904},
  year = {1983},
  doi = {10.1214/aos/1176346255},
}

@article{duerinckx2016mean,
  author = {Duerinckx, Mitia},
  title = {Mean-field limits for some {Riesz} interaction gradient flows},
  journal = {SIAM Journal on Mathematical Analysis},
  volume = {48},
  number = {3},
  pages = {2269--2300},
  year = {2016},
  publisher = {SIAM},
  doi = {10.1137/15M1042620},
}

@article{serfaty2020mean,
  author = {Serfaty, Sylvia},
  title = {Mean field limit for {Coulomb}-type flows},
  journal = {Duke Mathematical Journal},
  volume = {169},
  number = {15},
  pages = {2887--2935},
  year = {2020},
  note = {With an appendix by Mitia Duerinckx and Sylvia Serfaty},
  doi = {10.1215/00127094-2020-0019},
}

@article{nguyen2021mean,
  author = {Nguyen, Quoc Hung and Rosenzweig, Matthew and Serfaty, Sylvia},
  title = {Mean-field limits of {Riesz}-type singular flows},
  journal = {arXiv preprint arXiv:2107.02592},
  year = {2021},
}

@article{rosenzweig2023global,
  author = {Rosenzweig, Matthew and Serfaty, Sylvia},
  title = {Global-in-time mean-field convergence for singular {Riesz}-type diffusive flows},
  journal = {The Annals of Applied Probability},
  volume = {33},
  number = {2},
  pages = {754--798},
  year = {2023},
  publisher = {Institute of Mathematical Statistics},
}

@article{rosenzweig2026sharp,
  author = {Rosenzweig, Matthew and Serfaty, Sylvia},
  title = {Sharp commutator estimates of all order for {Coulomb} and {Riesz} modulated energies},
  journal = {Communications on Pure and Applied Mathematics},
  volume = {79},
  number = {2},
  pages = {207--292},
  year = {2026},
  publisher = {Wiley Online Library},
}

@article{hauray2012mean,
  author = {Hauray, Maxime},
  title = {Mean field limit for the one dimensional {Vlasov-Poisson} equation},
  journal = {S{\'e}minaire Laurent Schwartz---EDP et applications},
  pages = {1--16},
  year = {2012--2013},
  note = {Talk no. 21},
  doi = {10.5802/slsedp.47},
}

@article{craig2016blob,
  author = {Craig, Katy and Bertozzi, Andrea},
  title = {A blob method for the aggregation equation},
  journal = {Mathematics of computation},
  volume = {85},
  number = {300},
  pages = {1681--1717},
  year = {2016},
}

@article{burger2023porous,
  author = {Burger, Martin and Esposito, Antonio},
  title = {Porous medium equation and cross-diffusion systems as limit of nonlocal interaction},
  journal = {Nonlinear Analysis},
  volume = {235},
  pages = {113347},
  year = {2023},
  publisher = {Elsevier},
  doi = {10.1016/j.na.2023.113347},
}

@article{carrillo2023nonlocal,
  author = {Carrillo, Jos{\'e} Antonio and Esposito, Antonio and Wu, Jeremy Sheung-Him},
  title = {Nonlocal approximation of nonlinear diffusion equations},
  journal = {Calculus of Variations and Partial Differential Equations},
  volume = {63},
  number = {4},
  pages = {100},
  year = {2024},
  doi = {10.1007/s00526-024-02690-z},
}

@article{craig2023blob,
  author = {Craig, Katy and Elamvazhuthi, Karthik and Haberland, Matt and Turanova, Olga},
  title = {A blob method for inhomogeneous diffusion with applications to multi-agent control and sampling},
  journal = {Mathematics of Computation},
  volume = {92},
  number = {344},
  pages = {2575--2654},
  year = {2023},
}

@article{di2025approximation,
  author = {Di Francesco, Marco and Iorio, Valeria and Schmidtchen, Markus},
  title = {The approximation of the quadratic porous medium equation via nonlocal interacting particles subject to repulsive {Morse} potential},
  journal = {SIAM Journal on Mathematical Analysis},
  volume = {57},
  number = {5},
  pages = {4631--4679},
  year = {2025},
  publisher = {SIAM},
  doi = {10.1137/24M1632760},
}

@article{doumic2024multispecies,
  author = {Doumic, Marie and Hecht, Sophie and Perthame, Beno{\^\i}t and Peurichard, Diane},
  title = {Multispecies cross-diffusions: from a nonlocal mean-field to a porous medium system without self-diffusion},
  journal = {Journal of Differential Equations},
  volume = {389},
  pages = {228--256},
  year = {2024},
  publisher = {Elsevier},
  doi = {10.1016/j.jde.2024.01.017},
}

@article{elbar2024limit,
  author = {Elbar, Charles and Perthame, Beno{\^\i}t and Skrzeczkowski, Jakub},
  title = {On the limit problem arising in the kinetic derivation of a {Cahn-Hilliard} equation},
  journal = {Communications in Mathematical Physics},
  volume = {405},
  number = {11},
  pages = {273},
  year = {2024},
  publisher = {Springer},
  doi = {10.1007/s00220-024-05142-z},
}

@article{carrillo2020particle,
  author = {Carrillo, Jose Antonio and Hu, Jingwei and Wang, Li and Wu, Jeremy},
  title = {A particle method for the homogeneous {Landau} equation},
  journal = {Journal of Computational Physics: X},
  volume = {7},
  pages = {100066},
  year = {2020},
  publisher = {Elsevier},
}

@article{carrillo2024landau,
  author = {Carrillo, Jos{\'e} Antonio and Delgadino, Matias G and Desvillettes, Laurent and Wu, Jeremy S-H},
  title = {The {Landau} equation as a gradient flow},
  journal = {Analysis \& PDE},
  volume = {17},
  number = {4},
  pages = {1331--1375},
  year = {2024},
  publisher = {Mathematical Sciences Publishers},
  doi = {10.2140/apde.2024.17.1331},
}

@article{craig2025blob,
  author = {Craig, Katy and Elamvazhuthi, Karthik and Lee, Harlin},
  title = {A blob method for mean field control with terminal constraints},
  journal = {ESAIM: Control, Optimisation and Calculus of Variations},
  volume = {31},
  pages = {20},
  year = {2025},
  publisher = {EDP Sciences},
  doi = {10.1051/cocv/2025010},
}

@article{craig2025nonlocal,
  author = {Craig, Katy and Jacobs, Matt and Turanova, Olga},
  title = {Nonlocal approximation of slow and fast diffusion},
  journal = {Journal of Differential Equations},
  volume = {426},
  pages = {782--852},
  year = {2025},
  publisher = {Elsevier},
}

@article{oelschlager1985law,
  author = {Oelschl{\"a}ger, Karl},
  title = {A law of large numbers for moderately interacting diffusion processes},
  journal = {Zeitschrift f{\"u}r Wahrscheinlichkeitstheorie und verwandte Gebiete},
  volume = {69},
  number = {2},
  pages = {279--322},
  year = {1985},
  publisher = {Springer},
  doi = {10.1007/BF02450284},
}

@article{oelschlager1990large,
  author = {Oelschl{\"a}ger, Karl},
  title = {Large systems of interacting particles and the porous medium equation},
  journal = {Journal of differential equations},
  volume = {88},
  number = {2},
  pages = {294--346},
  year = {1990},
  publisher = {Elsevier},
}

@article{oelschlager2001sequence,
  author = {Oelschl{\"a}ger, Karl},
  title = {A sequence of integro-differential equations approximating a viscous porous medium equation},
  journal = {Zeitschrift f{\"u}r Analysis und ihre Anwendungen},
  volume = {20},
  number = {1},
  pages = {55--91},
  year = {2001},
}

@article{figalli2008convergence,
  author = {Figalli, Alessio and Philipowski, Robert},
  title = {Convergence to the viscous porous medium equation and propagation of chaos},
  journal = {ALEA Lat. Am. J. Probab. Math. Stat},
  volume = {4},
  pages = {185--203},
  year = {2008},
}

@article{morale2005interacting,
  author = {Morale, Daniela and Capasso, Vincenzo and Oelschl{\"a}ger, Karl},
  title = {An interacting particle system modelling aggregation behavior: from individuals to populations},
  journal = {Journal of mathematical biology},
  volume = {50},
  number = {1},
  pages = {49--66},
  year = {2005},
  publisher = {Springer},
  doi = {10.1007/s00285-004-0279-1},
}

@article{chen2021rigorous,
  author = {Chen, Li and Daus, Esther S and Holzinger, Alexandra and J{\"u}ngel, Ansgar},
  title = {Rigorous derivation of population cross-diffusion systems from moderately interacting particle systems},
  journal = {Journal of Nonlinear Science},
  volume = {31},
  number = {6},
  pages = {94},
  year = {2021},
  publisher = {Springer},
  doi = {10.1007/s00332-021-09747-9},
}

@article{daneri2022deterministic,
  author = {Daneri, Sara and Radici, Emanuela and Runa, Eris},
  title = {Deterministic particle approximation of aggregation-diffusion equations on unbounded domains},
  journal = {Journal of Differential Equations},
  volume = {312},
  pages = {474--517},
  year = {2022},
  publisher = {Elsevier},
}

@article{carrillo2025rate,
  author = {Carrillo, Jos{\'e} Antonio and Elbar, Charles and Fronzoni, Stefano and Skrzeczkowski, Jakub},
  title = {Rate of convergence for a nonlocal-to-local limit in one dimension},
  journal = {arXiv preprint arXiv:2505.07015},
  year = {2025},
}

@article{carrillo2026new,
  author = {Carrillo, Jos{\'e} Antonio and Gwiazda, Piotr and Skrzeczkowski, Jakub},
  title = {A new formula for the {Wasserstein} distance between solutions to (nonlinear) continuity equations},
  journal = {arXiv preprint arXiv:2603.25634},
  year = {2026},
}

@article{chizat2026quantitative_svgd,
  author = {Chizat, L{\'e}na{\"\i}c and Colombo, Maria and Colombo, Roberto and Fern{\'a}ndez-Real, Xavier},
  title = {Quantitative Local Convergence of Mean-Field {Stein} Variational Gradient Flow},
  journal = {arXiv preprint arXiv:2605.09456},
  year = {2026},
}

@article{chizat2026quantitative_kmd,
  author = {Chizat, L{\'e}na{\"\i}c and Colombo, Maria and Colombo, Roberto and Fern{\'a}ndez-Real, Xavier},
  title = {Quantitative convergence of {Wasserstein} gradient flows of Kernel Mean Discrepancies},
  journal = {arXiv preprint arXiv:2603.01977},
  year = {2026},
}

@article{amassad2025deterministic,
  author = {Amassad, Amina and Zhou, Datong},
  title = {A deterministic particle method for the porous media equation},
  journal = {arXiv preprint arXiv:2501.18745},
  year = {2025},
}

@article{larios2009higher,
  author = {Larios, Adam and Titi, Edriss S},
  title = {On the higher-order global regularity of the inviscid {Voigt}-regularization of three-dimensional hydrodynamic models},
  journal = {arXiv preprint arXiv:0910.3354},
  year = {2009},
}

@article{deng2026drifting,
  author = {Deng, Mingyang and Li, He and Li, Tianhong and Du, Yilun and He, Kaiming},
  title = {Generative Modeling via Drifting},
  journal = {arXiv preprint arXiv:2602.04770},
  year = {2026},
  url = {https://arxiv.org/abs/2602.04770},
}

@article{lai2026unified,
  author = {Lai, Chieh-Hsin and Nguyen, Bac and Murata, Naoki and Takida, Yuhta and Uesaka, Toshimitsu and Mitsufuji, Yuki and Ermon, Stefano and Tao, Molei},
  title = {A Unified View of Score-Based and Drifting Models},
  journal = {arXiv preprint arXiv:2603.07514},
  year = {2026},
  url = {https://arxiv.org/abs/2603.07514},
}

@article{turan2026drifting,
  author = {Turan, Erkan and Dufour, Nicolas and Ovsjanikov, Maks},
  title = {Generative Drifting is Secretly Score Matching: a Spectral and Variational Perspective},
  journal = {arXiv preprint arXiv:2603.09936},
  year = {2026},
  url = {https://arxiv.org/abs/2603.09936},
}

@article{girouard2021large,
  author = {Girouard, Alexandre and Lagac{\'e}, Jean},
  title = {Large {Steklov} eigenvalues via homogenisation on manifolds},
  journal = {Inventiones mathematicae},
  volume = {226},
  number = {3},
  pages = {1011--1056},
  year = {2021},
  publisher = {Springer},
  doi = {10.1007/s00222-021-01058-w},
}

@article{peletier2025nonlinear,
  author = {Peletier, Mark A and Shalova, Anna},
  title = {Nonlinear diffusion limit of non-local interactions on a sphere},
  journal = {arXiv preprint arXiv:2512.03185},
  year = {2025},
}

\appendix
\section{Kernel Estimates}\label{app:kernel_facts}

We first prove the two kernel facts stated in \Cref{sec:setting}, and then
record the estimates required for the vector field
\begin{equation}
\label{eq:def_V}
V[\mu]:=-\nabla\log \K_b\mu .
\end{equation}

\begin{lemma}[Distance bounds for \(K_b\)]\label{lem:wellposed_kernel_distance_bounds} For every $b>0$, there exists $C_{b,\M}>0$ such that, for all $x,y\in \M$ with $x\neq y$:
\begin{equation}\label{eq:wellposed_kernel_distance_zero_first}
|K_b(x,y)|\le C_{b,\M}\Phi_d(\operatorname{dist}_{\M}(x,y)),
\qquad
|\nabla_xK_b(x,y)|+|\nabla_yK_b(x,y)|
\le C_{b,\M}\operatorname{dist}_{\M}(x,y)^{1-d},
\end{equation}
and
\begin{equation}\label{eq:wellposed_kernel_distance_second}
|\nabla_x^2K_b(x,y)|+|\nabla_y\nabla_xK_b(x,y)|
\le C_{b,\M}\operatorname{dist}_{\M}(x,y)^{-d},
\end{equation}
where
\[
\Phi_d(r):=
\begin{cases}
1, & d=1,\\
\log(2+r^{-1}), & d=2,\\
r^{2-d}, & d\ge3.
\end{cases}
\]
\end{lemma}

\begin{proof} We denote throughout $r:=\operatorname{dist}_{\M}(x,y)$.
When  \(d=1\), the explicit formula \eqref{eq:explicit_circle_resolvent_kernel}
shows that \(K_b\) and its first derivative are bounded. On
each side of the pole, the equation \((\I-b^2\partial_\theta^2)K_b=0\) gives
\(|\partial_\theta^2K_b|\le C_b\). Since \(r^{-1}\ge1\)  for \(r\le1/2\), the
three displayed estimates follow  near the pole, while away from the pole the
same estimates follow by compactness, after enlarging the constant.

Assume now \(d\ge2\). By \eqref{eq:Kb_heat_kernel_representation}, after the change of variables \(s=b^2t\),
\[
K_b(x,y)=b^{-2}\int_0^\infty e^{-s/b^2}p_s(x,y)\,ds.
\]
Combining the heat-kernel Gaussian upper bound
\cite[Theorem~4]{cheng_li_yau_1981_upper_heat}, which on a fixed closed
manifold gives constants \(C_{\M},c_{\M}>0\) such that
\[
p_s(x,y)
\le
C_{\M}s^{-d/2}
\exp\left(-\frac{\operatorname{dist}_{\M}(x,y)^2}{c_{\M}s}\right),
\qquad 0<s\le1 .
\]
with Hsu's derivative estimate
\cite[Corollary~1.2]{hsu1999estimates},
\[
|\nabla_x^j p_s(x,y)|
\le
C_{\M,j}
\left(\frac{\operatorname{dist}_{\M}(x,y)}{s}+\frac{1}{\sqrt{s}}\right)^j
p_s(x,y),
\qquad j\le2,
\]
 and absorbing the polynomial factor into the Gaussian, we obtain that for every mixed derivative $D^\beta_{x,y}$ of total order $|\beta|=j\leq 2$,
\[
|D^\beta_{x,y}p_s(x,y)|
\le
C_{\M,j}s^{-(d+j)/2}
\exp\left(-\frac{\operatorname{dist}_{\M}(x,y)^2}{C_{\M,j}s}\right)
\qquad 0<s\le1 .
\]
 For \(s\ge1\), compactness and smoothness of the heat kernel give a bounded
contribution, depending on \(b\) and \(\M\). Hence
\[
|D^\beta_{x,y}K_b(x,y)|
\le
C_{b,\M}
+C_{b,\M}\int_0^1
s^{-(d+j)/2}\exp\left(-\frac{r^2}{C_{\M,j}s}\right)\,ds .
\]
With the change of variables \(q=r^2/(C_{\M,j}s)\),
\[
\int_0^1
s^{-(d+j)/2}\exp\left(-\frac{r^2}{C_{\M,j}s}\right)\,ds
=
C_{\M,j}r^{2-d-j}
\int_{r^2/C_{\M,j}}^\infty q^{(d+j)/2-2}e^{-q}\,dq .
\]
 Therefore the last integral is bounded by
\[
C_{\M,j}
\begin{cases}
1, & d+j<2,\\
\log(2+r^{-1}), & d+j=2,\\
r^{2-d-j}, & d+j>2.
\end{cases}
\]
Taking \(j=0\) gives the bound with \(\Phi_d\), taking \(j=1\) gives the first
derivative bound, and taking \(j=2\) gives both the pure and mixed second
derivative bounds. Since \(\M\) is compact, the bounded terms are
absorbed into the displayed singular bounds after enlarging \(C_{b,\M}\).
\end{proof}

For the next kernel estimates, we adapt the near-field/far-field decompositions of \cite[Appendix A]{chizat2026quantitative_kmd}.

\begin{lemma}[Fixed-\(b\) kernel estimates]\label{lem:wellposed_kernel_estimates}
Let \(M>0\). If \(\rho\) is a probability density on  \(\M\) with
\(\|\rho\|_{L^\infty}\le M\), and \(u=\Kb\rho\), then there exist $\kappa_b>0$ and $C_{b,M,\M}>0$ such that
\begin{equation}\label{eq:wellposed_fixed_b_basic}
\kappa_b\le u\le M,\qquad \|\nabla u\|_{L^\infty}\le C_{b,M,\M}.
\end{equation}

Furthermore, for every $x,y\in \M$ with $x\neq y$, let \(P_{y\to x}\) the parallel
transport along any minimizing geodesic from \(y\) to \(x\), then:
\[
|\nabla u(x)-P_{y\to x}\nabla u(y)|
\le C_{b,M,\M}\operatorname{dist}_{ \M}(x,y)\log(2+\operatorname{dist}_{ \M}(x,y)^{-1}).
\]
Consequently
\[
|V[\rho](x)-P_{y\to x}V[\rho](y)|
\le C_{b,M,\M}\operatorname{dist}_{ \M}(x,y)\log(2+\operatorname{dist}_{ \M}(x,y)^{-1}).
\]
If \(\rho\in C^{k,\alpha}( \M)\), \(k\ge0\), \(0<\alpha<1\), then for all $b>0$
\begin{equation}
\label{eq:schauder}
\|\K_b\rho\|_{C^{k+2,\alpha}}
\le C_{b,k,\alpha,\M}\|\rho\|_{C^{k,\alpha}},
\quad
\|V[\rho]\|_{C^{k+1,\alpha}}
\le C_{b,k,\alpha,\M}\bigl(1+\|\rho\|_{C^{k,\alpha}}\bigr)^{k+2}.
\end{equation}
\end{lemma}

\begin{proof}
By \eqref{eq:Kb_positive_lower_bound} in the proof of
\Cref{lem:Kb_positive}, \(K_b(x,y)\ge\kappa_b\), and hence
\(u(x)\ge\kappa_b\int\rho=\kappa_b\).
Also \(\Kb1=1\) and \(K_b\) is positive, hence \(u=\Kb\rho\le M\Kb1=M\).

By \Cref{lem:wellposed_kernel_distance_bounds},
\[
\begin{aligned}
|\nabla_xK_b(x,z)|
&\le C_{b,\M}\operatorname{dist}_{\M}(x,z)^{1-d},\\
|\nabla_x^2K_b(x,z)|+|\nabla_z\nabla_xK_b(x,z)|
&\le C_{b,\M}\operatorname{dist}_{\M}(x,z)^{-d},
\end{aligned}
\]
for \(x\neq z\). The first estimate is integrable on  \(\M\), so
\(\|\nabla u\|_\infty\le C_{b,M,\M}\).

Let \(r=\operatorname{dist}_{ \M}(x,y)\), and fix a minimizing geodesic with parallel
transport \(P_{y\to x}\). 
We split the integral defining \(\nabla u(x)-P_{y\to x}\nabla u(y)\) into
\(B_{2r}(x)\) and \( \M\setminus B_{2r}(x)\). 
Here $B_{2r}(x)$ denotes the geodesic ball of radius $2r$ centered at $x$. Using geodesic polar coordinates $(\theta,s)$ about any \(a\in\M\),
 compactness gives a uniform bound for the polar
Jacobian (i.e. \cite[Eq.~(9)]{girouard2021large}),
\[
J_a(\theta,s)\le C_{\M}s^{d-1},
\qquad
0<s\le\operatorname{diam}(\M).
\]

If \(r\ge\operatorname{diam}(\M)/4\), the estimate follows from
\(\|\nabla u\|_{L^\infty}\le C_{b,M,\M}\), after enlarging the constant. Thus we
may assume \(r<\operatorname{diam}(\M)/4\).
On $B_{2r}(x)$,
\[
\begin{aligned}
&\left|
\int_{B_{2r}(x)}
\left(\nabla_xK_b(x,z)-P_{y\to x}\nabla_1K_b(y,z)\right)
\rho(z)\,dz
\right|\\
&\qquad\leq M\int_{B_{2r}(x)}
\left|\nabla_xK_b(x,z)-P_{y\to x}\nabla_1K_b(y,z)\right|\,dz\\
&\qquad\le
C_{b,\M}M\int_0^{2r}s^{1-d}s^{d-1}\,ds\\
&\qquad\quad
+C_{b,\M}M\int_0^{3r}s^{1-d}s^{d-1}\,ds\\
&\qquad\le C_{b,M,\M}r.
\end{aligned}
\]
On \( \M\setminus B_{2r}(x)\), 
let \(\gamma:[0,r]\to\M\) be the unit-speed minimizing geodesic from \(x\)
to \(y\). For \(z\notin B_{2r}(x)\) and \(s\in[0,r]\),
\[
\operatorname{dist}_{\M}(\gamma(s),z)
\ge \operatorname{dist}_{\M}(x,z)-s
\ge \frac12\operatorname{dist}_{\M}(x,z).
\]
Therefore,
\[
\begin{aligned}
&\left|
\int_{\M\setminus B_{2r}(x)}
\left(\nabla_xK_b(x,z)-P_{y\to x}\nabla_1K_b(y,z)\right)
\rho(z)\,dz
\right|\\
&\qquad\le
M\int_{\M\setminus B_{2r}(x)}
\int_0^r
|\nabla_1^2K_b(\gamma(s),z)|\,ds\,dz\\
&\qquad\le
C_{b,\M}Mr
\int_{\M\setminus B_{2r}(x)}
\operatorname{dist}_{\M}(x,z)^{-d}\,dz\\
&\qquad\le
C_{b,\M}Mr
\int_{2r}^{\operatorname{diam}(\M)}s^{-d}s^{d-1}\,ds\\
&\qquad\le C_{b,M,\M}r\log(2+r^{-1}).
\end{aligned}
\]

Collecting the estimates gives the log-Lipschitz
bound for \(\nabla u\). Since
\(u\ge\kappa_b\),
\[
V[\rho](x)-P_{y\to x}V[\rho](y)
=-\frac{\nabla u(x)-P_{y\to x}\nabla u(y)}{u(x)}
-P_{y\to x}\nabla u(y)\left(\frac1{u(x)}-\frac1{u(y)}\right),
\]
and \(|u(x)-u(y)|\le\|\nabla u\|_\infty r\), the same modulus holds for
\(V[\rho]\).

The Schauder estimate for \((\I-b^2\Delta)u=\rho\), applied in finitely many
charts on the  closed manifold \(\M\) \cite[Theorem~6.2 and Problem~6.1(a)]{gilbarg_trudinger_2001},
gives
\[
\|u\|_{C^{k+2,\alpha}}
\le C_{b,k,\alpha,\M}\bigl(\|\rho\|_{C^{k,\alpha}}+\|u\|_{L^\infty}\bigr).
\]
The maximum principle gives \(\|u\|_{L^\infty}\le\|\rho\|_{L^\infty}\), and hence
the stated \(C^{k+2,\alpha}\) estimate for \(\Kb\rho\). The bound for
\(V[\rho]\) follows from \(u\ge\kappa_b\), the preceding Schauder estimate, and
composition with \(\log\) on \([\kappa_b,\infty)\).
\end{proof}

\begin{lemma}[Velocity stability]\label{lem:wellposed_velocity_stability}\label{lem:wellposed_velocity_stability_l2}
Let \(\rho,\eta\) be probability densities on  \(\M\), and assume
that \(\|\rho\|_{L^\infty},\|\eta\|_{L^\infty}\leq M\). Then
\[
\|V[\rho]-V[\eta]\|_{L^\infty}
\le C_{b,M,\M} W_2(\rho,\eta)^{ 1/(d+1)}.
\]
Moreover, 
\[
\|V[\rho]-V[\eta]\|_{L^2(dx)}
\le C_{b,M,\M}W_2(\rho,\eta).
\]
\end{lemma}

\begin{proof}
Set \(\sigma=\rho-\eta\), \(u_\rho=\mathcal K_b\rho\), and \(u_\eta=\mathcal K_b\eta\). By
\eqref{eq:wellposed_fixed_b_basic}, applied to \(\rho\) and \(\eta\),
\(u_\rho,u_\eta\ge\kappa_b\) and \(\|\nabla u_\eta\|_{L^\infty}\le C_{b,M,\M}\).
Since
\begin{equation}\label{eq:wellposed_velocity_difference}
V[\rho]-V[\eta]
=-\frac{\nabla \mathcal K_b\sigma}{u_\rho}
+\frac{(\nabla u_\eta)\mathcal K_b\sigma}{u_\rho u_\eta},
\end{equation}
it is enough to estimate \(\mathcal K_b\sigma\) and \(\nabla \mathcal K_b\sigma\).

Fix \(x\in \M\) and \(0<\varepsilon<\varepsilon_{\M}\). Let
\(\psi_x^\varepsilon\) be supported in \(B_{\varepsilon/4}(x)\), equal to
\(1\) on \(B_{\varepsilon/8}(x)\), and satisfy
\(|\nabla\psi_x^\varepsilon|\le C_{\M}\varepsilon^{-1}\). Then
\[
\nabla \mathcal K_b\sigma(x)
=\int\nabla_xK_b(x,z)\psi_x^\varepsilon(z)\,d\sigma(z)
+\int\nabla_xK_b(x,z)(1-\psi_x^\varepsilon(z))\,d\sigma(z)
=(I)+(II).
\]
For the first term,
\[
|(I)|
\le 2M\int_{B_{\varepsilon/4}(x)}|\nabla_xK_b(x,z)|\,dz
\le C_{b,\M} M\int_0^{\varepsilon/4}s^{1-d}s^{d-1}\,ds
\le C_{b,\M} M\varepsilon .
\]
For the second term, the map
\[
z\mapsto\nabla_xK_b(x,z)(1-\psi_x^\varepsilon(z))
\]
has Lipschitz constant at most \(C_{b,\M}\varepsilon^{-d}\), by
\eqref{eq:wellposed_kernel_distance_zero_first} and
\eqref{eq:wellposed_kernel_distance_second}. Kantorovich--Rubinstein duality
therefore gives
\[
|(II)|\le C_{b,\M}\varepsilon^{-d}W_1(\rho,\eta)
\le C_{b,\M}\varepsilon^{-d}W_2(\rho,\eta).
\]
By \eqref{eq:wellposed_kernel_distance_zero_first},
\[
2M\int_{B_{\varepsilon/4}(x)}|K_b(x,z)|\,dz\le C_{b,\M} M\varepsilon .
\]
The cutoff map \(z\mapsto K_b(x,z)(1-\psi_x^\varepsilon(z))\) has Lipschitz
constant at most \(C_{b,\M}\varepsilon^{-d}\), again by
\eqref{eq:wellposed_kernel_distance_zero_first}. Thus the same duality
argument gives
\[
|\mathcal K_b\sigma(x)|
\le C_{b,\M} M\varepsilon+C_{b,\M}\varepsilon^{-d}W_2(\rho,\eta).
\]
Thus
\[
\|\nabla \mathcal K_b\sigma\|_{L^\infty}+\|\mathcal K_b\sigma\|_{L^\infty}
\le C_{b,\M} M\varepsilon+C_{b,\M}\varepsilon^{-d}W_2(\rho,\eta).
\]
If \(W_2(\rho,\eta)=0\), there is nothing to prove. If
\(W_2(\rho,\eta)\ge M\varepsilon_{\M}^{d+1}\), the desired \(L^\infty\) bound
follows from the uniform estimate \(\|V[\rho]\|_{L^\infty}+\|V[\eta]\|_{L^\infty}
\le C_{b,M,\M}\), after possibly enlarging the constant. We may therefore assume
\(W_2(\rho,\eta)<M\varepsilon_{\M}^{d+1}\). Choosing
\[
\varepsilon
=\left(\frac{W_2(\rho,\eta)}{M}\right)^{ 1/(d+1)}\in(0,\varepsilon_{\M}),
\]
and gathering the estimates for \((I)\) and \((II)\), we get
\[
\|\nabla \mathcal K_b\sigma\|_{L^\infty}+\|\mathcal K_b\sigma\|_{L^\infty}
\le C_{b,M,\M}W_2(\rho,\eta)^{ 1/(d+1)}.
\]
The bound above, combined with \eqref{eq:wellposed_velocity_difference}, gives the claimed
\(L^\infty\) estimate for \(V[\rho]-V[\eta]\).

For the $L^2$ estimate, \eqref{eq:wellposed_velocity_difference} implies
\[
\|V[\rho]-V[\eta]\|_{L^2}
\le C_{b,M,\M}\bigl(\|\nabla \mathcal K_b\sigma\|_{L^2}+\|\mathcal K_b\sigma\|_{L^2}\bigr).
\]
Writing
\(\sigma=\sum_j\sigma_je_j\), where \(-\Delta e_j=\Lambda_je_j\) and
\(\{e_j\}\) is a  Laplacian eigenbasis on \(\M\),
\[
\mathcal K_b\sigma=\sum_j\frac{\sigma_j}{1+b^2\Lambda_j}e_j.
\]
Therefore
\[
\|\mathcal K_b\sigma\|_{H^1}^2
=
\sum_j
\frac{1+\Lambda_j}{(1+b^2\Lambda_j)^2}|\sigma_j|^2
\le
C_b\sum_j(1+\Lambda_j)^{-1}|\sigma_j|^2
=
C_b\|\sigma\|_{H^{-1}}^2.
\]
The \(H^{-1}\)--\(W_2\) comparison follows from the proof of
\cite[Theorem~5]{peyre2018comparison}, combined with \cite[Eq.~(3.2)]{ambrosio2014metric}, and gives 
\[
\|\rho-\eta\|_{H^{-1}( \M)}\le C_{M,\M}W_2(\rho,\eta).
\]

Combining the last three equations gives the claim.
\end{proof}

The following lemma extends the \(H^{-1}\)--\(W_2\) comparison of \cite[Theorem~5]{peyre2018comparison} from manifolds with nonnegative Ricci curvature to arbitrary smooth, closed, connected Riemannian manifolds, at the cost of a geometry-dependent constant.

\begin{lemma}[$H^{-1}$--$W_2$ on closed manifolds]
\label{lem:hminusone_wasserstein}
Let $(\M,g)$ be a smooth, closed, connected Riemannian manifold, and set
\begin{align*}
K_{\M}
:=
\min_{\substack{x\in\M\\ v\in T_x\M,\ |v|_g=1}}
\Ric_x(v,v),
\quad
K_{\M}^-:=\max\{-K_{\M},0\}, \quad
D_{\M}:=\operatorname{diam}(\M).
\end{align*}
Let $\rho_0,\rho_1\in L^\infty(\M)$ be probability densities and write
$M:=\|\rho_0\|_{L^\infty}\vee \|\rho_1\|_{L^\infty}$. 
Then
\begin{equation}
\label{eq:hminusone_wasserstein_uniform}
\|\rho_0-\rho_1\|_{H^{-1}(\M)}
\leq
M^{1/2}\exp\!\left(\frac{K_{\M}^-D_{\M}^2}{16}\right)
W_2(\rho_0,\rho_1).
\end{equation}
\end{lemma}

\begin{proof}
Compactness gives $K_{\M}>-\infty$ and $D_{\M}<\infty$, so that
$\Ric\geq K_{\M}g$. Moreover, the metric-measure space
$(\M,\operatorname{dist}_{\M},\d x)$ is nonbranching and satisfies the
$\mathrm{CD}(K_{\M},\infty)$ condition. It is therefore a strong
$\mathrm{CD}(K_{\M},\infty)$ space. Since $\M$ is compact, the endpoint
measures have bounded support and finite entropy. Hence
\cite[Remark~3.2, Proposition~3.3, and Eq.~(3.2)]{ambrosio2014metric}
provides a constant-speed Wasserstein geodesic
$\rho_t$, $t\in[0,1]$, joining
$\rho_0$ to $\rho_1$ and satisfying
\begin{equation}
\label{eq:displacement_linfty}
\|\rho_t\|_{L^\infty}
\leq
M\exp\!\left(\frac{K_{\M}^-D_{\M}^2}{8}\right).
\end{equation}

This bound can be readily employed in the proof of~\cite[Theorem~5]{peyre2018comparison}, which we reproduce here for completeness. 
Let $v_t$ be the minimal velocity field of this geodesic. Then
for every
$\varphi\in C^\infty(\M)$,
\begin{align*}
\left|\int_{\M}\varphi\, \d (\rho_1-\rho_0)\right|
&=
\left|\int_0^1\!\int_{\M}
\langle\nabla\varphi,v_t\rangle_g\, \d\rho_t\,\d t\right| \\
&\leq
\|\nabla\varphi\|_{L^2}
\int_0^1
\|\rho_t\|_{L^\infty}^{1/2}
\left(\int_{\M}|v_t|_g^2\, \d\rho_t\right)^{1/2}
\d t \\
&\leq
\|\nabla\varphi\|_{L^2}
W_2(\rho_0,\rho_1)
\int_0^1\|\rho_t\|_{L^\infty}^{1/2}\,\d t .
\end{align*}
To conclude, take the supremum    over $\varphi$ with
$\|\nabla\varphi\|_2\leq1$ and apply~\eqref{eq:displacement_linfty}.
\end{proof}

In what follows, we use the Hölder seminorm
\[
[f]_{C^{0,\alpha}(\M)}
:=
\sup_{x\neq y}
\frac{|f(x)-f(y)|}
{\operatorname{dist}_{\M}(x,y)^\alpha}.
\]
\begin{lemma}[Logarithmic bound]\label{lem:wellposed_log_lipschitz}
Let \(0<\alpha<1\). If \(f\in C^{0,\alpha}( \M)\), then
\begin{equation}
\label{eq:log_bound_hess}
\|\nabla^2\K_b f\|_{L^\infty}
\le
C_{b,\alpha,\M}\|f\|_{L^\infty}
\left(1+\log\left(1+\frac{[f]_{C^{0,\alpha}}}{\|f\|_{L^\infty}}\right)\right),
\end{equation}
with the convention that the right-hand side is zero when \(f=0\). Hence, if
\(\rho\in C^{0,\alpha}( \M)\) is a probability density with
\(\|\rho\|_{L^\infty}\le M\), then
\[
\|\nabla V[\rho]\|_{L^\infty}
\le
C_{b,M,\alpha,\M}\left(1+\log\bigl(2+\|\rho\|_{C^{0,\alpha}}\bigr)\right).
\]
\end{lemma}

\begin{proof}
If \(f=0\), there is nothing to prove. Let \(u=\mathcal K_bf\), and let
\(\varepsilon_{\M}>0\) be smaller than the injectivity radius of \(\M\).
Fix \(x\in\M\) and \(0<\varepsilon<\varepsilon_{\M}\), and define
\[
\psi_x^\varepsilon(z)
=
\psi\!\left(
\frac{\operatorname{dist}_{\M}(x,z)}{\varepsilon}
\right),
\]
where \(\psi\in C_c^\infty([0,1/4);[0,1])\) and
\(\psi=1\) on \([0,1/8]\).
Assume first that \(f\) is smooth. Since \(\mathcal K_b1=1\), we have
\(\nabla^2\mathcal K_b1=0\), and hence, applying the cancellation before
using the kernel representation,
\[
\nabla^2u(x)
=
\nabla^2\mathcal K_b\bigl(f-f(x)\bigr)(x)
=
\int_\M
\nabla_x^2K_b(x,z)\bigl(f(z)-f(x)\bigr)\,\d z.
\]
This integral is absolutely convergent. Indeed, by
\eqref{eq:wellposed_kernel_distance_second},
\[
\left|
\nabla_x^2K_b(x,z)\bigl(f(z)-f(x)\bigr)
\right|
\le
C_{b,\M}[f]_{C^{0,\alpha}(\M)}
\operatorname{dist}_{\M}(x,z)^{\alpha-d},
\]
and the right-hand side is locally integrable because
\[
\int_0^\varepsilon s^{\alpha-d}s^{d-1}\,\d s
=
\frac{\varepsilon^\alpha}{\alpha}<\infty.
\]

Splitting the integral with \(\psi_x^\varepsilon\), the near-field term
satisfies
\[
\begin{aligned}
&\left|
\int_\M
\psi_x^\varepsilon(z)\nabla_x^2K_b(x,z)
\bigl(f(z)-f(x)\bigr)\,\d z
\right|\\
&\qquad\le
C_{b,\M}[f]_{C^{0,\alpha}(\M)}
\int_0^{\varepsilon/4}
s^{-d}s^\alpha s^{d-1}\,\d s\\
&\qquad\le
C_{b,\alpha,\M}
[f]_{C^{0,\alpha}(\M)}\varepsilon^\alpha.
\end{aligned}
\]
Since \(1-\psi_x^\varepsilon\) vanishes on \(B_{\varepsilon/8}(x)\), the
far-field term satisfies
\[
\begin{aligned}
&\left|
\int_\M
\bigl(1-\psi_x^\varepsilon(z)\bigr)
\nabla_x^2K_b(x,z)
\bigl(f(z)-f(x)\bigr)\,\d z
\right|\\
&\qquad\le
C_{b,\M}\|f\|_{L^\infty(\M)}
\int_{\varepsilon/8}^{\operatorname{diam}(\M)}
s^{-d}s^{d-1}\,\d s\\
&\qquad\le
C_{b,\M}\|f\|_{L^\infty(\M)}
\bigl(1+\log(2+\varepsilon^{-1})\bigr).
\end{aligned}
\]
Consequently,
\[
\|\nabla^2\mathcal K_bf\|_{L^\infty(\M)}
\le
C_{b,\alpha,\M}
[f]_{C^{0,\alpha}(\M)}\varepsilon^\alpha
+
C_{b,\M}\|f\|_{L^\infty(\M)}
\bigl(1+\log(2+\varepsilon^{-1})\bigr).
\]

If \([f]_{C^{0,\alpha}(\M)}=0\), choose
\(\varepsilon=\varepsilon_{\M}/2\). Otherwise, choose
\[
\varepsilon
=
\min\left\{
\frac{\varepsilon_{\M}}2,
\left(
\frac{\|f\|_{L^\infty(\M)}}
{[f]_{C^{0,\alpha}(\M)}}
\right)^{1/\alpha}
\right\}.
\]
With this choice,
\[
[f]_{C^{0,\alpha}(\M)}\varepsilon^\alpha
\le
C_{\M}\|f\|_{L^\infty(\M)}
\]
and
\[
1+\log(2+\varepsilon^{-1})
\le
C_{\alpha,\M}
\left(
1+\log\left(
1+
\frac{[f]_{C^{0,\alpha}(\M)}}
{\|f\|_{L^\infty(\M)}}
\right)
\right).
\]
This proves the claimed estimate \eqref{eq:log_bound_hess} when \(f\) is smooth.
For \(f\in C^{0,\alpha}(\M)\), choose \(f_n\in C^\infty(\M)\) such that
\[
f_n\longrightarrow f
\quad\text{in }C^{0,\alpha/2}(\M),
\qquad
\sup_n [f_n]_{C^{0,\alpha}(\M)}
\le C_{\M}[f]_{C^{0,\alpha}(\M)},
\]
In particular,
\[
\|f_n\|_{L^\infty(\M)}
\longrightarrow
\|f\|_{L^\infty(\M)}.
\]
Since \(f\neq0\), for all sufficiently large \(n\),
\[
\frac12\|f\|_{L^\infty(\M)}
\le
\|f_n\|_{L^\infty(\M)}
\le
2\|f\|_{L^\infty(\M)}.
\]
Applying \eqref{eq:log_bound_hess}, already proved for smooth functions, gives
\begin{equation}
\label{eq:smooth_approx_log_bound_hess}
\begin{aligned}
\|\nabla^2\mathcal K_bf_n\|_{L^\infty(\M)}
&\le
C_{b,\alpha,\M}\|f_n\|_{L^\infty(\M)}
\left(
1+\log\left(
1+
\frac{[f_n]_{C^{0,\alpha}(\M)}}
{\|f_n\|_{L^\infty(\M)}}
\right)
\right)\\
&\le
C_{b,\alpha,\M}\|f\|_{L^\infty(\M)}
\left(
1+\log\left(
1+
\frac{[f]_{C^{0,\alpha}(\M)}}
{\|f\|_{L^\infty(\M)}}
\right)
\right).
\end{aligned}
\end{equation}
On the other hand, the Schauder estimate for
\(\K_b\) \eqref{eq:schauder} yields
\[
\|\mathcal K_b(f_n-f)\|_{C^{2,\alpha/2}(\M)}
\le
C_{b,\alpha,\M}
\|f_n-f\|_{C^{0,\alpha/2}(\M)}
\longrightarrow0.
\]
Therefore,
\[
\|\nabla^2\mathcal K_bf_n-\nabla^2\mathcal K_bf\|_{L^\infty(\M)}
\longrightarrow0.
\]
Taking \(n\to\infty\) in the \eqref{eq:smooth_approx_log_bound_hess} proves the claimed
inequality \eqref{eq:log_bound_hess} for \(f\in C^{0,\alpha}(\M)\).

Now let \(u=\mathcal K_b\rho\). Since \(V[\rho]=-\nabla\log u\),
\[
\nabla V[\rho]
=-\nabla^2\log u
=-\frac{\nabla^2u}{u}
+\frac{\nabla u\otimes\nabla u}{u^2}.
\]
By \Cref{lem:wellposed_kernel_estimates}, \(u\ge\kappa_b\) and
\(\|\nabla u\|_\infty\le C_{b,M,\M}\). Applying the first part with \(f=\rho\),
and using that the volume of \(\M\) is normalized, so
\(\|\rho\|_{L^\infty}\ge1\), gives the
claimed bound for \(\nabla V[\rho]\).
\end{proof}

\begin{lemma}[Divergence bound]\label{lem:wellposed_divergence_bound}
If \(\rho\) is a probability density on  \(\M\) with \(\|\rho\|_{L^\infty}\le M\), then
\[
\|\divg V[\rho]\|_{L^\infty}\le C_{b,M,\M}.
\]
\end{lemma}

\begin{proof}
Let \(u=\mathcal K_b\rho\). Since \((\I-b^2\Delta)u=\rho\),
\begin{equation}\label{eq:short_chizat_G}
\Delta\log u
=
\frac{\Delta u}{u}-|\nabla\log u|^2
=
\frac{u-\rho}{b^2u}-|\nabla\log u|^2 .
\end{equation}
By \Cref{lem:wellposed_kernel_estimates}, \(\kappa_b\le u\le M\) and
\(\|\nabla u\|_\infty\le C_{b,M,\M}\). Therefore
\(\|\Delta\log u\|_\infty\le C_{b,M,\M}\). Since
\(\divg V[\rho]=-\Delta\log u\), the proof is complete.
\end{proof}

\section{Well-posedness of the Nonlocal PDE}
\label{app:pde_wellposedness}

The well-posedness theory is obtained by combining the estimates of
\Cref{app:kernel_facts} with the Yudovich scheme developed in
\cite[Section~2]{chizat2026quantitative_kmd}. 

For a curve of probability measures
\(t\mapsto\mu_t\in\mathcal P(\M)\), we write
\[
\mu\in C_{w^*}\bigl([0,T);\mathcal P(\M)\bigr)
\]
if it is continuous in time with respect to the weak-\(*\) topology, namely if
\[
t\longmapsto\int_\M\varphi\,\d\mu_t
\]
is continuous on \([0,T)\) for every \(\varphi\in C(\M)\).
A bounded solution of \eqref{eq:pde} on \([0,T)\), with initial datum
\(\mu_{\mathrm{in}}\), is a curve
\[
\mu\in
C_{w^*}\bigl([0,T);\mathcal P(\M)\bigr)
\cap
L^\infty_{\mathrm{loc}}\bigl([0,T);L^\infty(\M)\bigr)
\]
such that \(\mu_0=\mu_{\mathrm{in}}\), the vector field $V[\mu]$ defined in \eqref{eq:def_V} satisfies
\[
V[\mu]\in
L^1_{\mathrm{loc}}
\bigl((0,T)\times\M;\d\mu_t\,\d t\bigr)
\]
and, for every \(t\in(0,T)\) and every
\(\varphi\in C^\infty(\M)\),
\[
\int_\M\varphi\,\d\mu_t
=
\int_\M\varphi\,\d\mu_{\mathrm{in}}
+
\int_0^t\int_\M
\nabla\varphi\cdot V[\mu_r]\,\d\mu_r\,\d r.
\]

\begin{proposition}[Local bounded well-posedness and continuation criterion]\label{prop:pde_local_bounded_wellposedness}
Let \(b>0\) and let \(\mu_{\mathrm{in}}\) be a bounded probability density on
 \(\M\).
Then there are \(T_{\max}\in(0,\infty]\) and a unique maximal bounded
solution of
\[
\partial_t\mu+\divg(\mu V[\mu])=0,\qquad V[\mu]=-\nabla\log \mathcal K_b\mu,
\]
with initial datum \(\mu_{\mathrm{in}}\), satisfying
\[
\mu\in L^\infty_{\rm loc}([0,T_{\max});L^\infty( \M)).
\]
It preserves mass, and the following continuation criterion holds:
\[
T_{\max}<\infty
\quad\Longrightarrow\quad
\limsup_{t\uparrow T_{\max}}\|\mu_t\|_{L^\infty( \M)}=+\infty .
\]
\end{proposition}

\begin{proof}
By
\Cref{lem:wellposed_kernel_estimates,lem:wellposed_velocity_stability},
for probability densities \(\rho,\eta\) bounded by \(M\), the velocity
\(V[\rho]=-\nabla\log \mathcal K_b\rho\) satisfies the following continuity estimates,
where \(C\) depends only on \(b\), \(M\), and  \(\M\), and $P_{y\to x}$ denotes parallel transport along a minimizing geodesic from $y$ to $x$:
\[
\mathcal K_b\rho\ge \kappa_b>0,\qquad
|V[\rho](x)-P_{y\to x}V[\rho](y)|
\le C\operatorname{dist}_{ \M}(x,y)
\log\!\left(2+\operatorname{dist}_{ \M}(x,y)^{-1}\right),
\]
and
\[
\|V[\rho]-V[\eta]\|_{L^2}
\le CW_2(\rho,\eta)
\]
as well as
\[
\|V[\rho]-V[\eta]\|_{L^\infty}
\le CW_2(\rho,\eta)^{ 1/(d+1)}.
\]
Together with the divergence bound
\[
\|\divg V[\rho]\|_{L^\infty}\le C,
\]
which follows from \Cref{lem:wellposed_divergence_bound}, these are the estimates needed to adapt the
\(L^\infty\) Yudovich construction of
\cite[Propositions~2.8--2.9]{chizat2026quantitative_kmd} in
the present setting, with the geodesic distance on  \(\M\) replacing
the distance on \(\Torus^d\).
Therefore the analogues of those propositions give local existence,
uniqueness, stability, and a maximal solution in
\[
L^\infty_{\rm loc}([0,T_{\max});L^\infty( \M)).
\]
The solution preserves mass because it solves a continuity equation.  Moreover, the Yudovich construction cited above yields the following continuation criterion:
\[
T_{\max}<\infty
\quad\Longrightarrow\quad
\limsup_{t\uparrow T_{\max}}\|\mu_t\|_{L^\infty}=+\infty .
\]
\end{proof}

\begin{proposition}[Local propagation of Hölder regularity]\label{prop:pde_holder_propagation}
Let \(0<\alpha<1\), \(k\in\mathbb N\cup\{0\}\), and let
    \(\mu_{\mathrm{in}}\in C^{k,\alpha}( \M)\) be a bounded probability density.
Let \(\mu\) be the maximal bounded solution from
\Cref{prop:pde_local_bounded_wellposedness}, with maximal time
\(T_{\max}\).  Then
\[
\mu\in L^\infty_{\rm loc}([0,T_{\max});C^{k,\alpha}( \M)).
\]
If \(\mu_{\mathrm{in}}\in C^\infty\), then \(\mu\) is smooth in space and time on
\([0,T_{\max})\times \M\).
\end{proposition}

\begin{proof}
Fix \(T\in(0,T_{\max})\).  The proof proceeds in three stages: short-time Hölder regularity,
continuation, and a higher-order bootstrap.  
Set
\[
M_T:=\sup_{t\in[0,T]}\|\mu_t\|_{L^\infty}<\infty .
\]
Throughout the proof, \(C\) denotes a positive, finite constant whose value may
change from line to line.  In Steps 1--2 it depends only on
 \(\M\), \(b\), \(\alpha\), \(T\), \(M_T\), and
\(\|\mu_{\mathrm{in}}\|_{C^{0,\alpha}}\), while in Step 3 it may also depend on the
lower-order Hölder norms already controlled in the bootstrap.
Our equation can be written in an equivalent way as
\[
\partial_t\mu_t+V[\mu_t]\cdot\nabla\mu_t=\mu_t\Delta\log u_t.
\]

\smallskip
\noindent\textbf{Step 1.}
We first prove that \(\mu\in L^\infty([0,T_1];C^{0,\alpha})\) for some
\(0<T_1\le T\).  Consider the Picard sequence
\[
\mu^0_t=\mu_{\mathrm{in}},\qquad
\mu^{n+1}_t=(X^n_t)_\#\mu_{\mathrm{in}},
\]
where \(X^n\) is the flow generated by \(V[\mu^n]\).  We claim that, after
choosing a constant \(A_0\) depending only on
\( \M,b,\alpha,\|\mu_{\mathrm{in}}\|_{C^{0,\alpha}}\) and decreasing \(T_1\) if necessary,
\begin{equation}\label{eq:holder-picard-claim}
\|\mu^n\|_{L^\infty([0,T_1];C^{0,\alpha})}
\le A_0
\qquad\forall n\ge0 .
\end{equation}
Assume the claim holds for some \(n\).  By
\Cref{lem:wellposed_kernel_estimates},
\[
\|V[\mu^n_t]\|_{C^{1,\alpha}}
\le
C\bigl(1+\|\mu^n_t\|_{C^{0,\alpha}}\bigr)^2
\le L
\qquad\forall t\in[0,T_1]
\]
for \(L=L( \M,b,\alpha,A_0)\).  Hence
\[
\|X^n_t\|_{C^{1,\alpha}}+\|(X^n_t)^{-1}\|_{C^{1,\alpha}}
\le C\exp(CLT_1)
\qquad\forall t\in[0,T_1].
\]
 Let \(JX_t^n\) denote the Riemannian Jacobian of \(X_t^n\).
The same estimate gives
\[
\left\|\frac1{ JX_t^n}\right\|_{C^{0,\alpha}}
\le C\exp(CLT_1).
\]
In particular, exploiting the representation formula
\[
\mu^{n+1}_t
=
\left(\frac{\mu_{\mathrm{in}}}{ JX_t^n}\right)\circ (X_t^n)^{-1},
\]
we conclude
\[
\begin{aligned}
\|\mu^{n+1}_t\|_{C^{0,\alpha}}
&\le
C\left(1+\|(X_t^n)^{-1}\|_{C^{1,\alpha}}^\alpha\right)
\left\|\frac1{ JX_t^n}\right\|_{C^{0,\alpha}}
\|\mu_{\mathrm{in}}\|_{C^{0,\alpha}}  \\
&\le
C\exp(CLT_1)\|\mu_{\mathrm{in}}\|_{C^{0,\alpha}}
\le A_0,
\end{aligned}
\]
provided that \(A_0\) is fixed large enough and then \(T_1>0\) is chosen
sufficiently small.  This proves
\eqref{eq:holder-picard-claim}.  
For each fixed \(t\), the uniform \(C^{0,\alpha}\) bound makes
\(\{\mu_t^n\}_{n\ge0}\) relatively compact in \(C^0(\M)\) by the
Ascoli-Arzel\'a theorem. Moreover, the convergence established in the
Yudovich construction shows that every uniform subsequential limit of these
Picard iterates coincides with \(\mu_t\).
Hence
\(\mu_t\in C^{0,\alpha}\) and
\(\|\mu_t\|_{C^{0,\alpha}}\le A_0\) for \(t\in[0,T_1]\), by lower
semicontinuity of the Hölder seminorm under uniform convergence.

\smallskip
\noindent\textbf{Step 2.}
Let
\[
T_\ast:=\sup\{s\in[T_1,T]:
\mu\in L^\infty([0,s);C^{0,\alpha})\}.
\]
We prove that \(T_\ast=T\).  All estimates below are first made on
\([0,t]\) with \(t<T_\ast\), and the constants are independent of \(t\).  Let
\(X_t\) be the flow generated by \(V[\mu_t]\), and set
\[
\tilde\mu_t:=\mu_t\circ X_t,\qquad
Z(t):=\int_0^t\|\nabla V[\mu_s]\|_{L^\infty}\,ds .
\]
The flow bounds follow from Gronwall's lemma applied to
\(\frac{d}{dt}\operatorname{dist}_{ \M}(X_t(x),X_t(y))\le
\|\nabla V[\mu_t]\|_{L^\infty}\operatorname{dist}_{ \M}(X_t(x),X_t(y))\), and similarly for the
backward flow:
\[
\operatorname{Lip}(X_t)+\operatorname{Lip}(X_t^{-1})\le C e^{Z(t)},
\]
where the constant only accounts for working in finitely many charts on
 \(\M\). Since \(\mu_t=\tilde\mu_t\circ X_t^{-1}\), this gives
\[
\begin{aligned}
\|\mu_t\|_{C^{0,\alpha}}
&\le \|\tilde\mu_t\|_{L^\infty}
+[\tilde\mu_t]_{C^{0,\alpha}}\operatorname{Lip}(X_t^{-1})^\alpha  \\
&\le C e^{\alpha Z(t)}\|\tilde\mu_t\|_{C^{0,\alpha}} .
\end{aligned}
\]
Along the flow,
\[
\partial_t\tilde\mu_t
=\left(\partial_t\mu_t+V[\mu_t]\cdot\nabla\mu_t\right)\circ X_t
=-\left(\mu_t\divg V[\mu_t]\right)\circ X_t .
\] 
In our
notation \(V[\mu_t]=-\nabla\log u_t\), hence
\(\divg V[\mu_t]=-\Delta\log u_t\), and therefore
\[
\partial_t\tilde\mu_t=\tilde\mu_t\,(\Delta\log u_t)\circ X_t .
\]
Using \eqref{eq:short_chizat_G}, we write this equation as
\begin{equation}
\label{eq:wp_A_B}
\partial_t\tilde\mu_t
=B_t\tilde\mu_t+A_t\tilde\mu_t^2,
\qquad
A_t:=-\frac1{b^2(u_t\circ X_t)},\qquad
B_t:=\frac1{b^2}-|\nabla\log u_t|^2\circ X_t .
\end{equation}
We now only need rough Hölder bounds on the coefficients.  From
\eqref{eq:wellposed_fixed_b_basic}, \(\kappa_b\le u_t\le M_T\) and
\(\|\nabla u_t\|_{L^\infty}\le C\).  Thus \(u_t\) and \(1/u_t\) are bounded in
\(C^{0,\alpha}\).  Moreover, the log-Lipschitz bound for \(\nabla u_t\) in
\Cref{lem:wellposed_kernel_estimates} implies
\([\nabla u_t]_{C^{0,\alpha}}\le C\), since for all $\alpha \in (0,1)$
\(r\log(2+r^{-1})\le Cr^\alpha\) on the compact manifold
 \(\M\).  Hence, we have
\[
\left\|\frac1{u_t}\right\|_{C^{0,\alpha}}
+\left\||\nabla\log u_t|^2\right\|_{C^{0,\alpha}}
\le C .
\]
Composing with \(X_t\) and using \(\operatorname{Lip}(X_t)\le Ce^{Z(t)}\), we get
\[
\|A_t\|_{L^\infty}+\|B_t\|_{L^\infty}\le C,
\qquad
[A_t]_{C^{0,\alpha}}+[B_t]_{C^{0,\alpha}}
\le Ce^{\alpha Z(t)} .
\]
Since \(\|\tilde\mu_t\|_{L^\infty}=\|\mu_t\|_{L^\infty}\le M_T\), the product
estimate in \(C^{0,\alpha}\) gives
\[
\|B_t\tilde\mu_t+A_t\tilde\mu_t^2\|_{C^{0,\alpha}}
\le
C\|\tilde\mu_t\|_{C^{0,\alpha}}+Ce^{\alpha Z(t)} .
\]
Therefore
\[
\frac{d}{dt}\|\tilde\mu_t\|_{C^{0,\alpha}}
\le
C\|\tilde\mu_t\|_{C^{0,\alpha}}+Ce^{\alpha Z(t)} .
\]
Gronwall's lemma gives
\[
\|\tilde\mu_t\|_{C^{0,\alpha}}
\le
C\|\mu_{\mathrm{in}}\|_{C^{0,\alpha}}
+C\int_0^t e^{\alpha Z(r)}\,dr
\le Ce^{\alpha Z(t)},
\]
because \(Z\) is nondecreasing and \(t\le T\). Hence
\[
\|\mu_t\|_{C^{0,\alpha}}
\le Ce^{2\alpha Z(t)} .
\]
On the other hand, by \Cref{lem:wellposed_log_lipschitz},
\[
\|\nabla V[\mu_t]\|_{L^\infty}
\le
C\left(1+\log(2+\|\mu_t\|_{C^{0,\alpha}})\right).
\]
Combining the last two estimates,
\[
Z'(t)=\|\nabla V[\mu_t]\|_{L^\infty}
\le C\bigl(1+Z(t)\bigr).
\]
Another application of Gronwall's lemma gives a finite bound for \(Z\), and
hence for \(\sup_{t\in[0,T_\ast)}\|\mu_t\|_{C^{0,\alpha}}\). If
\(T_\ast<T\), choose \(t_0<T_\ast\) close enough that the short-time
construction of Step 1, restarted from \(\mu_{t_0}\), exists beyond
\(T_\ast\). By uniqueness this construction is \(\mu\), contradicting the
definition of \(T_\ast\), and therefore \(T_\ast=T\).

\smallskip
\noindent\textbf{Step 3.}
We now bootstrap.  Assume, for some \(0\le q<k\), that
\(\mu\in L^\infty([0,T];C^{q,\alpha})\).  By
\Cref{lem:wellposed_kernel_estimates},
\[
u_t=\mathcal K_b\mu_t\in C^{q+2,\alpha},
\qquad
V[\mu_t]=-\nabla\log u_t\in C^{q+1,\alpha},
\]
uniformly in \(t\).  Thus the characteristic flow \(X_t\), and its inverse,
are uniformly bounded in \(C^{q+1,\alpha}\).  For
\(\tilde\mu_t:=\mu_t\circ X_t\), Step 2 gives \eqref{eq:wp_A_B}, i.e.:
\[
\partial_t\tilde\mu_t=B_t\tilde\mu_t+A_t\tilde\mu_t^2,
\qquad
A_t=-\frac1{b^2(u_t\circ X_t)},\qquad
B_t=\frac1{b^2}-|\nabla\log u_t|^2\circ X_t .
\]
Since \(u_t\ge\kappa_b\), the algebra and composition estimates in Hölder
spaces give
\[
\|A_t\|_{C^{q+1,\alpha}}+\|B_t\|_{C^{q+1,\alpha}}\le C,
\qquad
\|\tilde\mu_t\|_{C^{q,\alpha}}\le C .
\]
Consequently a product estimate yields
\[
\|B_t\tilde\mu_t+A_t\tilde\mu_t^2\|_{C^{q+1,\alpha}}
\le C\bigl(1+\|\tilde\mu_t\|_{C^{q+1,\alpha}}\bigr).
\]
The key point is that all terms containing \(q+1\) derivatives of
\(\tilde\mu_t\) are linear in the top norm, while every remaining factor is already
controlled in \(C^{q,\alpha}\).  Hence,
\[
\frac{d}{dt}\|\tilde\mu_t\|_{C^{q+1,\alpha}}
\le C\bigl(1+\|\tilde\mu_t\|_{C^{q+1,\alpha}}\bigr).
\]
Gronwall's lemma gives
\(\tilde\mu\in L^\infty([0,T];C^{q+1,\alpha})\), and composition with
\(X_t^{-1}\) gives
\(\mu\in L^\infty([0,T];C^{q+1,\alpha})\).  Iterating from Step 2 proves the
\(C^{k,\alpha}\) bound.

If \(\mu_{\mathrm{in}}\in C^\infty\), the same argument applies for every \(k\).  The
equation then gives smoothness in time by differentiating
\(\partial_t\mu=-\divg(\mu V[\mu])\) and using the already obtained spatial
regularity, arguing inductively on the number of time derivatives.
\end{proof}

\begin{proposition}[Maximum principle and global well-posedness]\label{prop:pde_global_bounded_wellposedness}
Let \(b>0\) and let \(\mu_{\mathrm{in}}\) be a bounded probability density on
 \(\M\).  Then the maximal solution from
\Cref{prop:pde_local_bounded_wellposedness} is global.  It preserves mass and
satisfies
\[
\|\mu_t\|_{L^\infty( \M)}
\le
\|\mu_{\mathrm{in}}\|_{L^\infty( \M)}
\qquad\forall t\ge0.
\]
\end{proposition}

\begin{proof}
It remains only to rule out the \(L^\infty\) blow-up alternative in
\Cref{prop:pde_local_bounded_wellposedness}.  First assume that \(\mu_{\mathrm{in}}\) is
smooth and strictly positive.  By \Cref{prop:pde_holder_propagation}, the
maximal bounded solution is smooth on every compact subinterval of
\([0,T_{\max})\).  Set
\[
M(t):=\max_{ \M}\mu(t,\cdot).
\]
At a maximum point \(x\), \(\nabla\mu(t,x)=0\), and the equation gives
\[
\partial_t\mu(t,x)=\mu(t,x)\Delta\log u(t,x).
\]
Since \(K_b\) is positive and preserves constants,
\[
u(t,x)=\mathcal K_b\mu_t(x)\le \mathcal K_bM(t)=M(t)=\mu(t,x).
\]
Using \eqref{eq:short_chizat_G}, we get
\[
\Delta\log u(t,x)
=
\frac{u(t,x)-\mu(t,x)}{b^2u(t,x)}
-|\nabla\log u(t,x)|^2
\le0.
\]
Hence \(\frac{d}{dt}M(t)\le0\), and therefore
\[
\|\mu_t\|_\infty\le\|\mu_{\mathrm{in}}\|_\infty
\qquad\forall t<T_{\max}.
\]
The continuation criterion then forces \(T_{\max}=\infty\).

For a general bounded probability density, choose smooth strictly positive
probability densities \(\mu^\varepsilon_{in}\) such that
\[
\mu^\varepsilon_{in}\to\mu_{\mathrm{in}}\quad\text{in }W_2,
\qquad
\|\mu^\varepsilon_{in}\|_\infty\le\|\mu_{\mathrm{in}}\|_\infty .
\]
The stability part of the local theory passes the global smooth solutions to
the \(L^\infty\)-solution starting from \(\mu_{\mathrm{in}}\) on every compact time
interval contained in its maximal interval of existence.  Weak-\(*\) lower
semicontinuity of the \(L^\infty\) norm gives
\[
\|\mu_t\|_\infty\le\|\mu_{\mathrm{in}}\|_\infty
\qquad\forall t<T_{\max},
\]
and the continuation criterion again excludes finite-time breakdown,
proving global existence. Uniqueness and mass preservation were already part of
the local maximal theory.
\end{proof}

An immediate corollary is obtained by applying \Cref{prop:pde_holder_propagation} after
\Cref{prop:pde_global_bounded_wellposedness}, which gives \(T_{\max}=\infty\).
\begin{corollary}[Global propagation of Hölder regularity]\label{cor:pde_global_holder_propagation}
Let \(0<\alpha<1\), \(k\in\mathbb N\cup\{0\}\), and let
\(\mu_{\mathrm{in}}\in C^{k,\alpha}( \M)\) be a bounded probability
density.  If \(\mu\) denotes the global bounded solution from
\Cref{prop:pde_global_bounded_wellposedness}, then
\[
\mu\in L^\infty_{\rm loc}([0,\infty);C^{k,\alpha}( \M)).
\]
If \(\mu_{\mathrm{in}}\in C^\infty\), then
\(\mu\in C^\infty([0,\infty)\times \M)\).
\end{corollary}

This result provides a proof of \Cref{thm:pde_global_smooth_existence}:
\begin{proof}[Proof of \Cref{thm:pde_global_smooth_existence}]
Apply \Cref{prop:pde_global_bounded_wellposedness}.  This gives a unique global bounded solution, mass
conservation, and the \(L^\infty\) bound.  Since \(\mu_{\mathrm{in}}\in C^\infty\),
\Cref{cor:pde_global_holder_propagation} gives
\[
\mu_b\in C^\infty([0,\infty)\times \M).
\]
Strict positivity follows from the smooth characteristic formula
\[
\mu_b(t,X_t(a))
=
\mu_{\mathrm{in}}(a)
\exp\left(\int_0^t\Delta\log u_b(r,X_r(a))\,dr\right),
\]
because \(\mu_{\mathrm{in}}>0\) and the exponent is finite on every compact time
interval.  Finally, \(u_b=\mathcal K_b\mu_b\) is smooth and strictly positive by
\Cref{lem:wellposed_kernel_estimates}.
\end{proof}

\section{Heat-Flow Preliminaries}\label{app:heat_preliminaries}
In this section we collect a couple of classical estimates on the heat equation on
 $\M$ that are used in the proofs of the main results.

\begin{lemma}[Heat-flow lower bound and decay]\label{lem:heat_flow_bounds}
Let $\mu$ solve $\partial_t \mu = \Delta \mu$ on  $\M$ with initial datum $\mu_{\mathrm{in}}$ from \Cref{assump:initial}. Set
\[
c_0 := \inf_{x\in  \M}\mu_{\mathrm{in}}(x) > 0.
\]
Then
\[
\mu(t,x)\ge c_0
\qquad
\text{for all }(t,x)\in [0,\infty)\times  \M.
\]
Moreover, for every integer $s\ge0$,
\[
\|\mu(t)-1\|_{H^s( \M)}
\le
e^{- \lambda_1t}\|\mu_{\mathrm{in}}-1\|_{H^s( \M)}.
\]
In particular, there exists $C=C(\mu_{\mathrm{in}}, \M)$ such that
\[
\|\nabla \mu(t)\|_{L^\infty( \M)} + \|\Delta \mu(t)\|_{L^\infty( \M)}
\le
Ce^{- \lambda_1t}
\]
and
\[
\|\nabla \log \mu(t)\|_{L^\infty( \M)}
+
\|\Delta \log \mu(t)\|_{L^\infty( \M)}
\le
Ce^{- \lambda_1t}
\]
for all $t\ge0$.
\end{lemma}

\begin{proof}
The function $\mu-c_0$ solves the heat equation with nonnegative initial datum $\mu_{\mathrm{in}}-c_0$, so the parabolic maximum principle gives $\mu(t,x)\ge c_0$ for all $(t,x)\in[0,\infty)\times \M$. The same argument yields the upper bound $\mu(t,x)\le M_0:=\|\mu_{\mathrm{in}}\|_{L^\infty( \M)}$.

Set $h:=\mu-1$. Then $\int_{ \M} h(t)\,dx=0$ and $h(t)=e^{t\Delta}h(0)$. On the mean-zero subspace of $L^2( \M)$, the operator $-\Delta$ has spectrum contained in $[ \lambda_1,\infty)$; hence the spectral theorem gives, for every integer $s\ge0$,
\[
\|h(t)\|_{H^s( \M)}
\le
e^{- \lambda_1t}\|h(0)\|_{H^s( \M)}.
\]
This is exactly the stated $H^s$ decay.

Choose an integer $s> d/2+2$. By the compact-manifold Sobolev embedding in
\cite[Theorem~6.3]{hebey_robert_2008}, applied with $p=2$, differentiability
order $s$, and target order $2$, we obtain
\[
\|\mu(t)-1\|_{W^{2,\infty}( \M)}
\le
C\|\mu(t)-1\|_{H^s( \M)}
\le
Ce^{- \lambda_1t}.
\]
Since $\log$ is smooth on the compact interval $[c_0,M_0]$, the chain rule implies
\[
\|\nabla \log \mu(t)\|_{L^\infty}
\le
C_{c_0,M_0}\|\nabla \mu(t)\|_{L^\infty},
\]
and
\[
\|\nabla^2 \log \mu(t)\|_{L^\infty}
\le
C_{c_0,M_0}\bigl(\|\nabla^2 \mu(t)\|_{L^\infty}+\|\nabla \mu(t)\|_{L^\infty}^2\bigr).
\]
Since $\|\nabla \mu(t)\|_{L^\infty}+\|\nabla^2 \mu(t)\|_{L^\infty}\le Ce^{- \lambda_1t}$, as proved above, we obtain
\[
\|\nabla \log \mu(t)\|_{L^\infty( \M)}
+
\|\Delta \log \mu(t)\|_{L^\infty( \M)}
\le
Ce^{- \lambda_1t},
\]
where we also used $|\Delta f|\le  d\|\nabla^2 f\|_{L^\infty}$ on  $\M$.
\end{proof}

\begin{lemma}[Heat-semigroup smoothing]\label{lem:heat_semigroup_smoothing}
Let $P_t:=e^{t\Delta}$ be the heat semigroup on  $\M$. Then there
exists a constant $C=C( \M)>0$ such that every bounded measurable $\phi$
satisfies
\[
\|\nabla P_t\phi\|_{L^\infty( \M)}
\le
C t^{-1/2}\|\phi\|_{L^\infty( \M)}
\qquad\text{for all }t>0.
\]
\end{lemma}

\begin{proof}
Since  $\M$ has bounded Ricci curvature, the estimate in
\cite[(1.3)]{wang2004gradient} gives
\[
\|\nabla P_t f\|_{L^\infty( \M)}
\le
C_{ \M}t^{-1/2}\|f\|_{L^\infty( \M)}
\]
for nonnegative bounded measurable $f$. If
$M:=\|\phi\|_{L^\infty}$, then $f:=\phi+M$ is nonnegative,
$\|f\|_{L^\infty}\le 2M$, and $P_t1=1$, hence
$\nabla P_t f=\nabla P_t\phi$. Therefore
\[
\|\nabla P_t\phi\|_{L^\infty}
\le
2C_{ \M}t^{-1/2}\|\phi\|_{L^\infty},
\]
which proves the claim.
\end{proof}

\section{Entropy Sketch for Uniform Convergence}\label{app:entropy_convergence}

Throughout this appendix section we assume, in addition, that \(
\Ric\ge0\) on $\M$.
The goal is to show how one may obtain a uniform-in-time convergence
estimate by following the relative entropy between $\mu_b$ and the heat flow
$\mu$.  We use the heat-flow estimates already collected in
\Cref{lem:heat_flow_bounds} and do not repeat them here.

\begin{lemma}[Bochner entropy identity]\label{lem:entropy_bochner}
Let $u\in C^\infty(\M)$ be positive. Then
\begin{align*}
&\int_\M \Delta u
\left(\Delta\log u+\frac12|\nabla\log u|^2\right)\,dx
\\
&\qquad =
\int_\M u\left(
|\Hess(\log u)|^2+\Ric(\nabla\log u,\nabla\log u)
\right)\,dx .
\end{align*}
In particular, under $\Ric\ge0$, the right-hand side is nonnegative.
\end{lemma}

\begin{proof}
Write $f:=\log u$, so that $\nabla u=u\nabla f$ and
$\Delta u=\divg(u\nabla f)$.  By integration by parts and the Bochner formula,
\[
\int_\M \Delta u\,\Delta f\,dx
=
-\int_\M u\langle\nabla f,\nabla\Delta f\rangle\,dx
\]
and
\[
\frac12\Delta|\nabla f|^2
=
|\Hess f|^2+\langle\nabla f,\nabla\Delta f\rangle
+\Ric(\nabla f,\nabla f).
\]
Substituting the second identity into the first gives
\[
\int_\M \Delta u\,\Delta f\,dx
=
\int_\M u\bigl(|\Hess f|^2+\Ric(\nabla f,\nabla f)\bigr)\,dx
-\frac12\int_\M u\,\Delta|\nabla f|^2\,dx .
\]
Since $\M$ is closed,
\[
\int_\M u\,\Delta|\nabla f|^2\,dx
=
\int_\M \Delta u\,|\nabla f|^2\,dx,
\]
which proves the claimed identity.
\end{proof}

\begin{lemma}[Entropy dissipation]\label{lem:entropy_fisher_bound}
Let $\mu_b$ solve
\[
\partial_t\mu_b=\divg(\mu_b\nabla\log u_b),
\qquad
\mu_b=u_b-b^2\Delta u_b,
\]
on $\M$.  If
\[
H(\mu_b):=\int_\M \mu_b\log\mu_b\,dx,
\qquad
\mathcal I(u_b):=\int_\M u_b|\nabla\log u_b|^2\,dx,
\]
then
\[
\frac{d}{dt}H(\mu_b(t))
\le
-\frac12\mathcal I(u_b(t)).
\]
\end{lemma}

\begin{proof}
Differentiating the entropy and using the continuity equation gives
\[
\frac{d}{dt}H(\mu_b)
=
\int_\M \log\mu_b\,\divg(\mu_b\nabla\log u_b)\,dx
=
-\int_\M \langle\nabla\mu_b,\nabla\log u_b\rangle\,dx .
\]
Since $\mu_b=u_b-b^2\Delta u_b$,
\[
\frac{d}{dt}H(\mu_b)
=
-\mathcal I(u_b)
+b^2\int_\M \langle\nabla\Delta u_b,\nabla\log u_b\rangle\,dx .
\]
Integrating the second term by parts and applying
\Cref{lem:entropy_bochner} with $u=u_b$ yields
\[
b^2\int_\M \langle\nabla\Delta u_b,\nabla\log u_b\rangle\,dx
=
-b^2\int_\M \Delta u_b\,\Delta\log u_b\,dx
\le
\frac{b^2}{2}\int_\M \Delta u_b|\nabla\log u_b|^2\,dx .
\]
Using $b^2\Delta u_b=u_b-\mu_b$, we obtain
\[
\frac{d}{dt}H(\mu_b)
\le
-\mathcal I(u_b)
+\frac12\int_\M (u_b-\mu_b)|\nabla\log u_b|^2\,dx
=
-\frac12\mathcal I(u_b)
-\frac12\int_\M \mu_b|\nabla\log u_b|^2\,dx ,
\]
and the last term is nonpositive.
\end{proof}

\begin{proposition}[Relative entropy comparison]\label{prop:relative_entropy_convergence}
Let $\mu$ be the heat flow with initial datum $\mu_{\mathrm{in}}$, and let $\mu_b$ solve
the regularized equation with the same initial datum.  Then
\begin{align*}
\frac{d}{dt}\mathcal H(\mu_b(t)\mid\mu(t))
&\le
-\int_\M
\mu_b|\nabla\log u_b-\nabla\log\mu|^2\,dx
\\
&\quad
+Cb^2\left(
\|\nabla\log\mu\|_{L^\infty(\M)}^2\mathcal I(u_b)
+\|\Delta\log\mu\|_{L^\infty(\M)}^2
\right),
\end{align*}
where
\[
\mathcal H(\mu_b\mid\mu)
:=
\int_\M \mu_b\log\frac{\mu_b}{\mu}\,dx .
\]
Consequently,
\[
\sup_{t\ge0}\mathcal H(\mu_b(t)\mid\mu(t))
\le
C(\M,\mu_{\mathrm{in}})b^2 .
\]
\end{proposition}

\begin{proof}
Using
\[
\partial_t\mu_b=\divg(\mu_b\nabla\log u_b),
\qquad
\partial_t\mu=\Delta\mu=\divg(\mu\nabla\log\mu),
\]
and mass conservation, differentiation gives
\[
\frac{d}{dt}\mathcal H(\mu_b\mid\mu)
=
\int_\M \partial_t\mu_b(\log\mu_b-\log\mu)\,dx
-\int_\M \frac{\mu_b}{\mu}\partial_t\mu\,dx .
\]
After integration by parts,
\[
\frac{d}{dt}\mathcal H(\mu_b\mid\mu)
=
-\int_\M
\mu_b\langle\nabla\log\mu_b-\nabla\log\mu,
\nabla\log u_b-\nabla\log\mu\rangle\,dx .
\]
Inserting and subtracting $\nabla\log u_b$ in the first factor yields
\begin{align*}
\frac{d}{dt}\mathcal H(\mu_b\mid\mu)
&=
-\int_\M \mu_b|\nabla\log u_b-\nabla\log\mu|^2\,dx
\\
&\quad
-\int_\M
\mu_b\langle\nabla\log\mu_b-\nabla\log u_b,
\nabla\log u_b-\nabla\log\mu\rangle\,dx .
\end{align*}
Let $R$ denote the second line.  From
$\mu_b=u_b-b^2\Delta u_b$,
\[
\mu_b(\nabla\log\mu_b-\nabla\log u_b)
=
-b^2\left(\nabla\Delta u_b-(\Delta u_b)\nabla\log u_b\right).
\]
Therefore
\[
R
=
b^2I_1-b^2I_2,
\]
where
\[
I_1:=
\int_\M
\left\langle\nabla\log u_b,
\nabla\Delta u_b-(\Delta u_b)\nabla\log u_b\right\rangle\,dx
\]
and
\[
I_2:=
\int_\M
\left\langle\nabla\log\mu,
\nabla\Delta u_b-(\Delta u_b)\nabla\log u_b\right\rangle\,dx .
\]
For $I_1$, integration by parts gives the cancellation
\[
I_1
=
-\int_\M \frac{(\Delta u_b)^2}{u_b}\,dx .
\]
For $I_2$,
\[
I_2
=
-\int_\M \Delta\log\mu\,\Delta u_b\,dx
-\int_\M \Delta u_b
\langle\nabla\log\mu,\nabla\log u_b\rangle\,dx .
\]
Hence, by Young's inequality with weight $u_b$,
\begin{align*}
R
&\le
-b^2\int_\M \frac{(\Delta u_b)^2}{u_b}\,dx
+b^2\int_\M \frac{|\Delta u_b|}{\sqrt{u_b}}\sqrt{u_b}
\left(
|\Delta\log\mu|
+|\nabla\log\mu|\,|\nabla\log u_b|
\right)\,dx
\\
&\le
-\frac{b^2}{2}\int_\M \frac{(\Delta u_b)^2}{u_b}\,dx
+Cb^2\int_\M u_b
\left(
|\Delta\log\mu|^2
+|\nabla\log\mu|^2|\nabla\log u_b|^2
\right)\,dx
\\
&\le
Cb^2\left(
\|\Delta\log\mu\|_{L^\infty(\M)}^2
+\|\nabla\log\mu\|_{L^\infty(\M)}^2\mathcal I(u_b)
\right),
\end{align*}
where we used $\int_\M u_b\,dx=1$.  This proves the differential inequality.

It remains to integrate in time.  Since the initial data agree,
$\mathcal H(\mu_b(0)\mid\mu(0))=0$.  By
\Cref{lem:entropy_fisher_bound},
\[
\int_0^\infty \mathcal I(u_b(t))\,dt
\le
2H(\mu_{\mathrm{in}}).
\]
Moreover, \Cref{lem:heat_flow_bounds} gives exponential decay of the two
logarithmic heat-flow norms appearing above.  Integrating the displayed
differential inequality on $[0,t]$ therefore yields
\[
\mathcal H(\mu_b(t)\mid\mu(t))
\le
C(\M,\mu_{\mathrm{in}})b^2
\qquad\text{for all }t\ge0 .
\]
\end{proof}

\begin{corollary}[Uniform \(L^1\) consequence]\label{cor:entropy_l1_convergence}
Under the assumptions of \Cref{prop:relative_entropy_convergence},
\[
\sup_{t\ge0}\|\mu_b(t)-\mu(t)\|_{L^1(\M)}
\le
C(\M,\mu_{\mathrm{in}})b .
\]
\end{corollary}

\begin{proof}
By the Csisz\'ar-Kullback-Pinsker inequality,
\[
\|\mu_b(t)-\mu(t)\|_{L^1(\M)}^2
\le
2\mathcal H(\mu_b(t)\mid\mu(t)).
\]
The result follows from \Cref{prop:relative_entropy_convergence}.
\end{proof}
\begin{remark}
The relative entropy argument extends, by approximation with smooth positive densities, to give \emph{qualitative} \(L^1\)-convergence for initial data \(\mu_{\mathrm{in}}\in L^\infty(\M)\). We do not pursue this extension here.
\end{remark}
\section{Numerical comparison of kernels}\label{app:numerics}

We compare the particle system \eqref{eq:particle_system_main} on $\Sphere^1$ with $N=256$ for the Bessel--Yukawa kernel \eqref{eq:explicit_circle_resolvent_kernel} and the Gaussian (von Mises) kernel $\widetilde K_\sigma(\theta,\theta')=e^{\sigma\cos(\theta'-\theta)}/I_0(\sigma)$. Both systems start from the same particle configuration, approximating the initial density in \Cref{fig:kernel_dynamics}, and evolve up to $t=10$. Their scales satisfy $I_1(\sigma)/I_0(\sigma)=1/(1+b^2)$, so their first Fourier coefficients agree (here $I_0$ and $I_1$ are modified Bessel functions of the first kind). We report the histogram $L^1$ error and the circular $W_1$ distance to the exact heat solution. Plotted densities are with respect to $\d\theta$.

For these parameters, the Laplace dynamics give smaller errors and approach an evenly spaced configuration, with $W_1$ close to $\pi/(2N)$, the optimal error for approximating the uniform measure by $N$ equally weighted particles. The Gaussian dynamics retain visible density oscillations and a larger error over the simulated time interval.

\begin{figure}[htbp]
\centering
\includegraphics[width=\linewidth]{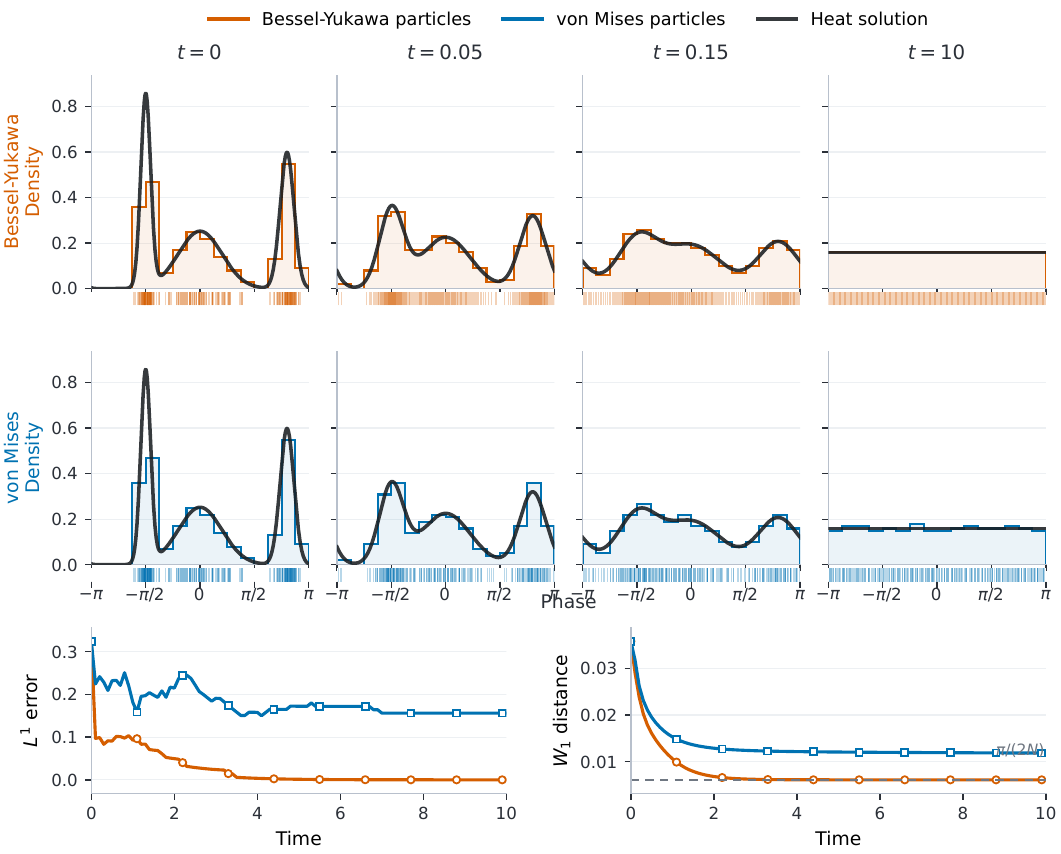}
\caption{Particle approximation of heat flow on $\Sphere^1$ with $N=256$. The top and middle rows show particle histograms for the Bessel--Yukawa and Gaussian kernels, respectively, together with the exact heat solution. Rug plots indicate particle locations. The bottom row shows the kernel profiles, histogram $L^1$ error, and circular $W_1$ distance. The dashed line marks $\pi/(2N)$.}
\label{fig:kernel_dynamics}
\end{figure}

\end{document}